\documentclass[a4paper]{article}
\usepackage[utf8]{inputenc}
\usepackage{fullpage}

\usepackage{hyperref}
\usepackage[auth-lg]{authblk}
\usepackage[english]{babel}
\usepackage{amsmath, amsfonts, amssymb, amsthm}
\usepackage{graphicx}
\usepackage{mathtools}
\usepackage{subfigure}
\usepackage{multirow}
\usepackage{array}
\renewcommand{\arraystretch}{1.3}
\newcolumntype{x}[1]{>{\centering\arraybackslash\hspace{0pt}}p{#1}}
\usepackage{color}
\usepackage{comment}
\usepackage{quiver}
\usepackage{svg}
\usepackage{booktabs}
\usepackage{tabularx}
\usepackage[page]{appendix}
\usepackage[shortlabels]{enumitem}

\usepackage{algorithm,algpseudocode}

\definecolor{fom}{HTML}{003391}
\definecolor{ldm}{HTML}{21a300}
\definecolor{dldm}{HTML}{ff4f9d}
\definecolor{dldmtheta}{HTML}{ffba3b}

\newtheorem{theorem}{Theorem}[section]

\newtheorem{proposition}[theorem]{Proposition}
\theoremstyle{definition}
\newtheorem{definition}[theorem]{Definition}

\theoremstyle{remark}
\newtheorem{remark}[theorem]{Remark}

\numberwithin{equation}{section}

\newcommand{\bgamma}{{\boldsymbol{\gamma}}}
\newcommand{\bbeta}{{\boldsymbol{\beta}}}
\newcommand{\beps}{{\boldsymbol{\epsilon}}}
\newcommand{\bveps}{{\boldsymbol{\varepsilon}}}
\newcommand{\btau}{{\boldsymbol{\tau}}}
\newcommand{\bmu}{{\boldsymbol{\mu}}}

\newcommand{\bdelta}{\boldsymbol{\delta}}
\newcommand{\btheta}{{\boldsymbol{\theta}}}

\newcommand{\bxi}{{\boldsymbol{\xi}}}
\newcommand{\bu}{{\mathbf{u}}}
\newcommand{\tbu}{{\tilde{\mathbf{u}}}}
\newcommand{\hbu}{{\hat{\mathbf{u}}}}

\newcommand{\bx}{{\mathbf{x}}}
\newcommand{\by}{{\mathbf{y}}}
\newcommand{\bz}{{\mathbf{z}}}
\newcommand{\tbz}{{\tilde{\mathbf{z}}}}

\newcommand{\bn}{{\mathbf{n}}}
\newcommand{\bb}{{\mathbf{b}}}

\newcommand{\bh}{{\mathbf{h}}}

\newcommand{\bff}{{\mathbf{f}}}
\newcommand{\bS}{{\mathbf{S}}}
\newcommand{\bM}{{\mathbf{M}}}

\newcommand{\Dt}{{\Delta t}}

\newcommand{\dldm}{{\Delta\text{LDM}}}

\makeatletter
\let\OldStatex\Statex
\renewcommand{\Statex}[1][3]{%
  \setlength\@tempdima{\algorithmicindent}%
  \OldStatex\hskip\dimexpr#1\@tempdima\relax}
\makeatother

\begin{document}
\setlength\parindent{0pt}
\title{On latent dynamics learning in nonlinear reduced order modeling}

\author[1]{Nicola Farenga} 
\author[1]{Stefania Fresca}
\author[1]{Simone Brivio}
\author[1]{Andrea Manzoni}
\affil[1]{\normalsize MOX, Department of Mathematics, Politecnico di Milano, Milan, Italy}
\date{}
\maketitle
\vspace{-1.2cm}
\begin{center}
{\small \{\texttt{nicola.farenga}, \texttt{stefania.fresca}, 
\texttt{simone.brivio},
\texttt{andrea1.manzoni}\} \texttt{@polimi.it}}
\end{center}

\begin{abstract}
In this work, we present the novel mathematical framework of \textit{latent dynamics models} (LDMs) for reduced order modeling of parameterized nonlinear time-dependent PDEs. Our framework casts this latter task as a nonlinear dimensionality reduction problem, while constraining the latent state to evolve accordingly to an (unknown) dynamical system, namely a latent vector ordinary differential equation (ODE). A time-continuous setting is employed to derive error and stability estimates for the LDM approximation of the full order model (FOM) solution. We analyze the impact of using an explicit Runge-Kutta scheme in the time-discrete setting, resulting in the $\dldm$ formulation, and further explore the learnable setting, $\dldm_\theta$, where deep neural networks approximate the discrete LDM components, while providing a bounded approximation error with respect to the FOM. 
Moreover, we extend the concept of parameterized Neural ODE -- recently proposed as a possible way to build data-driven dynamical systems with varying input parameters -- to be a convolutional architecture, where the input parameters information is injected by means of an affine modulation mechanism, while designing a convolutional autoencoder neural network able to retain spatial-coherence, thus enhancing interpretability at the latent level. Numerical experiments, including the Burgers' and the advection-reaction-diffusion equations, demonstrate the framework’s ability to obtain, in a multi-query context, a \textit{time-continuous} approximation of the FOM solution, thus being able to query the LDM approximation at any given time instance while retaining a prescribed level of accuracy. Our findings highlight the remarkable potential of the proposed LDMs, representing a mathematically rigorous framework to enhance the accuracy and approximation capabilities of reduced order modeling for time-dependent parameterized PDEs.
\end{abstract}

\section{Introduction}
The numerical solution of parameterized nonlinear time-dependent partial differential equations (PDEs) by means of traditional high-fidelity techniques \cite{Thomee2006, Quarteroni2008}, e.g. the finite element method, entails large computational costs, primarily stemming from the large number of degrees of freedom (DoFs) involved in their spatial discretization, and further compounded by the need to select a temporal discretization that ensures the required level of accuracy \cite{CFL}. As a result, the computationally intensive nature of such techniques prevents them to be employed in \textit{real-time} and \textit{multi-query} scenarios, where the PDE solution has to be computed for multiple parameters instances, within reasonable time-frames.

Reduced order models (ROMs) address such issue by introducing a low-dimensional representation approximating the high-fidelity solution of the full order model (FOM), thereby mitigating the computational burden arising from the solution of the original high-dimensional problem. The core assumption underlying such techniques is that FOM solutions lie on a low-dimensional  manifold, i.e., \textit{the solution manifold}. In particular, two main tasks have to be performed in order to construct a ROM: { \em (i)} the computation of a suitable \textit{trial manifold} approximating the solution manifold, namely the set of PDE's solutions by varying time and parameters, and {\em (ii) } the identification of a \textit{latent dynamics} embedding the reduced representation onto the trial manifold. Traditional reduced basis (RB) methods \cite{quarteroni2015reduced, BennerGugercinWillcox, cicci2023projection} approximate the {solution manifold} by means of a \textit{linear trial subspace}, spanned by a set of basis functions. In the parameterized setting, RB methods face significant challenges with \textit{nonaffine} parameter dependencies, namely, when the operators arising from the problem's spatial discretization cannot be expressed as a linear combination of parameter-independent matrices weighted by parameter-dependent coefficients \cite{NEGRI2015431}. This nonaffine structure is present in many scenarios, for instance when the parameter dependence is nonlinear, in problems featuring geometrical deformations, as well as in multi-physics systems. In such settings, the components involved in the decomposition of the FOM operators needs to be re-assembled for each specific time and parameters instances, thus breaking the efficiency of the \textit{online-offline} procedure.
Moreover, RB methods rely on linear projections which limit their expressivity for problems exhibiting slowly-decaying \textit{Kolmogorov n-width} \cite{kolmogorov, greif2019decay, peherstorfer2022breaking, Nonino_2023}. This implies that achieving substantial dimensionality reduction becomes impossible, as the intrinsic low-dimensional structure cannot be accurately captured via small linear subspaces.
Furthermore, the intrusive nature of such techniques requires access to the high-fidelity operators for the construction of the \textit{reduced order dynamics} governing the evolution of the reduced state. This aspect represents a restrictive requirement in those cases in which neither the algebraic-structure of the original problem nor the operators are accessible, or when black-box FOM solvers are employed. 

The previous limitations have motivated the recent adoption of data-driven methods \cite{Hesthaven_Pagliantini_Rozza_2022}, addressing the common issues related to traditional linear ROMs via machine learning-based techniques. 
In the specific context of reduced order modeling for nonlinear parameterized time-dependent problems, multiple deep neural network (DNN)-based approaches have been proposed to address two primary tasks: {\em (i)} nonlinear dimensionality reduction, by means of convolutional autoencoders (AEs) \cite{gonzalez2018deep,san2019artificial,lee2020model} or geometry-informed techniques \cite{franco2023geom, franco2023meshinformed}; and {\em (ii)} modeling the latent dynamics of the resulting time-evolving parameterized reduced representation, via regressive approaches \cite{fresca2021comprehensive, mucke2021, pant2021, fresca2022pod}, recurrent strategies \cite{kani2017drrnndeepresidualrecurrent, Conti_2023}, or by coupling these two modeling paradigms \cite{10.3934/mine.2023096}.
In particular, recent DL-based \textit{latent dynamics learning} approaches have predominantly relied on recurrent neural networks (RNNs) architectures \cite{doi:10.1098/rspa.2017.0844, vlachas2022multiscale} and autoregressive schemes \cite{wu2022learning, li2024latent}, aiming at modeling the temporal dependency of the latent state in a sequential manner.

However, both traditional and DL-based ROMs are intrinsically tied to the temporal discretization employed during the offline training phase. This aspect represents a critical constraint, restricting ROMs' ability to adapt to different temporal resolutions during the online phase. As a consequence, it is often necessary to rebuild the reduced basis or undergo additional fine-tuning stages, thereby entailing increased offline computational costs. 
To this end, a central aspect in reduced order modeling resides in whether the ROM solution is a meaningful \textit{time-continuous approximation} of the FOM, in which case it could be possible to query the learned ROM at any given time by retaining a prescribed level of accuracy.

A recently proposed DL-based reduced order modeling approach consists in leveraging {neural ordinary differential equations} (NODEs) \cite{NEURIPS2018_69386f6b} coupled with AEs, resulting in the AE-NODE architecture, in order to implicitly learn the vector field parameterizing the latent dynamics \cite{Lee_2021, Di_Sante_2022}. 
In particular, as highlighted by \cite{krishnapriyan2023learning}, the \textit{continuous-time}\footnote{For a detailed description of continuous-depth and continuous-time models, we refer to Appendix \ref{appendix:continuous-time}.} nature of NODEs, employed in time-series modeling setting for the approximation of dynamical systems' evolution, have demonstrated promising capabilities at the task of zero-shot generalization across temporal discretizations, when adopting higher-order Runge-Kutta (RK) schemes (ODE-Nets). However, as already pointed out by \cite{queiruga2020continuousindepth}, models characterized by a continuous-time inductive bias, such as ODE-Nets, do not necessarily guarantee that the learned representation will exhibit a time-continuous approximation property. This is usually due to the fact that learning settings are inherently time-discrete, thus possibly leading to overfitting with respect to the training temporal discretization.
Despite the centrality of this aspect, the time-continuous approximation properties of the AE-NODE paradigm have not yet been addressed in a multi-query reduced order modeling context.

\subsection{Contributions}
\begin{figure}[t]
    \centering
    \includegraphics[width=.875\textwidth]{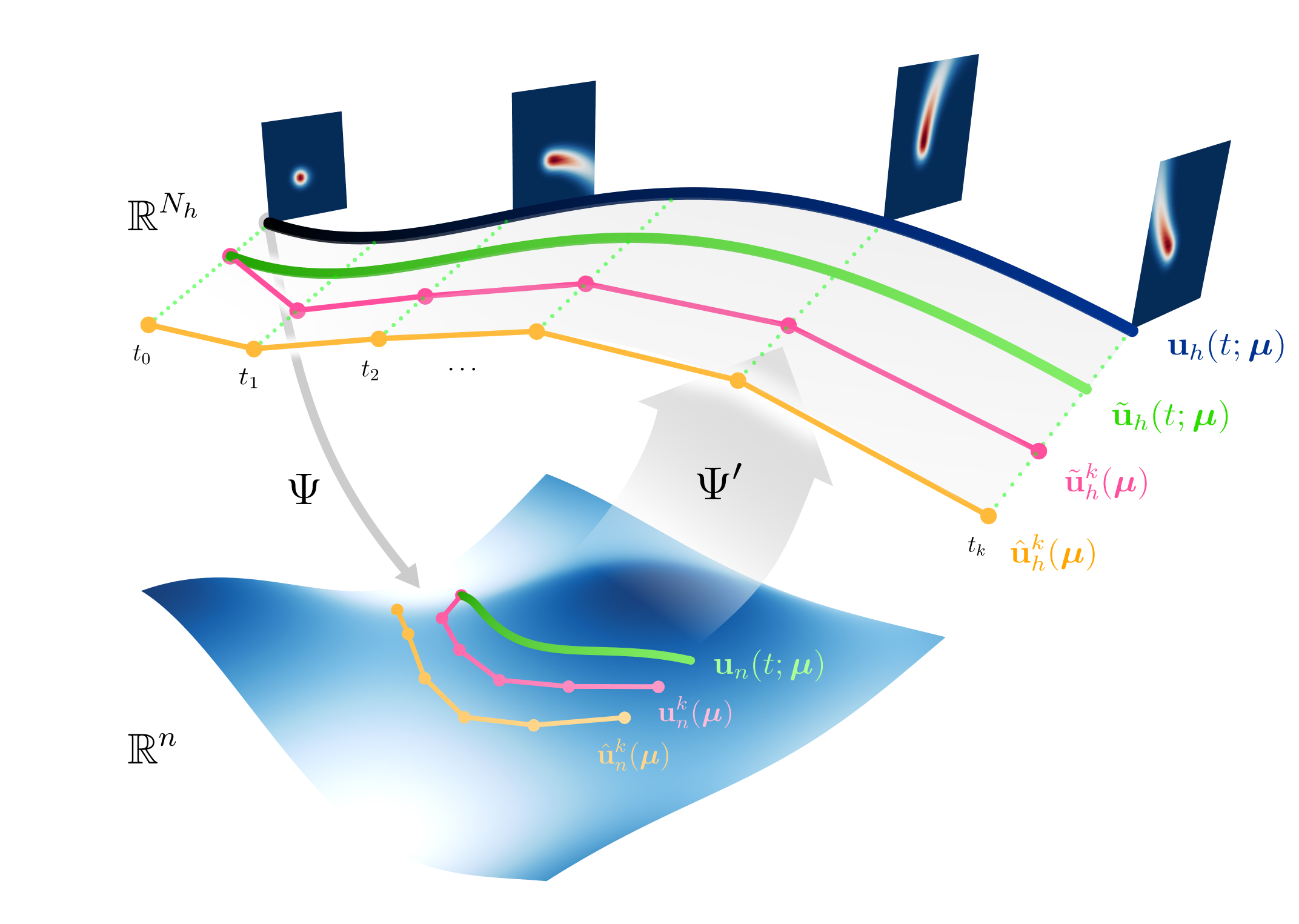}
    \caption{Different levels of approximation in the proposed LDM framework, illustrating: {\em (i)} the time-continuous full-order model \textcolor{fom}{$\bu_h(t;\bmu)$}, the time-continuous LDM approximation \textcolor{ldm}{$\tilde{\bu}_h(t;\bmu)$} with the associated latent dynamics \textcolor{ldm}{$\bu_n(t;\bmu)$}; {\em (ii)}  the time-discrete approximation \textcolor{dldm}{$\tilde{\bu}_h^k(\bmu)$} arising from the numerical solution of the latent dynamics \textcolor{dldm}{$\bu_n^k(\bmu)$}, leading to the notion of $\dldm$; {\em (iii)}  the learned approximation \textcolor{dldmtheta}{$\hat{\bu}_h^k(\bmu)$} and the corresponding learned latent dynamics \textcolor{dldmtheta}{$\hat{\bu}_n^k(\bmu)$}, associated to the learnable $\dldm_\theta$. Here, $\Psi$ and $\Psi'$ denote the nonlinear projection and reconstruction maps, respectively.}
\end{figure}

In this work, we aim at:
\begin{itemize}
    \item introducing the \textit{latent dynamics model} (LDM) mathematical framework and rigorously demonstrating that it is able to provide an accurate and time-continuous approximation of the FOM solution, arising from the semi-discretization of parameterized nonlinear time-dependent PDEs;
    \item characterizing the resulting reduced order modeling framework based on the AE-NODE architecture \cite{Lee_2021, Di_Sante_2022, legaard2022constructing, Chen_2022}, which combines {\em (a)} autoencoders (AEs) to reduce the dimensionality of the state to a handful of latent coordinates, and {\em (b)} ODE-Nets to describe their dynamics.
\end{itemize} 

Specifically, the problem of nonlinear dimensionality reduction for parameterized dynamical systems is first addressed, highlighting the role of the nonlinear projection and reconstruction maps 
and motivating the introduction of a unified framework addressing both the problem of dimensionality reduction and latent dynamics modeling. Thus, the novel concept of \textit{latent dynamics model} (LDM) is introduced, as a mathematical framework formalizing the problem of computing an accurate and time-continuous approximation of the FOM solution which, at the best of our knowledge, is currently missing in the scientific literature.

The LDM framework is introduced in a time-continuous setting in order to consecutively analyze {\em (i)} the impact of employing a numerical integration scheme for the solution of the latent dynamics, leading to a time-discrete LDM formulation, namely $\dldm$; {\em (ii)} the approximation capabilities in a learnable setting, where DNNs are adopted to approximate the components involved in the $\dldm$, via a $\dldm_\theta$ of learnable weights $\theta$.

More precisely, at the time-continuous level, suitable error and stability estimates for the LDM approximation of the FOM solution are derived. In the time-discrete setting, the role of an explicit Runge-Kutta (RK) scheme in approximating the solution of the latent dynamics within the $\dldm$ scheme is analyzed. In particular, convergence, consistency, and zero-stability results are provided, highlighting the inheritance of these properties by the high-dimensional $\dldm$ approximation under a Lipschitz requirement on the nonlinear projection and reconstruction maps. Additionally, an error decomposition formula that accounts for the different error sources is derived. 
Finally, in the $\dldm_\theta$ learnable setting, the issue of approximating the discrete LDM components using DNNs within a parameterized AE-NODE architecture is addressed. Specifically, the approximation capabilities of the $\dldm_\theta$ with respect to the FOM solution are demonstrated through the derivation of a suitable error bound.

From a DL perspective, the novel architectural choices, underlying the proposed $\dldm_\theta$ formulation, are described, departing from the traditional AE-NODE setting. In particular, a fully-convolutional AE architecture, based on a prior interpolation operation, is considered. It aims at avoiding input reshape operations and exploiting spatial correlations, with the possibility to retain \textit{spatially-coherent} information at the latent level. A parameterized convolutional NODE architecture is introduced for the approximation of the latent dynamics. Specifically, an \textit{affine-modulation} scheme is adopted to include the parametric information into the latent representation.

The introduced LDM framework is tested on high-dimensional dynamical systems arising from the semi-discretization of parameterized nonlinear time-dependent PDEs. Specifically, the nonlinearity can either result from the presence of nonlinear terms into the governing equations or from nonlinear parametric dependencies, leading to nonaffine scenarios that are difficult to address using traditional ROMs.
In particular, two problems are considered: {\em (i)} one-dimensional Burgers' equation, {\em (ii)} two-dimensional advection-reaction diffusion equation. 
These test cases are intended to demonstrate that the introduced framework exhibits a {time-continuous approximation} property, resulting in the capability of zero-shot generalization to finer temporal discretizations than the one employed at training time.

In summary, the paper is organized as follows: after an extensive literature review, Section 2 details the problem formulation and introduces the time-continuous approximation property. Section 3 presents the LDM framework in a time-continuous setting, and then covers the discrete and learnable LDM formulations. Section 4 discusses the architectural choices in the DL context. Numerical experiments and results are presented in Section 5. The paper concludes with Section 6, summarizing the main findings and outlining possible future directions.

\subsection{Literature review}
The concept of latent dynamics has been widely explored in the field of deep learning in order to properly model the evolution of a reduced representation (latent state) of high-dimensional temporally-evolving observations, with applications ranging from planning and control \cite{watter2015embed, hafner2019learning} to computer graphics \cite{https://doi.org/10.1111/cgf.13645}. 
The concept emerged with the application of autoencoders (AEs) in a dynamical context, to reduce the dimensionality of time-evolving data, leading to the development of \textit{dynamical} AEs \cite{girin2020dynamical}. 
In particular, modeling the temporal and parametric relations of reduced representations is a central aspect in the construction of ROMs, as highlighted in Section \ref{sec-2}. In the following, we examine the main approaches used to model latent state temporal and parametric relations in the context of DL-based reduced-order and surrogate PDEs modeling, with a focus on parameterized multi-query settings.

\paragraph{Direct (regressive) approaches.}

A significant class of data-driven approaches models the latent dynamics as a map from the time-parameter space to the reduced manifold using regression-based frameworks.
These methods aim to directly learn such mappings via feedforward NNs \cite{WANG2019289, fresca2021comprehensive}, explicitly capturing the dependence of the latent state on temporal and parametric inputs. 
This approach was extended in \cite{fresca2022pod} by adopting a hybrid POD-AE architecture to reduce the offline computational costs associated with expensive training procedures, resulting in the POD-DL-ROM architecture. Additionally, the approximation capabilities of DL-ROMs have been analyzed in \cite{franco2023deep, brivio2023error}.
Similar approaches are employed into the field of \textit{operator learning} (OL) \cite{kovachki2023neural}, where recent observations suggest that learning operators in the latent space leads to improvements in accuracy and computational efficiency \cite{Kontolati2024}. In \cite{brivio2024ptpi} the concepts of DL-ROMs and OL are combined within a physics-informed framework, where a fine-tuning strategy is proposed to improve generalization capabilities in a small data regime.

\paragraph{Recurrent \& autoregressive approaches.}
The landscape of data-driven reduced order modeling methods based on latent dynamics learning has been primarily shaped by the adoption of recurrent architectures as the main time-advancing scheme within the latent space, which, unlike direct approaches, leverage the sequential nature of the time-dependent reduced state. More broadly speaking, recurrent neural networks (RNNs) have found extensive application in high-dimensional dynamical systems modeling \cite{doi:10.1098/rspa.2017.0844, vlachas2020backpropagation} and in the construction of DL-based ROMs for time-dependent problems \cite{Conti_2023, TORZONI2023110376}.
RNN-based techniques have been first coupled with traditional linear dimensionality reduction techniques, adopting long short-term memory (LSTM) \cite{6795963, 6789445} units to model the reduced state evolution obtained by means of POD \cite{kani2017drrnndeepresidualrecurrent, https://doi.org/10.1002/fld.4416, LARIO2022111475}.
Considering AE-based architectures, \cite{gonzalez2018deep} proposes a recurrent convolutional AE coupled with LSTMs to model the latent state evolution. 
Among hybrid approaches, leveraging both POD and AEs in a recurrent setting, we cite the POD-LSTM-ROM architecture \cite{10.3934/mine.2023096}.
Concerning fluid flow modeling applications, LSTMs units are also adopted in \cite{https://doi.org/10.1111/cgf.13620}, where a temporal context, encoding the current state and a series of previous states, is leveraged by the RNN to produce a latent sequence which is decoded back to the high-dimensional space. Similarly, the framework of \textit{learning the effective dynamics} (LED), proposed in \cite{vlachas2021accelerated, vlachas2022multiscale}, relies on an AE-LSTM architecture and adopts an end-to-end formulation to simultaneously learn the low-dimensional latent representation and its dynamics, rather than using a two-stage approach. This framework has been further developed in \cite{KICIC2023116204}, incorporating uncertainty quantification techniques and leveraging continual learning. 
Regarding autoregressive approaches, \cite{wu2022learning} propose the \textit{latent evolution of partial differential equations} (LE-PDE) framework, which adopts a latent time-advancing scheme based on a residual update. This approach employs both a dynamic and a static encoder, to respectively reduce the full-order state dimensionality and embed system parameters, allowing for a more flexible way of handling parametric information. Similarly, \cite{li2024latent} adopts a residual update to learn a latent convolutional propagator.

\paragraph{Dynamically-motivated approaches.}
The inherent dynamical structure of the FOM has inspired the development of approaches that leverage ODE-based formulations and physics inductive biases in order to capture the underlying dynamics of the modeled problem.
The sparse identification of nonlinear dynamics (SINDy)  framework, proposed in \cite{doi:10.1073/pnas.1517384113}, and extended by \cite{Kaheman_2022} to leverage automatic differentiation through RK schemes, has been adopted in the context of latent dynamics modeling \cite{doi:10.1073/pnas.1906995116, Fukami_Murata_Zhang_Fukagata_2021, doi:10.1098/rspa.2023.0422}, allowing to perform latent system identification via sparse regression on a set of candidate basis functions.
While ensuring an interpretable representation of the latent dynamics, such a method faces difficulties in the parameterized context, lacking a proper parameterization mechanism.
To address this limitation, \cite{fries2022lasdi, bonneville2024comprehensive} propose to identify a set of local models, each one covering a sub-region of the parameter space, whose combination allows guaranteeing a specified accuracy level across the entire parameter space. 
A parameterized extension of the AE-SINDy method is presented in \cite{conti2023reduced}, in which the reduced dynamics is represented via a parameterized ODE system and continuation algorithms are adopted to track periodic responses' evolution. Other approaches include \textit{dynamic mode decomposition} (DMD) \cite{schmid2010dynamic, doi:10.1137/22M1481658}, and models based on \textit{Koopman operator} (KO) theory \cite{doi:10.1073/pnas.17.5.315, brunton2016koopman, otto2019linearly}. 
In particular, in the multi-query setting, \cite{DUAN2024112621} proposes a DMD-based approach to model parameterized latent dynamics by fitting multiple latent models for different parameter instances, and then relying on interpolation during the online stage.
A variation of \cite{vlachas2022multiscale} based on KO and the Mori-Zwanzig formalism is proposed in \cite{menier2023interpretable}.

\paragraph{Implicit dynamics learning.}
At the intersection of deep learning and differential equations \cite{E2017}, the recently introduced concept of \textit{neural ODEs} (NODEs) \cite{NEURIPS2018_69386f6b, kidger2022neuraldifferentialequations} has enabled a novel paradigm in modeling time-dependent data \cite{NEURIPS2019_42a6845a, huang2020learning}, departing from traditional discrete-time approaches. NODEs' ability to learn continuously-evolving representations, by implicitly learning the vector field parameterizing the underlying dynamics, has bridged the gap between DNNs and continuous dynamical systems \cite{ayed2019learningdynamicalsystemspartial, djeumou2022neural, LINOT2023111838}. 
This has motivated their adoption in the context of reduced order modeling for time-dependent nonlinear problems \cite{Di_Sante_2022, farenga2022neural}, by coupling NODEs with AEs, leading to the AE-NODE architecture. Such a paradigm has been adapted to physics-informed settings by \cite{Sholokhov2023} to address data-scarce scenarios.
Specifically, in the context of multi-query reduced order modeling, \cite{Lee_2021} extended the concept of NODEs to include the parametric information, leading to the concept of \textit{parameterized neural ODEs} (PNODEs), resulting in improved generalization capabilities in a parameterized context and extending the applicability of AE-(P)NODE-based ROMs in a multi-query setting.
Recently, NODEs-based latent dynamics modeling has been extended towards a spatially-continuous formulation via implicit neural representation (INR)-based decoders \cite{yin2022continuous, wen2023reduced}.

\section{Towards a time-continuous ROM approximation}
\label{sec-2}
In this section, we introduce the problem we are interested in and review the construction of a ROM by providing a brief overview of the main steps involved, namely, the computation of a suitable trial manifold and the identification of a reduced (latent) dynamics on it.

\subsection{Problem formulation}
Our focus is on the solution of time-dependent nonlinear parameterized PDEs, whose parametric dependence is characterized by a set of parameters $\bmu\in\mathcal{P}\subset \mathbb{R}^{n_\mu}$, with $\mathcal{P}$ compact, representing system configurations, physical or geometrical-related properties, boundary conditions and/or initial conditions. To ensure consistency with the formulation of the methods proposed later, we adopt an algebraic approach to the problem. Indeed, by means of numerical methods (e.g., finite element method, spectral element method), the semi-discretized high-fidelity problem consists of a high-dimensional dynamical system, which we refer to as full-order model (FOM). Given an instance of the parameters $\boldsymbol \mu \in \mathcal{P}$, we thus consider the following initial value problem (IVP)
\begin{equation}
    \begin{cases}
        \dot{\bu}_h(t;\bmu) = \bff_h(t,\bu_h(t;\bmu);\bmu), & t \in (t_0,T],\\
        \bu_h(t_0;\bmu) = \bu_{0,h}(\bmu),
    \end{cases}
    \label{eq:FOM}
\end{equation}
where $\bu_h:[t_0,T]\times\mathcal{P}\rightarrow\mathbb{R}^{N_h}$ denotes the high-fidelity parameterized solution of \eqref{eq:FOM} and $\bu_{0,h}:\mathcal{P}\rightarrow\mathbb{R}^{N_h}$ the initial value. The nonlinear function $\bff_h:(t_0,T]\times \mathbb{R}^{N_h}\times \mathcal{P}\rightarrow \mathbb{R}^{N_h}$ defines the parameterized dynamics of the system.
Thus, the aim of reduced order modeling techniques is to provide an approximation of the \textit{solution manifold}
\begin{equation*}
    \mathcal{S}_h = \{\bu_h(t;\bmu): t\in[t_0,T],\bmu \in \mathcal{P}\}\subset\mathbb{R}^{N_h},
\end{equation*}
that is, the set of solutions of \eqref{eq:FOM} in a suitable time interval $t \in [t_0,T]$ and for a prescribed set of parameters $\bmu \in \mathcal{P}\subset\mathbb{R}^{n_\mu}$.

\paragraph{Solution manifold.} The approximation is performed through the construction of the so-called \textit{reduced trial manifold}, denoted by $\tilde{\mathcal{S}}_h^n$, which can be either linear or nonlinear, depending on the chosen dimensionality reduction technique.
In both linear and nonlinear cases, the fundamental assumption is that the intrinsic dimensionality, denoted by $n$, of the solution manifold $\mathcal{S}_h$ is significantly smaller than the FOM dimension $N_h$, i.e. $n\ll N_h$, formally stated as the \textit{manifold hypothesis}. Such assumption, in the general nonlinear case, translates into the existence of a nonlinear map $\Psi':\mathbb{R}^n\rightarrow\mathbb{R}^{N_h}$ such that 
\begin{equation}
    \label{eq:psi}
    \bu_n(t;\bmu) \mapsto \Psi'(\bu_n(t;\bmu)) = \tbu_h(t;\bmu) \approx \bu_h(t;\bmu),
\end{equation}
where $\bu_n(t;\bmu) : [t_0,T]\times\mathcal{P}\rightarrow\mathbb{R}^{n}$ are the ROM intrinsic coordinates, or latent variables, describing the dynamics of the system onto the low-dimensional trial manifold.
Thus, the reduced \textit{nonlinear} trial manifold can be characterized as follows 
\begin{equation*}
    \mathcal{S}_h \approx \tilde{\mathcal{S}}_h ^n= \{\Psi'(\bu_n(t;\bmu)): \bu_n(t;\bmu) \in \mathbb{R}^n,  t\in[t_0,T],\bmu \in \mathcal{P}\}\subset\mathbb{R}^{N_h},
\end{equation*}
and the associated dimensionality reduction problem takes the form
\begin{equation}
    \min_{\Psi',\Psi} \|\bu_h - \Psi'\circ\Psi(\bu_h)\|^2_{L^2([t_0,T] \times \mathcal{P}; \mathbb{R}^{N_h})}.
    \label{eq:nonlinear_rom_pb}
\end{equation}

The adoption of nonlinear dimensionality reduction paradigms has been mainly driven by the advancements in the field of DL \cite{lee2020model}, where the nonlinear maps $\Psi,\Psi'$, referred to as \textit{encoder} and \textit{decoder}, with $\Psi:\mathbb{R}^{N_h}\rightarrow \mathbb{R}^{n}$, are parameterized by NNs. Their composition $\Psi'\circ\Psi :\mathbb{R}^{N_h}\rightarrow\mathbb{R}^{N_h}$ approximating the identity $\operatorname{Id}_{N_h}$, is referred to as \textit{autoencoder} (AE) \cite{hinton1993autoencoders}, and the reduced state $\bu_n(t;\bmu) = \Psi(\bu_h(t;\bmu)) \in \mathbb{R}^n$ is known as the \textit{latent} state. Moreover, we highlight that, in the linear case, the map (\ref{eq:psi}) is represented by a matrix $V\in \mathbb{R}^{N_h\times n}$, which leads to projection-based methods.

\paragraph{Reduced dynamics.} The second main step in the construction of a ROM is represented by the identification of the reduced dynamics on the reduced trial manifold. In this regard, multiple strategies can be adopted. In a classical projection-based ROM, the time-evolution of the reduced dynamics is modeled by means of a parameterized dynamical system of dimension $n$. This involves replacing the high-fidelity solution $\bu_h(t;\bmu)$ in \eqref{eq:FOM} with the projection-based approximation $V\bu_n(t;\bmu)$, and imposing that the residual associated to the first equation of \eqref{eq:FOM} is orthogonal to the $n$-dimensional subspace spanned by the columns of the matrix $V$, thus obtaining 
\begin{equation}
\begin{cases}
    \dot{\bu}_n(t;\bmu) = V^T\bff_h(t,V\bu_n(t;\bmu);\bmu), & t \in (t_0,T],\\
    \bu_n(t_0;\bmu) = V^T\bu_{0,h}(\bmu).
    \label{eq:ROM}
\end{cases}
\end{equation}
The three main bottlenecks often arising with projection-based ROMs are: {\em (i)} the high dimension $n \gg n_{\boldsymbol{\mu}} + 1$ of the low-dimensional subspace, much larger than the intrinsic dimension of the solution manifold, {\em (ii)} the need to rely on hyper-reduction techniques, such as the Empirical Interpolation Method (EIM) and its discrete variant \cite{BARRAULT2004667, RYCKELYNCK2005346, M2AN_2007__41_3_575_0, DEIM}, to assemble the operators appearing in the ROM \eqref{eq:ROM} in order not to rely on expensive $N_h$-dimensional arrays, and {\em (iii)} their intrusive nature, due to the fact that the dynamics $\bff_h$ must be known in order to construct the ROM. 

More recently, in order to overcome these critical issues, regression-based approaches have been adopted, by modeling the relationship $(t,\bmu)\mapsto \bu_n(t;\bmu)$ via a mapping $\phi:[t_0,T]\times\mathcal{P} \rightarrow \mathbb{R}^n$. Such strategy has been employed in multiple DL-based reduced order modeling approaches \cite{hesthaven2018, fresca2021comprehensive, WANG2019289,bhattacharya2021}; indeed, due to its regression-based formulation, it avoids solving the reduced dynamical system \eqref{eq:ROM} by directly modeling the relationship between the parameters and the latent variables. However, we emphasize that, in contrast to classical projection-based ROMs, regression-based techniques lose the IVP structure of the problem. Moreover, another issue of such approaches is represented by their poor performance at time-extrapolation tasks and interpolating the solution at time instances in between the discrete training data. This has motivated the adoption of recurrent-based approaches to model the reduced dynamics, which partially solve the previous issues \cite{wang2020recurrent, 10.3934/mine.2023096}.

\subsection{Time-continuous approximation}
In the reduced order modeling context, most of the attention, at least up to now, has been placed on trying to decrease as much as possible the spatial complexity. Indeed, less attention has been focused on the modeling and treatment of the temporal evolution. Motivated by this, we aim at developing a framework whose structure resembles a dynamical system and is capable of generating ROMs that capture the underlying temporal dynamics of the modeled system, in the sense that the approximation can be queried at any arbitrary time instance by retaining a sufficient accuracy. Indeed, to enhance generalization capabilities at time interpolation and extrapolation tasks, we have identified the following key properties: 
\begin{enumerate}
\item \textit{Zero-shot time-resolution invariance.} The approximation error between the FOM and ROM solutions is almost constant as the time discretization is refined, even if the ROM is constructed using data sampled at a specific time-resolution;
\item \textit{Causality.} The approximation at the next step depends on the previous steps;
\item \textit{Initial value problem (IVP) structure.} The ROM preserves the FOM structure, and only the initial value is needed to predict the time-evolution.
\end{enumerate}

Property 1 refers to the capability of the ROM to represent the continuous dynamics of a given system -- that is, our model shall approximate a continuous differential operator, rather than only discrete points. 
This property is not granted a priori \cite{queiruga2020continuousindepth, krishnapriyan2023learning}. Indeed, several ROM techniques are fundamentally tied to the time discretization employed during the offline training stage. This dependency on time discretization carries over to recent DL-based approaches as well. For instance, pure regressive approaches may suffer from overfitting with respect to the time variable, leading to increasing errors when querying such models on a finer temporal grid than the one used during the training phase. As a result, these models eventually require rebuilding, additional training or fine-tuning stages, if they need to be tested on novel temporal discretizations, in contrast with property 1. 

Property 2  is related to the time-directionality of dynamical systems and is required at both offline training and online testing times, and it could be achieved, for instance, by means of recurrent architectures or ODE-Nets based on explicit RK schemes, given their causal nature.

Property 3 is deeply motivated by the IVP-structure of the FOM, with the aim of modeling PDEs' solution evolution by means of a model architecture which has an IVP-like structure, thus only the initial value is needed in order to predict the time evolution. So, the desired framework is different from other DL-based surrogate models employed in the literature in the context of evolutionary problems, usually recurrent or autoregressive-based, typically relying on an input context window of multiple previous high-fidelity snapshots (look-back) to produce the next step. 
From a computational perspective, this aspect compromises the full independence of the ROM from the FOM in terms of online costs, potentially limiting the efficiency of the resulting ROM.
Moreover, autoregressive approaches, despite showing short-term accuracy, often lack long-term stability, requiring ad hoc methods to promote stable long roll-outs \cite{brandstetter2022message, NEURIPS2023_d529b943}. 

The previous three concepts led us to the definition of a \textit{time-continuous approximation} property, which is an extension of the one provided in \cite{krishnapriyan2023learning}. Properties 2 and 3 affect the NN structure of the ROM solution depending only on the initial datum, the previous time instances, and the physical parameter $\bmu$, while property 1 results in a practical criterion to assess the generalization capabilities of the approximation with respect to the temporal discretization.

\begin{definition}[Time-continuous approximation]
\label{def:timecont}
Let $[t_{0}, T]\subset\mathbb{R}_+$ be a generic time interval and $\bmu \in \mathcal{P} \subset \mathbb{R}^{n_\mu}$. 
Let us denote by $\mathcal{T}_{\Delta t} = \{t_i\}_{i=0}^{N_{\Delta t}}$ a discretization of $[t_0,T]$, with $t_i = t_0+i\Dt$ and $\Dt>0$ a chosen time step.
Suppose that the ROM surrogate $\hbu_h^i(\boldsymbol{\mu}) = \boldsymbol{\Theta}(\bu_{0,h}, \{t_j\}_{j=0}^i; \bmu) \approx \mathbf{u}_h(t_i; \boldsymbol{\mu})$ is trained over data generated with a time discretization $\ \mathcal{T}_{\Dt_{train}} = \{t_k\}_{k=0}^{N_{\Dt_{train}}}$, $t_k = t_0+k\Dt_{train}$, with a time step $\Dt_{train}>0$. 
Then, the ROM solution is said to be a \textit{time-continuous approximation} if it holds
\begin{equation}
\sup_{t_i \in \mathcal{T}_{\Delta t}} || \mathbf{u}_h(t_i; \boldsymbol{\mu}) -\hbu_h^i(\boldsymbol{\mu})|| \leq C\sup_{t_k \in \mathcal{T}_{\Dt_{train}}} || \mathbf{u}_h(t_k; \boldsymbol{\mu}) - \hbu_h^k( \boldsymbol{\mu})||, \quad \forall \Delta t \le \Dt_{train}, \quad \forall \boldsymbol{\mu} \in \mathcal{P},
\label{eq:time-continuous}
\end{equation}
with $C = C(\Dt_{train},|T-t_0|)$ independent of $\Delta t$.
\end{definition}

In other words, we are asking that the maximum error over any time discretization, employing $\Dt<\Dt_{train}$, is of the same order as the maximum error over the training time discretization, that is, the training accuracy is preserved on finer time discretizations employed during testing, up to a constant that is independent of the test discretization. Such property highlights that the model has time generalization capabilities, meaning that the learned dynamics meaningfully captures the underlying phenomena at a continuous level, rather than simply fitting the discrete data points. In the following, we aim at showing that by means of latent dynamics models, it is possible to ensure the fulfillment of the time-continuous approximation property at the high-dimensional level, that is, to obtain a time-continuous ROM-based approximation $\hbu_h(t;\bmu) \in \mathbb{R}^{N_h}$.

\section{Modeling of the latent dynamics}

In the following, we formalize the concept of \textit{latent dynamics model} (LDM), in the context of dimensionality reduction techniques for parameterized nonlinear dynamical systems. Our abstract setting allows understanding the temporal behavior of such modeling framework, starting from the time-continuous setting, which can be assimilated to the case of infinite time observations,  with the end goal of deriving suitable error bound and stability estimates for the approximated solution.
Developing a learnable framework operating in a time-continuous context represents a crucial step to {\em (i)} remove the dependence on the time-step employed at training time, {\em (ii)} be able to deal with irregularly spaced time grids, and {\em (iii)} enhance long-term time-extrapolation capabilities beyond the training time horizon. In particular, such modeling framework would allow us to generalize to temporal grids much finer than the one employed during the offline stage, thus removing the need of rebuilding the ROM, redefining the model architecture, or requiring additional training and fine-tuning on the novel discretizations. The time-continuous setting will serve as starting point for the introduction of numerical integration schemes to solve the temporal dynamics in a space of reduced dimensions, while showing that the resulting time-discrete framework satisfies classical properties of temporal discretization schemes for the numerical solution of dynamical systems. Finally, we present a learnable setting where the different components are approximated by means of NNs. In particular, we aim at studying the approximation properties of the ROM solution, by showing that the approximation error between the ROM and the FOM solutions is bounded, while addressing its time-continuous approximation capability.

\subsection{Latent dynamics models}
\label{sec:ldm}

In nonlinear dimensionality reduction, a key challenge, regardless of the nature of the maps adopted in the dimensionality reduction technique, lies in modeling the dynamics of the reduced (latent) state $\bu_n(t;\bmu)$. Characterizing the problem in a time interval $[t_0,T]$, and assuming a sufficient regularity of the latent state, that is, $\bu_n \in C^1([t_0,T] \times \mathcal{P};\mathbb{R}^{n})$, the latter task entails the discovery of $\bff_n:(t_0,T]\times\mathbb{R}^n\times\mathcal{P}\rightarrow\mathbb{R}^n$ encoding the latent state dynamics, so that
\begin{equation}
    \dot{\bu}_n(t;\bmu) =  \bff_n(t,\bu_n(t;\bmu);\bmu).
\label{eq:latent_dyn_implicit}
\end{equation}
Taking advantage of \eqref{eq:latent_dyn_implicit}, in the following we properly address the problem of characterizing both the nonlinear dimensionality reduction and the identification of the latent dynamics. We first provide a formal definition of \textit{latent dynamics problem} (LDP), and we proceed by introducing the concept of \textit{latent dynamics model} (LDM), defining the mathematical context of time-continuous dimensionality reduction in a general setting.

\begin{definition}[Latent dynamics problem]\label{def:ldp}
Let $\bu_{h}:[t_0,T]\times\mathcal{P}\rightarrow \mathcal{S}_{h}\subset\mathbb{R}^{N_h}$ be a high-dimensional parameterized time-dependent state, such that $\bu_h \in C^1([t_0,T] \times \mathcal{P};\mathbb{R}^{N_h})$, evolving accordingly to the first-order IVP
\begin{equation}
    \left\{
    \begin{array}{ll}
    \dot{\bu}_{h}(t;\bmu) = \bff_{h}(t,\bu_{h}(t;\bmu);\bmu), & \quad t\in (t_0, T], \\
    \bu_{h}(t_0;\bmu)=\bu_{0,h}(\bmu), & \\
    \end{array}
    \right.
    \label{eq:ldp_hf_ivp}
\end{equation}
with $\bff_{h} : (t_0,T]\times\mathbb{R}^{N_h} \times\mathcal{P} \supset D_h \rightarrow\mathbb{R}^{N_h}$ the high-dimensional nonlinear vector field describing its dynamics, for each $\bmu \in \mathcal{P}$. Let $\bu_n: [t_0,T]\times\mathcal{P}\rightarrow \mathcal{S}_{n}\subset\mathbb{R}^{n}$ be a low-dimensional parameterized time-dependent state ($n\ll N_h$), such that $\bu_n \in C^1([t_0,T]\times \mathcal{P};\mathbb{R}^{n})$. 
Then, the low-dimensional state $\bu_n(t;\bmu)$ defines a latent dynamics for the high-dimensional state evolution $\bu_{h}(t;\bmu)$ if there exists a triple $(\Psi^*,\Psi'^*,\bff^*_n)$ solving the following \textit{latent dynamics problem}
\begin{equation}
\min_{\Psi', \bu_n}  \| \bu_h  - \Psi'(\bu_n) \|^2_{L^2([t_0,T] \times \mathcal{P}; \mathbb{R}^{N_h})} \quad
 \textnormal{ s.t. } \quad \left\{
\begin{array}{ll}
\dot{\bu}_n(t;\bmu) = \bff_n(t,\bu_n(t;\bmu);\bmu), & \quad t\in (t_0,T], \\  \bu_n(t_0;\bmu)=\Psi(\bu_{0,h}(\bmu)),
\end{array}
\right.
\label{eq:ldp}
\end{equation} 
with $\bff_n : (t_0,T]\times\mathbb{R}^n \times\mathcal{P} \supset D_n\rightarrow \mathbb{R}^n$ the low-dimensional nonlinear vector field describing the dynamics, while $\Psi:\mathbb{R}^{N_h}\rightarrow\mathbb{R}^{n}$ and $\Psi':\mathbb{R}^n\rightarrow\mathbb{R}^{N_h}$ are Lipschitz-continuous mappings.
\end{definition}

We emphasize that $D_h=(t_0,T]\times\mathcal{S}_{h}\times \mathcal{P}$ and $D_n=(t_0,T]\times\mathcal{S}_{n}\times \mathcal{P}$ indicate the domains of $\bff_h$ and $\bff_n$, respectively. Then, minimal requirements to ensure well-posedness of the high- and low-dimensional IVPs \eqref{eq:ldp_hf_ivp}-\eqref{eq:ldp} on the vector fields $\bff_h$ ($\bff_n$), are { \em (i)} continuity in $D_h$ ($D_n$), and { \em (ii)} local Lipschitz continuity in $D_h$ ($D_n$) with respect to $\bu_h$ ($\bu_n$), uniformly in $t$ and $\bmu$; in this way, we ensure local existence and uniqueness of the solutions, with $D_h$ and $D_n$ open in $(t_0, T]\times \mathbb{R}^{N_h} \times \mathcal{P}$ and $(t_0, T]\times \mathbb{R}^{n} \times \mathcal{P}$, respectively. Additionally, boundedness of $\bff_h$ and $\bff_n$ in $\bar{D}_h$ and $\bar{D}_n$, respectively, ensures global existence of the solutions of \eqref{eq:ldp_hf_ivp}-\eqref{eq:ldp}. 

\begin{definition}[Latent dynamics model]
    The triple $(\Psi,\Psi',\bff_n)$ satisfying \eqref{eq:ldp} is called \textit{latent dynamics model} for $\bu_{h}(t;\bmu)$, and $\tbu_h(t;\bmu)$ is the resulting approximation reading in an explicit form as
    \begin{equation}
        \bu_{h}(t;\bmu) \approx \tbu_h(t;\bmu) \coloneqq \Psi'\bigg(\Psi(\bu_{0,h}(\bmu)) + \int_{t_0}^t \bff_n(s,\bu_n(s;\bmu);\bmu)ds\bigg).
        \label{eq:ldm}
    \end{equation}
\end{definition}

Let us discuss the regularity properties introduced in Definition \ref{def:ldp}. Specifically, the smoothness assumptions on the high-dimensional parameterized state $\bu_h$ and its low-dimensional \textit{latent} counterpart $\bu_n$, belonging to a functional space of the form $C^1([t_0,T]\times\mathcal{P};\mathbb{R}^{*})$, entail two key properties. Firstly, it ensures that the states are continuously differentiable with respect to the temporal variable over $[t_0,T]$, secondly that the parameter-to-solution map $(t,\bmu)\mapsto \bu_n(t;\bmu)$ is Lipschitz-regular.
Moreover, we highlight that the LDP is formulated as a joint minimization problem with respect to both the maps $\Psi, \Psi'$ and the latent dynamics $\bff_n$, where the latter play a central role in the dimensionality reduction aspect of the proposed framework. In contrast to DL-ROM paradigms \cite{fresca2021comprehensive}, the role of $\Psi$ is to map only the high-dimensional initial value $\bu_{0,h}\in\mathbb{R}^{N_h}$ to $\Psi(\bu_{0,h})\in\mathbb{R}^{n}$ as the dynamics is recovered through the evolution of the latent state, via the integration of $\bff_n$. On the other hand, the mapping $\Psi'$ acts at any time instance, allowing to project the latent state evolution $\bu_n(t;\bmu)$ back to the high-dimensional space. We refer to Figure \ref{fig:ldm-commutative} for a visualization of the overall dimensionality reduction flow. Nonetheless, we emphasize that the choice of the latent dimension $n$ is essential to calibrate the accuracy of LDMs; in the following, we present a framework to suitably address this matter.

\begin{figure}[ht]
\centering
\[\begin{tikzcd}
{\bu_{0,h}} && {\bu_h(t)} & {\mathbb{R}^{N_h}} \\
\\
{\bu_{0,n}} && {\bu_n(t)} & {\mathbb{R}^n}
\arrow["\Psi"', maps to, from=1-1, to=3-1]
\arrow["{\Psi'}"', maps to, from=3-3, to=1-3]
\arrow["{\dot{\bu}_h = \bff_h}", maps to, from=1-1, to=1-3, dashed]
\arrow["{\dot{\bu}_n = \bff_n}", maps to, from=3-1, to=3-3]
\end{tikzcd}\]
\caption{Commutative diagram explaining the LDM scheme, approximating the mapping $\bu_{0,h} \mapsto \bu_h(t)$ defined by the full-order dynamics $\bff_h$, by means of a latent dynamics $\bff_n$ and the nonlinear projection and reconstruction maps $\Psi, \Psi'$.}
\label{fig:ldm-commutative}
\end{figure}
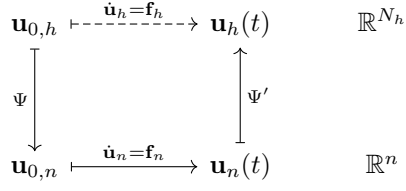

\subsubsection{Nonlinear dimensionality reduction in a dynamical setting}
In linear dimensionality reduction, the selection of the reduced dimension $n$ is determined by the decay of the singular values of the snapshot matrix \cite{quarteroni2015reduced}, while the characterization of $n$ in a nonlinear setting requires a different theoretical environment. Specifically, here, we aim at characterizing the nonlinear dimensionality reduction framework in a dynamical setting, starting from the LDM formulation. To do that, we first provide a preliminary result that unveils the contribution of $\Psi, \Psi'$ and $\bff_n$ to the approximation error, namely,

\begin{proposition}
\label{error_decomposition}
    For any $\bmu \in \mathcal{P}$, under the hypotheses of Section \ref{sec:ldm}, it holds that
    \begin{align*}
        \|\bu_h(t;\bmu) - \tbu_h(t;\bmu)\| \le  \| \bu_{0,h} &- \Psi' \circ \Psi (\bu_{0,h}) \|_{L^\infty(\mathcal{P}; \mathbb{R}^{N_h})} \\ &+ |T-t_0| \|\bff_h(\bu_h) - \mathbf{J}_{\Psi',\bu_n} \bff_n(\bu_n)) \|_{L^\infty([t_0,T] \times \mathcal{P}; \mathbb{R}^{N_h})},
    \end{align*}
    for any $t \in [t_0,T]$, where $\mathbf{J}_{\Psi',\bu_n}$ is the Jacobian of $\Psi'$ with respect to $\bu_n$.
\end{proposition}

We refer the reader to Appendix \ref{sec:additional_proofs} for the formal proof. We remark that Proposition \ref{error_decomposition} reveals the deep connection between the maps $\Psi,\Psi'$ involved in the dimensionality reduction of the full order state $\bu_h(t;\bmu)$ and the latent dynamics $\bff_n(t,\bu_n(t; \boldsymbol{\mu}); \bmu)$. In particular, the first term in the upper bound is related to the capability of the encoder and the decoder to reconstruct the initial condition. On the other hand, the second term depends on $\Psi'$ and $\bff_n$, and describes how accurate is the latent state dynamics approximation of the exact high-fidelity dynamics. We emphasize that as the latent dimension $n$ increases, we expect that both terms converge quickly to $0$. To characterize such convergence, in the wake of other works concerning the approximation capabilities of AE-based architectures \cite{franco2023deep,brivio2023error}, we introduce the perfect embedding assumption, specifically crafted for latent dynamics modeling. In particular, we argue that the true FOM dynamics can be exactly represented through a sufficiently rich LDM, according to the following definition.

\begin{definition}(Perfect embedding assumption)
\label{perfect_embedding} There exists
$n^* \ll N_h$ such that, for any $n \ge n^*$, there exist 
\begin{itemize}
    \item a decoder function $\Psi' \in C^1(\mathbb{R}^n;\mathbb{R}^{N_h})$,
    \item a latent dynamics function $\mathbf{f}_n(t, \mathbf{u}_n(t; \boldsymbol{\mu}); \bmu)$, which is Lipschitz continuous with respect to $\bu_n$, uniformly in $t$, and Lipschitz continuous with respect to $t$,
\end{itemize}
such that
\begin{equation}
{\mathbf{u}}_h(t; \boldsymbol{\mu}) = \Psi'(\mathbf{u}_n(t; \boldsymbol{\mu})) \quad \textnormal{and} \quad \dot{\bu}_n(t; \boldsymbol{\mu}) = \bff_n(t,\bu_n(t; \boldsymbol{\mu})) \quad \forall t \in (t_0, T].
\end{equation}
Moreover, for any $n \ge n^*$ there exists $\Psi \in C^1(\mathbb{R}^{N_h};\mathbb{R}^{n})$ such that $\mathbf{u}_n(t_0; \boldsymbol{\mu}) = \Psi(\mathbf{u}_h(t_0; \boldsymbol{\mu}))$. The previous assumptions hold for any $\boldsymbol{\mu} \in \mathcal{P}$.
\end{definition}

The purpose of the perfect embedding assumption is both to provide a suitable heuristic criterion to determine the latent dimension $n$ and the regularity of $\Psi'$, $\Psi$ and $\bff_n$. We refer the interested reader to \cite{franco2023deep} for more details on the characterization of the latent dimension in nonlinear dimensionality reduction. Moreover, we highlight a deep connection between the error decomposition formula of Proposition \ref{error_decomposition} and the Assumption \ref{perfect_embedding}; indeed, it is straightforward to see that if the perfect embedding assumption is attained, then  both terms in the upper bound of Proposition \ref{error_decomposition} vanish.
Finally, we remark that Assumption \ref{perfect_embedding} will play a crucial role in the subsequent sections of this work, especially for the characterization of the time-continuous approximation property. Before delving into that matter, we first focus on the properties and numerical analysis of LDMs.

\subsubsection{Stability} 
The dynamical nature of the proposed framework, preserving the ODE-IVP structure of the FOM, enhances the model's interpretability from a mathematical perspective, thereby enabling us to address specific aspects such as \textit{stability}, and later, its numerical properties in a time-discrete setting.
Specifically, LDMs' integral formulation \eqref{eq:ldm} defines a parameterized problem of the form
\begin{equation}
    \begin{cases}
        \tbu_h(t;\bmu) = \Psi'(\bu_n(t;\bmu)), & t\in (t_0,T],\\
        \dot{\bu}_n(t;\bmu) = \bff_n(t,\bu_n(t;\bmu);\bmu), & t\in (t_0,T],\\
        \bu_n(t_0;\bmu) = \Psi(\bu_{0,h}(\bmu)),
    \end{cases}
    \label{eq:ldm_ivp}
\end{equation}
where the latter two equations define a \textit{latent} IVP in $\mathcal{S}_n\subset\mathbb{R}^n$, while the former defines the mapping of the latent state $\bu_n(t;\bmu)$ to the high-dimensional space $\mathcal{S}_h\subset\mathbb{R}^{N_h}$, producing an approximation $\tbu_h(t;\bmu)$ of the FOM state. In order to analyze the stability of \eqref{eq:ldm_ivp} in the sense of Lyapunov \cite{arnold1978ordinary, Quarteroni2007},  we consider the following perturbed problem 
\begin{equation}
    \begin{cases}
        \tbz_h(t;\bmu) = \Psi'(\bz_n(t;\bmu)), & t\in (t_0,T],\\
        \dot{\bz}_n(t;\bmu) = \bff_n(t,\bz_n(t;\bmu);\bmu) + \bdelta(t), & t\in (t_0,T],\\
        \bz_n(t_0;\bmu) = \Psi(\bu_{0,h}+\bdelta_0),
    \end{cases}
    \label{eq:ldm_ivp_pert}
\end{equation}
denoting by $(\bdelta_0,\bdelta(t)) \in \mathbb{R}^{N_h} \times \mathbb{R}^{n}$ two perturbations such that $\|\bdelta_0\|< \varepsilon,\  \|\bdelta(t)\| < \varepsilon, \ \forall t \in (t_0,T]$. In particular, by setting $\varepsilon > 0$ such that the latent IVP is well-defined, we aim to show the existence of a constant $C>0$, independent of $\varepsilon$, such that the following bound on the high-dimensional solutions holds: 
\begin{equation*}
    \|\tbz_h(t;\bmu) - \tbu_h(t;\bmu)\| < C\varepsilon, \qquad \forall t \in (t_0, T].
\end{equation*}
The goal is to assess the sensitivity of $\tbu_h(t;\bmu)$ with respect to the introduced perturbations, and to unveil the role of the maps $\Psi$ and $\Psi'$ in guaranteeing stability, together with the appropriate requirements. We proceed by fixing $\bmu \in \mathcal{P}$, and by considering
\begin{equation}
    \|\tbz_h(t) - \tbu_h(t)\| = \|\Psi'(\bz_n(t)) - \Psi'(\bu_n(t))\| \leq Lip(\Psi')\|\bz_n(t) - \bu_n(t)\| =  Lip(\Psi')\|\bxi(t)\|, \label{eq:stability_w}
\end{equation}
having set $\bxi(t) = \bz_n(t) - \bu_n(t)$ and denoting by $Lip(\cdot)$ the Lipschitz constant of a given map. By differentiating, so that $\dot{\bxi}(t) = \bff_n(t,\bz_n(t)) - \bff_n(t,\bu_n(t)) + \bdelta(t)$, and then integrating on $(t_0,t)$, with $t\in(t_0,T]$, it follows that
\begin{equation*}
    \bxi(t) = \bz_n(t_0) - \bu_n(t_0) + \int_{t_0}^{t} \Big(\bff_n(s, \bz_n(s)) - \bff_n(s, \bu_n(s))\Big)ds + \int_{t_0}^{t}\bdelta(s)ds.
\end{equation*}
Then, 
\begin{align}
    \|\bxi(t)\| &\leq \|\bz_n(t_0) - \bu_n(t_0) \| + Lip(\bff_n)\int_{t_0}^{t}\|\bxi(s)\|ds + \int_{t_0}^{t}\|\bdelta(s)\|ds\notag\\
    &= \|\Psi(\bu_h(t_0) + \bdelta_0) - \Psi(\bu_h(t_0)) \| + Lip(\bff_n)\int_{t_0}^{t}\|\bxi(s)\|ds + \int_{t_0}^{t}\|\bdelta(s)\|ds\notag\\
    &\leq Lip(\Psi)\|\bdelta_0\| + Lip(\bff_n)\int_{t_0}^{t}\|\bxi(s)\|ds + \int_{t_0}^{t}\|\bdelta(s)\|ds\notag\\
    &\leq (Lip(\Psi) + |t-t_0|)\varepsilon + Lip(\bff_n)\int_{t_0}^t\|\bxi(s)\|ds \leq (Lip(\Psi) + |t-t_0|)\varepsilon e^{Lip(\bff_n)|t-t_0|} \label{eq:stability_gronwall}
\end{align} 
where \eqref{eq:stability_gronwall} follows from Gronwall's Lemma. Finally, considering  \eqref{eq:stability_w}, we obtain
\begin{align*}
    \|\tbz_h(t) - \tbu_h(t)\| &\leq Lip(\Psi') (Lip(\Psi) + |t-t_0|)\varepsilon e^{Lip(\bff_n)|t-t_0|} \qquad \forall t\in(t_0,T]\\
    & \leq Lip(\Psi') (Lip(\Psi) + |T-t_0|)\varepsilon e^{Lip(\bff_n)|T-t_0|}
\end{align*}
Thus, setting 
\[
C = C(\Psi,\Psi',\bff_n, |T-t_0|) = Lip(\Psi') (Lip(\Psi) + |T-t_0|) e^{Lip(\bff_n)|T-t_0|},
\]
we can write 
\begin{equation*}
    \|\tbz_h(t) - \tbu_h(t)\| \leq C\varepsilon,
\end{equation*}
with $C$ independent of the perturbation magnitude $\varepsilon$. The above analysis highlights the role of the projection maps $\Psi, \Psi'$  in ensuring the stability of \eqref{eq:ldm_ivp} by means of their Lipschitz continuity.

\subsection{Discrete latent dynamics models}
\label{section:DLDM}

The actual construction of LDMs deeply relies on numerical integration schemes for the solution of the latent IVP \eqref{eq:ldp}. Thus, we aim to describe the notion of LDM in a time-discrete setting, by considering one-step schemes for numerically approximating the time-evolution of the latent dynamics. Opting for one-step methods may appear limiting, however this choice is supported by the possibility to easily extend the framework to DL scenarios. Indeed, in such cases, explicit RK schemes are predominantly employed due to their effectiveness in balancing training efficiency and inference performance. In particular, we assume the existence of a latent state $\bu_n(t;\bmu)$ for the high-dimensional state $\bu_h(t;\bmu)$, obtained by means of a LDM $(\Psi, \Psi', \bff_n)$ solving the associated LDP. 

Thus, by introducing a discretization of the time domain $[t_0,T]$ into $N_\Dt$ intervals of width $\Dt$, with $t_k = t_0 + k\Dt$ for $k=0,\ldots,N_\Dt$, let $\bu_n^k(\bmu)$ be the numerical approximation at time step $t_k$ of the time-continuous latent state $\bu_n(t_k;\bmu)$. Adopting Henrici's notation \cite{henrici1962discrete}, a single step of an explicit RK scheme for the numerical approximation of the latent dynamics, reads as
\begin{equation}
    \bu_n^{k}(\bmu) = \bu_n^{k-1}(\bmu) + \Dt \Phi(t_{k-1}, \bu_n^{k-1}(\bmu); \Dt, \bff_n, \bmu), \qquad 
     k = 1,\ldots,N_\Dt,
    \label{eq:latent_discretization}
\end{equation}
denoting by $\Phi$ the method's increment function. Thus, by inserting \eqref{eq:latent_discretization} into the LDM framework, we can define the discrete LDM formulation ($\dldm$), reading as 
\begin{equation}
    \begin{cases}
    \tbu_h^{k}(\bmu) = \Psi'(\bu_n^{k}(\bmu)), & k = 1,\ldots,N_\Dt,\\
    \bu_n^{k}(\bmu) = \bu_n^{k - 1}(\bmu) + \Dt \Phi(t_{k - 1}, \bu_n^{k - 1}(\bmu);\Dt,\bff_n,\bmu), & k = 1,\ldots,N_\Dt,\\
    \bu_n^0(\bmu) = \Psi(\bu_{0,h}(\bmu)).\\
    \end{cases}
    \label{eq:DLDM}
\end{equation}
 As highlighted in \eqref{eq:DLDM}$_1$, numerically solving the latent IVP \eqref{eq:ldp} leads to a numerical approximation $\tbu_h^k(\bmu)$ of the time-continuous LDM high-dimensional reconstructed high-fidelity state $\tbu_h(t;\bmu)$, thus leading to a further approximation of the high-dimensional FOM state $\bu_h(t;\bmu)$ by means of $\tbu_h^k(\bmu)$. Thus, the error between the discrete approximation provided by $\dldm$ and the FOM solution, at $t_k$, can be decomposed as follows 
\begin{equation}
    \|\bu_h(t_k;\bmu) - \tbu_h^k(\bmu)\| \leq  \|\bu_h(t_k;\bmu) - \tbu_h(t_k;\bmu)\| + \|\tbu_h(t_k;\bmu) - \tbu_h^k(\bmu)\|,
    \label{eq:dldm_error_dec}
\end{equation}
taking into account (i) the first stage of approximation of the full-order state by means of the LDM in time-continuous settings, and (ii) the second stage of numerical approximation, accounting for the numerical error source due to the $\dldm$ approximating the LDM. 
In this view, it is of critical importance to characterize the concepts of \textit{convergence, consistency} and \textit{zero-stability} for the introduced discrete formulation \eqref{eq:DLDM}, in order to better understand how the properties of the numerical scheme are transferred to the $\dldm$ approximation, and to further extend the error decomposition formula \eqref{eq:error_decomposition} to the discrete case, by means of \eqref{eq:dldm_error_dec}. In the following the parameter dependence is temporarily dropped, assuming again to fix $\bmu\in \mathcal{P}$.

\paragraph{Convergence.} Given the $\dldm$ formulation, we aim to characterize its convergence, defined as 
\begin{equation*}
    \lim_{\Dt \rightarrow 0}\|\tbu_h(t_k) - \tbu^k_h\| = 0, \quad \forall k=0,\ldots,N_\Dt.
\end{equation*}
The convergence of the $\dldm$ discrete evolution to the LDM approximation is naturally related, by construction, to the convergence of the time integration scheme employed in the latent space. In particular, assuming to employ a method of order $p$, the following convergence result holds for the discretized latent problem:
\begin{equation*}
    \|\bu_n(t_k) - \bu_n^k\| = O(\Dt^p).
\end{equation*}
Thus, a convergence result for the overall framework can be built on top of the one related to the latent time-stepping scheme, by involving the reconstruction map $\Psi'$.
Indeed, by means of $\Psi'$ Lipschitz continuity, the order of convergence is preserved under the mapping from $\mathcal{S}_n\rightarrow\mathcal{S}_h$ for the high-dimensional reconstruction $\tbu_h^k$, reading as
\begin{align*}
    \|\tbu_h(t_k) - \tbu^k_h\| = \|\Psi'(\bu_n(t_k)) - \Psi'(\bu_n^k)\| \leq Lip(\Psi')\|\bu_n(t_k) - \bu_n^k\|\leq Lip(\Psi') C \Dt^p.
\end{align*}

\paragraph{Consistency.} Here, we aim to characterize how the consistency of the numerical method for the solution of the latent IVP, employed in the $\Delta$LDM formulation, is transferred to the high dimensional state, investigating the relationship between the high and low-dimensional local truncation errors (LTEs). 
In order to define the consistency property for the method employed for discretizing the latent ODE, we consider the residual $\bveps_n^{k}\in\mathbb{R}^n$ at time $t_k$, defined as
\begin{equation*}
    \bveps_n^{k} = \bu_n(t_{k}) - \bu_n(t_{k-1}) - \Dt \Phi(t_{k-1}, \bu_n(t_{k-1}); \bff_n, \Dt).
\end{equation*}
Moreover, we define the residual $\bveps_h^{k}\in\mathbb{R}^{N_h}$ that generates after one-step of the $\Delta$LDM scheme, at $t_k$, with respect to the time-continuous FOM solution, by considering of having as initial condition $\bu_h(t_k)$ 
\begin{align*}
    \bveps_h^{k} &= \underbrace{\tbu_h(t_{k})}_{\text{LDM}} - \underbrace{\Psi'(\Psi(\bu_h(t_{k-1})) + \Dt \Phi(t_{k-1}, \Psi(\bu_h(t_{k-1}));\bff_n, \Dt))}_{\text{1-step} \  \Delta\text{LDM}}.
\end{align*}
At this point, we relate the high- and low-dimensional residuals, in order to draw the connection between consistency of the LDM scheme and the one of the method employed in the latent space, as follows:
\begin{align*}
    \|\bveps_h^{k}\| &= \|\Psi'(\bu_n(t_{k})) - \Psi'(\Psi(\bu_h(t_{k-1})) + \Dt \Phi(t_{k-1}, \Psi(\bu_h(t_{k-1}));\bff_n, \Dt))\|\\
    &\leq Lip(\Psi')\|\bu_n(t_{k}) - \Psi(\bu_h(t_{k-1})) - \Dt \Phi(t_{k-1}, \Psi(\bu_h(t_{k-1}));\bff_n, \Dt)\|\\
    & = Lip(\Psi')\|\bu_n(t_{k}) - \bu_n(t_{k-1}) - \Dt \Phi(t_{k-1}, \bu_n(t_{k-1});\bff_n, \Dt)\|\\
    &= Lip(\Psi')\|\bveps_n^{k}\|.
\end{align*}
Thus, letting $\btau_*^{k}(\Dt) = \bveps_*^{k}/\Dt$ be the LTEs, with $*$ either $h$ or $n$, the following holds
\begin{equation*}
    \|\btau_h^{k}(\Dt)\| \leq Lip(\Psi')\|\btau_n^{k}(\Dt)\|, \qquad 1\leq k \leq N_\Dt,
\end{equation*} which, in terms of the global truncation errors, reads as
\begin{equation*}
    \tau_h(\Dt) \leq Lip(\Psi')\tau_n(\Dt).
\end{equation*} Thus, assuming consistency of the latent scheme, i.e. $\tau_n(\Dt)\rightarrow 0$, as $\Dt\rightarrow 0$, the consistency of the discretized LDM scheme is guaranteed.

\paragraph{Zero-stability.} Assuming to employ a \textit{zero-stable} explicit one-step numerical method for the approximation of the latent IVP, we aim to characterize how such property, belonging to the latent state approximation $\bu_n^k$, affects the reconstructed high-fidelity approximation $\tbu_h^k$. Thus, we proceed by considering the perturbed $\dldm$ formulation, defined as
\begin{equation}
    \begin{cases}
        \tbz_h^k = \Psi'(\bz_n^k), & k = 1,\ldots,N_\Dt,\\
        \bz_n^{k} = \bz_n^{k-1}+\Dt\Big(\Phi(t_{k-1},\bz_n^{k-1};\bff_n, \Dt) + \bdelta_n^{k}\Big), & k = 1,\ldots,N_\Dt,\\
        \bz_n^0 = \Psi(\bu_{0,h} + \bdelta_h^0),\\
    \end{cases}
\end{equation}
with $(\bdelta_h^0,\bdelta_n^k)\in\mathbb{R}^{N_h}\times \mathbb{R}^{n}$ suitable perturbations. In particular, we aim to show that $\exists \Dt^*>0, \ \exists \tilde{C}>0$ such that $\forall \Dt \leq \Dt^*$, $\forall \varepsilon>0$ small enough, if $\|\bdelta_h^0\|\leq \varepsilon, \ \|\bdelta_n^k\|\leq \varepsilon, \ 1\leq k\leq N_\Dt$, then 
\begin{equation}
    \|\tbz^k_h - \tbu^k_h\| \leq \tilde{C} \varepsilon, \qquad 1\leq k\leq N_\Dt.
    \label{eq:zero-stab-h}
\end{equation}
Assuming to employ a zero-stable method for the latent IVP problem approximation, a result of the form \eqref{eq:zero-stab-h} holds for the discretized latent state, i.e.
\begin{equation*}
    \|\bz^k_n - \bu^k_n\| \leq C \varepsilon, \qquad 1\leq k\leq N_\Dt,
\end{equation*}
thus, exploiting the Lipschitz continuity of $\Psi'$, we obtain
\begin{equation*}
    \|\tbz_h^k-\tbu_h^k\| = \|\Psi'(\bz_n^k)-\Psi'(\bu_n^k)\| \leq  Lip(\Psi') \|\bz_n^k-\bu_n^k\| \leq Lip(\Psi')C\varepsilon.
\end{equation*}
Then, we derive $C$, assessing its independence from $\varepsilon$, under the assumption of Lipschitz continuity of the increment function $\Phi$ with respect to the second variable, of constant\footnote{Independent of the discretization step $\Dt$ and of the nodes $t_k$.} $Lip(\Phi)$, being a sufficient condition guaranteeing zero-stability of the employed method \cite{Quarteroni2007}. 
Indeed, letting $\bxi^k_n = \bz_n^k-\bu_n^k$, we have that
\begin{equation*}
    \bxi^{k}_n = \bxi^{k-1}_n + \Dt \Big( \Phi(t_{k-1},\bz_n^{k-1};\bff_n,\Dt) - \Phi(t_{k-1},\bu_n^{k-1};\bff_n,\Dt) + \bdelta_n^k\Big),
\end{equation*}
then, summing over $k$, with $m=1,\ldots,N_\Dt$, we obtain
\begin{equation*}
    \bxi^{m}_n = \bxi^0_n + \Dt \sum_{k=1}^{m} \bdelta_n^{k} + \Dt \sum_{k=1}^{m}\Big(\Phi(t_{k-1},\bz_n^{k-1};\bff_n,\Dt) - \Phi(t_{k-1},\bu_n^{k-1};\bff_n,\Dt)\Big).
\end{equation*}
Thus, thanks to the Lipschitz continuity of  $\Phi$ and $\Psi$, it follows that
\begin{align}
    \|\bxi_n^m\| &\leq \|\bxi_n^0\| + \Dt \sum_{k=1}^{m} \|\bdelta_n^{k}\| + \Dt Lip(\Phi) \sum_{k=1}^{m}\|\bxi_n^{k-1}\|, &1\leq m \leq N_\Dt, \nonumber \\ &\leq Lip(\Psi) \|\bdelta_h^0\| + \Dt \sum_{k=1}^{m} \|\bdelta_n^{k}\| + \Dt Lip(\Phi) \sum_{k=1}^{m}\|\bxi_n^{k-1}\|, &1\leq m \leq N_\Dt, \nonumber \\
    &\leq (Lip(\Psi)+m\Dt)\varepsilon e^{m\Dt Lip(\Phi)}, &1\leq m \leq N_\Dt, \label{eq:discrete-gronwall} \smallskip \\
    &\leq \underbrace{(Lip(\Psi)+ |T-t_0|) e^{Lip(\Phi)|T-t_0|}}_{C} \varepsilon \nonumber
\end{align}
with \eqref{eq:discrete-gronwall} following from the discrete Gronwall lemma. Thus, setting 
\[
\tilde{C} = Lip(\Psi')(Lip(\Psi)+ |T-t_0|) e^{Lip(\Phi)|T-t_0|}
\]
in \eqref{eq:zero-stab-h}, we can conclude that the zero-stability property holds for the $\dldm$ scheme.

\paragraph{Error decomposition.}
Since the practical implementation of LDMs is obtained by means of its discrete formulation $\dldm$, it is necessary to extend the previously derived error bounds by including the error source related to the numerical approximation employed within the latent space. In particular, by referring to \eqref{error_decomposition}, it holds that 
\begin{align*}
    \|\bu_h(t_k;\bmu) - \tbu_h^k(\bmu)\| &= \|\bu_h(t_k;\bmu) - \tbu_h(t_k;\bmu) + \tbu_h(t_k;\bmu)  - \tbu_h^k(\bmu)\| \\
    &\leq \|\bu_h(t_k;\bmu) - \tbu_h(t_k;\bmu)\| + \|\tbu_h(t_k;\bmu) -  \tbu_h^k(\bmu)\|\\
    &\leq C_1 + Lip(\Psi')C\Dt^p.
\end{align*}
Thus, the final form of the error decomposition, accounting for { \em(i)} the continuous-time approximation, namely, the upper bound terms of \eqref{error_decomposition} (collected in $C_1$), and { \em(ii)} the numerical error source ($C_2 \Dt^p$), takes the form
\begin{equation}
    \|\bu_h(t_k,\bmu)-\tbu_h^k(\bmu)\| \leq C_1 + C_2 \Dt^p,
    \label{eq:error_decomposition}
\end{equation}
as depicted in Figure \ref{fig:error-bounds}.
\begin{figure}[h]
    \center
    \includegraphics[width=0.52\textwidth]{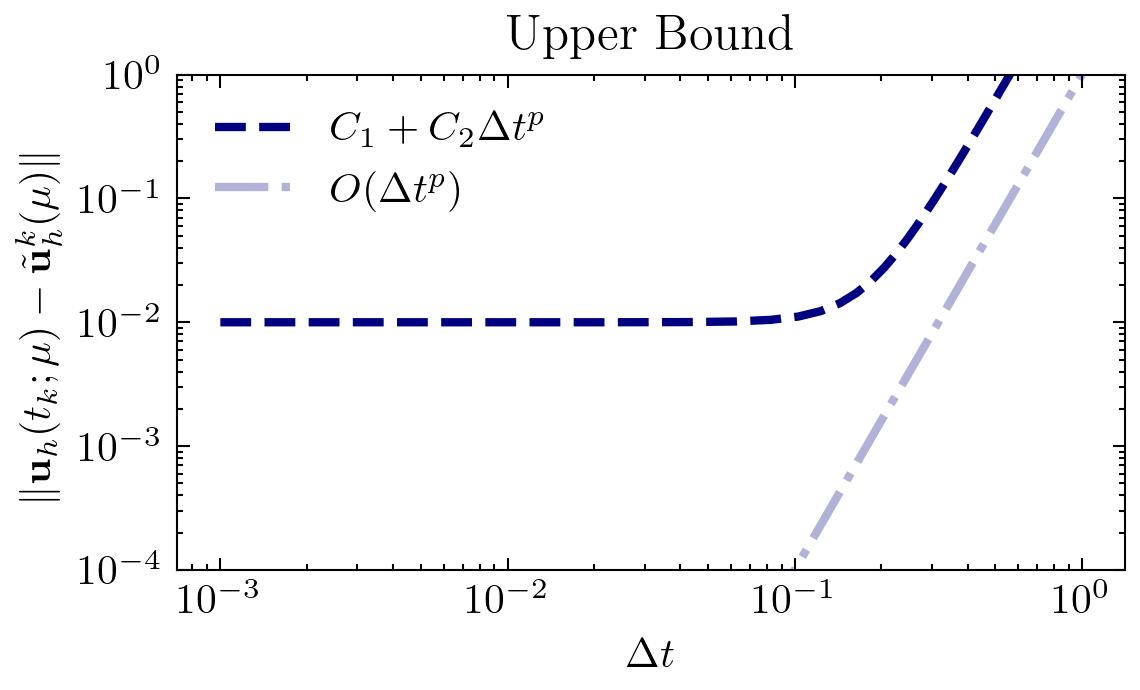}
    \caption{\textit{Error decomposition.} Illustration of the upper bound referring to the error decomposition formula \eqref{eq:error_decomposition} in the discrete case, for the full-order state approximation $\tbu_h^k$ provided by the $\dldm$ scheme.}
    \label{fig:error-bounds}
\end{figure}

\subsection{Learnable latent dynamics models}
\label{section:LLDM}

The LDP formulation described so far provides a well-structured framework that can be readily cast within a learning context, where maps equipped with learnable parameters can approximate the components involved in the LDM definition. In particular, we consider the problem of learning a LDM in a time-discrete setting, i.e. a $\dldm$. 
We proceed by considering {\em (i)} a discretization of the time interval $[t_0,T]$, with a discretization step $\Dt$, resulting in a sequence of $N_\Dt+1$ time instances $\{t_k\}_{k=0}^{N_\Dt}$, with $t_k = t_0+k\Dt$, and {\em (ii)} a proper sampling criterion over the compact set of parameters $\mathcal{P}\subset \mathbb{R}^{n_\mu}$, leading to a discrete set of $N_\mu$ sampled parameters instances $\{\bmu^j\}_{j=1}^{N_\mu}$. Then, a collection of high-fidelity FOM snapshots $\bu_h(t_k;\bmu^j)$, sampled in correspondence of $\{t_k\}_{k=0}^{N_\Dt}\times\{\bmu^j\}_{j=1}^{N_\mu}$, is considered. 
So, the \textit{learnable} $\dldm_\theta$ approximating the evolution of a high-fidelity trajectory $\bu_h(t;\bmu)$ by means of $\hbu_h^k(\bmu)$, for a given parameter instance $\bmu \in \mathcal{P}$, reads as
\begin{equation}
    \begin{cases}
    \hbu_h^{k}(\bmu) = \Psi'_\theta(\hbu_n^{k}(\bmu)), & k = 1,\ldots,N_\Dt,\\
    \hbu_n^{k}(\bmu) = \hbu_n^{k-1}(\bmu) + \Dt \Phi(t_{k-1}, \hbu_n^{k-1}(\bmu);\Dt,\bff_{n,\theta},\bmu), & k = 1,\ldots,N_\Dt,\\    \hbu_n^0(\bmu) = \Psi_\theta(\bu_{0,h}(\bmu)),\\
    \end{cases}
    \label{eq:learnable_DLDM}
\end{equation}
with $\btheta = (\btheta_\Psi,\btheta_{\Psi'},\btheta_{\bff_n}) \in\Theta$ denoting the vector of learnable parameters, in a suitable space $\Theta$, associated to the triple of learnable functions $(\Psi_\theta,\Psi'_\theta,\bff_{n,\theta})$, satisfying the following minimization problem
\begin{equation}
 \min_{\theta\in \Theta}\mathcal{L}(\theta), \qquad \mathcal{L}(\theta) =
    \frac{1}{N_\mu N_\Dt}\sum_{j=1}^{N_\mu} \sum_{k=1}^{N_\Dt} \| \bu_h(t_k;\bmu^j) - \hbu^k_h(\bmu^j) \|^2.
   \label{eq:dldp}
\end{equation}
 Specifically, in Algorithm \ref{alg:pseudo-training} we report a simplified version of the training procedure, highlighting the main steps to be carried out during one iteration of the $\dldm_\theta$ training phase, for approximating a single FOM trajectory $\bu_h(t_k;\bmu)$, sampled at temporal nodes $t_k$, for a fixed parameter instance $\bmu \in \mathcal{P}$. Hereafter, we explore the problem of learning a time-continuous approximation, by analyzing the approximations involved in the introduced learnable framework.

\begin{algorithm}[h]
\caption{$\dldm_\theta$ training algorithm (single trajectory)}
\begin{algorithmic}[1]
\Require FOM trajectory $\{\bu_h(t_k;\bmu)\}_{k=0}^{N_\Dt}$, timesteps $\{t_k\}_{k=0}^{N_\Dt}$, parameter instance $\bmu\in\mathcal{P}$
\State Project initial state $\hbu^0_n(\bmu) \leftarrow \Psi_\theta(\bu_h^0(\bmu)) \in \mathbb{R}^n$ \;
\For{$k = 1 : N_\Dt$}
    \State Latent Runge-Kutta step $\hbu_{n}^{k} \leftarrow \hbu_{n}^{k-1} + \Dt \Phi(t_{k-1}, \hbu_{n}^{k-1}(\bmu); \Dt, \bff_{n,\theta}, \bmu)$\;
    \State Reconstruct evolved state $\hbu_{h}^{k}(\bmu) \leftarrow \Psi'_\theta(\hbu_{n}^{k}(\bmu)) \in \mathbb{R}^{N_h}$\;
    \EndFor
\State Take optimization step on $\mathcal{L}(\theta)$ in \eqref{eq:dldp}.
\end{algorithmic}
\label{alg:pseudo-training}
\end{algorithm}

\subsubsection{Approximation results}
On the basis of the $\dldm_\theta$ framework just introduced and the perfect embedding assumption \ref{perfect_embedding}, we aim at providing one of the main results of this work, which is contained in the Theorem \ref{error_estimate}, and is endowed with a constructive proof founded on the approximation results of \cite{guhring2021appx}. In particular, we aim at proving that, for each small enough $\Dt$, the approximation error between the original FOM solution and the ROM approximation, provided by the $\dldm_\theta$, is bounded. Additionally, we show that the validity of this error bound directly ensures the satisfaction of the time-continuous approximation property by the $\dldm_\theta$ formulation.

\begin{theorem}
\label{error_estimate}
Let $\mathcal{T}_{\Dt}$ be a general discretization of $[t_{0}, T]$ with a discretization step $\Dt$, resulting in a sequence of $N_{\Dt}$ time instances $\mathcal{T}_{\Dt} = \{t_k\}_{k=1}^{N_{\Dt}}$. Then, under the perfect embedding assumption, for any tolerance $\varepsilon > 0$, there exist $n > 0$, 
\begin{itemize}
    \item a \textit{dynamics network} $(t,\bmu) \mapsto \mathbf{F}_{n, \theta}(t,\bmu)$, having at most $O(n\varepsilon^{-(n_\mu+1)})$ active weights,
    \item a decoder $\Psi'_{n,\theta}: \mathbb{R}^{n} \rightarrow \mathbb{R}^{N_h}$, having at most $O(N_h\varepsilon^{-n})$ active weights,
    \item an encoder $\Psi_{n,\theta}: \mathbb{R}^{N_h} \rightarrow \mathbb{R}^{n}$, having at most $O(n\varepsilon^{-N_h})$ active weights,
\end{itemize}
and  $\Dt^* = \Dt^*(\mathbf{F}_{n, \theta})$,
such that
\begin{equation}
    \label{eq:error_estimate}
    \sup_{k \in \{1, \ldots, N_{\Dt}\}} \| \mathbf{u}_h(t_k; \boldsymbol{\mu}) - \hat{\mathbf{u}}_h^k(\boldsymbol{\mu}) \| \le \varepsilon, \quad \forall \Delta
    t \le \Dt^*, \quad \forall \boldsymbol{\mu} \in \mathcal{P},
\end{equation}
where 
\begin{equation}
\label{eq:ldm_theta_explicit}
    \hat{\bu}_{h}^{k}(\bmu) := \Psi'_{n,\theta}(\bu_{n,\theta}^k(\bmu)) = \Psi'_{n,\theta}\bigg(\Psi_{n,\theta}(\bu_{0,h}(\bmu)) + \Dt \sum_{j=0}^{k} w_j \mathbf{F}_{n,\theta}(s_j;\bmu)\bigg), \qquad \forall k=1,\ldots,N_{\Dt},
\end{equation}
with $(w_j,s_j)_{j=1}^{N_{\Dt}}$ pairs of quadrature weights and nodes.
\end{theorem}

\begin{proof}
We let $\varepsilon > 0$. Moreover, without loss of generality, here we assume $t_0 = 0$. The proof is divided into 5 parts.
\begin{itemize}
    \item[(i)] \textit{Definition of the exact latent dynamics} \\
    By employing the definition of LDM and the perfect embedding assumption, we can state that there exists $n^* \ll N_h$ such that, for any $n \ge n^*$ there exists $\Psi'_n, \Psi_n, \bff_n$, such that
    \begin{equation*}
    \begin{cases}
         \bu_h(t;\bmu) = \tilde{\bu}_h(t;\bmu) = \Psi'_n(\bu_n(t;\bmu)) , &t \in (0,T],\\ 
         \dot{\bu}_{n}(t;\bmu) = \bff_n(t, \bu_n(t;\bmu); \bmu), &t \in (0,T], \\
         \bu_n(0;\bmu) = \Psi_n(\bu_{0,h}(\bmu)),&
    \end{cases}
    \end{equation*}
    for any $\bmu \in \mathcal{P}$. Thus, calling $\mathbf{F}_n(t;\bmu) = \bff_n(t, \bu_n(t;\bmu); \bmu)$, for any $t \in (0,T]$, we define explicitly
    \begin{equation*}
        \bu_{n}(t;\bmu) = \Psi_{n}(\bu_{0,h}(\bmu)) + \int_{0}^{t} \mathbf{F}_{n}(s;\bmu) ds.
    \end{equation*} 
    \item[(ii)] \textit{Definition of the neural network architectures} \\
    Since $\mathbf{F}_n(t;\bmu)\in W^{1,\infty}([0,T]\times \mathcal{P};\mathbb{R}^n)$, thanks to G\"uhring-Raslan Theorem \cite{guhring2021appx} there exists a neural network $(t,\bmu) \mapsto \mathbf{F}_{n,\theta}(t;\bmu)$ consisting of at most $O(n\varepsilon^{-(n_\mu+1)})$ active weights such that
    \begin{equation*}
    \|\mathbf{F}_n - \mathbf{F}_{n,\theta}\|_{L^\infty([0,T] \times \mathcal{P}; \mathbb{R}^n)} \le (8T)^{-1}\varepsilon.
    \end{equation*}
    Moreover, we highlight that the manifold $\mathcal{S}_{0} = \{\bu_{0,h}(\bmu) : \bmu \in \mathcal{P}\}$ is such that $diam(\mathcal{S}_{0})$ is bounded since $\mathcal{P}$ is compact, following from the Lipschitzness of $\bu_h$ with respect to $\bmu$, which is satisfied under the requirement $\bu_h \in C^1([0,T]\times\mathcal{P};\mathbb{R}^{N_h})$. Thus, owing to $\Psi_{n} \in C^1(\mathbb{R}^{N_h}; \mathbb{R}^n)$ , it is possible to construct a neural network architecture $\Psi_{n,\theta}: \mathcal{S}_0 \rightarrow \mathbb{R}^n$ having $O(n \varepsilon^{-N_h})$ active weights, such that 
    \begin{equation*}
    \|\Psi_{n}(\bu_{0,h}(\bmu)) - \Psi_{n,\theta}(\bu_{0,h}(\bmu))\|_{L^\infty( \mathcal{P}; \mathbb{R}^n)} \le 8^{-1}\varepsilon.
    \end{equation*}
    Then, it is possible to prove that $\mathcal{S}_n = \{\bu_n(t;\bmu) : t \in [0,T],\ \bmu \in \mathcal{P}\}$ has bounded diameter. Indeed, we emphasize that $[0,T] \times \mathcal{P}$ is compact and by definition of $\bu_n$, since $\Psi_n$  and $\mathbf{F}_n$ are bounded with respect to $t$ and Lipschitz-continuous with respect to $\bmu$, we have, $\forall t_1,t_2 \in [0, T]$ and $\forall \bmu_1, \bmu_2 \in \mathcal{P}\subset\mathbb{R}^{n_\mu}$,
    \begin{equation*}
    \begin{aligned}
        &\|\bu_n(t_1;\bmu_1) - \bu_n(t_2;\bmu_2)\| \le \|\bu_n(t_1;\bmu_1) - \bu_n(t_2;\bmu_1)\| + \|\bu_n(t_2;\bmu_1) - \bu_n(t_2;\bmu_2)\| \\
        &\le \|\mathbf{F}_n(t;\bmu)\|_{L^\infty([0,T] \times \mathcal{P}; \mathbb{R}^n)} |t_1 - t_2| + Lip(\Psi_n) \|\bmu_1 - \bmu_2\| + \int_{0}^{t_2} \|\mathbf{F}_n(s;\bmu_1) - \mathbf{F}_{n}(s;\bmu_2) \|ds\\
        &\le \|\mathbf{F}_n(t;\bmu)\|_{L^\infty([0,T] \times \mathcal{P}; \mathbb{R}^n)} |t_1 - t_2| + Lip(\Psi_n) \|\bmu_1 - \bmu_2\| + T Lip(\mathbf{F}_n) \|\bmu_1 - \bmu_2\|.
    \end{aligned}
    \end{equation*}
    Thus, applying G\"uhring-Raslan Theorem, it is possible to construct a neural network $\Psi'_{n,\theta}:\mathcal{S}_n \rightarrow \mathbb{R}^{N_h}$ having at most $O(N_h\varepsilon^{-n})$ active weights such that
    \begin{equation*}
        \|\Psi'_{n}(\bu_n(t;\bmu)) - \Psi'_{n,\theta}(\bu_n(t;\bmu)) \|_{L^\infty([0,T] \times \mathcal{P}; \mathbb{R}^{N_h})} \le 2^{-1}\varepsilon.
    \end{equation*}
    \item[(iii)] \textit{Approximation of the latent dynamics at a time-continuous level}\\
    Hence, we are now able to define the explicit form of the $\textnormal{LDM}_\theta$ at a continuous level, namely
    \begin{equation*}
        \bu_{n,\theta}(t;\bmu) = \Psi_{n,\theta}(\bu_{0,h}(\bmu)) + \int_{0}^{t} \mathbf{F}_{n,\theta}(s;\bmu) ds.
    \end{equation*}
    We remark that by suitably redefining $\bu_{n,\theta}(t;\bmu)$ and $\bu_{n}(t;\bmu)$, without loss of generality, we can assume that $Lip(\Psi'_{n,\theta}) = 1$.
    Then, it is straightforward to verify that
    \begin{equation*}
    \begin{aligned}
        E_{CONT} :&=  \|\bu_n -  \bu_{n,\theta} \|_{L^\infty([0,T] \times \mathcal{P}; \mathbb{R}^n)} \\
        & = \sup_{t \in [0,T]} \bigg\|\Psi_n(\bu_{0,h}(\bmu)) + \int_{0}^{t}\mathbf{F}_n(s;\bmu)ds -  \Psi_{n,\theta}(\bu_{0,h}(\bmu)) + \int_{0}^{t}\mathbf{F}_{n,\theta}(s;\bmu)ds \bigg\|_{L^\infty(\mathcal{P}; \mathbb{R}^n)} \\
        &\le \|\Psi_n(\bu_{0,h}(\bmu)) - \Psi_{n,\theta}(\bu_{0,h}(\bmu))\|_{L^\infty(\mathcal{P}; \mathbb{R}^n)} +\\
        & \hspace{3cm} + \sup_{t \in [0,T]} \int_{0}^{t}\|\mathbf{F}_n(s;\bmu) - \mathbf{F}_{n,\theta}(s;\bmu)\|_{L^\infty(\mathcal{P}; \mathbb{R}^n)}ds \\
        &\le 8^{-1}\varepsilon + T\|\mathbf{F}_n - \mathbf{F}_{n,\theta}\|_{L^\infty([0,T] \times \mathcal{P}; \mathbb{R}^n)} \\
        & \le 8^{-1}\varepsilon + T(8T)^{-1}\varepsilon = 4^{-1}\varepsilon.
    \end{aligned}
    \end{equation*}
    \item[(iv)] \textit{Approximation of the latent dynamics at a time-discrete level} \\
    Let $\Dt > 0$, then $\mathcal{T}_{\Dt}$ is a uniform discretization of $[0,T]$.
    We set $N_{\Dt} = \textnormal{card}(\mathcal{T}_{\Dt})$ and we define, for any $k=1,\ldots,N_{\Dt}$,
    \begin{equation*}
        \bu_{n,\theta}^k(\bmu) := \Psi_{n,\theta}(\bu_{0,h}(\bmu)) + \Dt \sum_{j=0}^{k} w_j \mathbf{F}_{n,\theta}(s_j;\bmu),
    \end{equation*}
    where $(w_j,s_j)_{j=1}^{N_{\Dt}}$ are the quadrature weights and nodes that implicitly depend on $\Dt$.
    We define the quadrature error as
    \begin{equation*}
         E_{QUAD}(\Dt) := \sup_{k \in \{1, \ldots, N_{\Dt}\}} \bigg\|\int_{0}^{t_k} \mathbf{F}_{n,\theta}(s;\bmu)ds - \Dt \sum_{j=0}^{k} w_j \mathbf{F}_{n,\theta}(s_j;\bmu) \bigg\|_{L^\infty(\mathcal{P}; \mathbb{R}^n)}. 
    \end{equation*}
   If the scheme is convergent, then there exists a monotonically increasing function $\Dt \mapsto h(\Dt)$ such that $h(\Dt) \rightarrow 0$, as $\Dt \rightarrow 0$, and a constant $C > 0$ for which $E_{QUAD}(\Dt) \le C h(\Dt)$. Observe that $\Dt \mapsto h(\Dt)$ implicitly depends on $\mathbf{F}_{n,\theta}$. Then, there exists $\Dt^* = \Dt^*(\varepsilon; \mathbf{F}_{n,\theta})>0$ such that for any $\Dt \le \Dt^*$ it holds that
   \begin{equation*}
       E_{QUAD}(\Dt) \le C h(\Dt) \le C h(\Dt^*) \le 4^{-1}\varepsilon.
   \end{equation*} 
   Thus, hereon, we let $\Dt \le \Dt^*$. Being $\mathcal{T}_{\Dt} \subset [0,T]$ and by employing the triangular inequality we obtain
    \begin{equation*}
    \begin{aligned}
        E_{DISCR} :&= \sup_{k \in \{1, \ldots, N_{\Dt}\}} \|\bu_n(t_k) - \bu_{n,\theta}^k\|_{L^\infty(\mathcal{P}; \mathbb{R}^n)} \\
        & \le \sup_{k \in \{1, \ldots, N_{\Dt}\}} \bigg[\|\bu_n(t_k) - \bu_{n,\theta}(t_k)\|_{L^\infty(\mathcal{P}; \mathbb{R}^n)} + \|\bu_{n,\theta}(t_k) - \bu_{n,\theta}^k\|_{L^\infty(\mathcal{P}; \mathbb{R}^n)} \bigg] \\
        & \le\|\bu_n -  \bu_{n,\theta} \|_{L^\infty([0,T] \times \mathcal{P}; \mathbb{R}^n)} + \sup_{k \in \{1, \ldots, N_{\Dt}\}} \|\bu_{n,\theta}(t_k) - \bu_{n,\theta}^k\|_{L^\infty(\mathcal{P}; \mathbb{R}^n)}  \\
        & = E_{CONT} + \sup_{k \in \{1, \ldots, N_{\Dt}\}} \|\bu_{n,\theta}(t_k) - \bu_{n,\theta}^k\|_{L^\infty( \mathcal{P}; \mathbb{R}^n)} \\
        & = E_{CONT} + \sup_{k \in \{1, \ldots, N_{\Dt}\}} \bigg\|\int_{0}^{t_k} \mathbf{F}_{n,\theta}(s;\bmu)ds - \Delta  t \sum_{j=0}^{k} w_j \mathbf{F}_{n,\theta}(s_j;\bmu) \bigg\|_{L^\infty(\mathcal{P}; \mathbb{R}^n)} \\
        & = E_{CONT} + E_{QUAD}(\Dt) \le 2^{-1}\varepsilon.
    \end{aligned}
    \end{equation*}
    \item[(v)] \textit{Final error bound} \\
    Defining
    \begin{equation*}
        \hat{\bu}_{h}^{k}(\bmu) := \Psi'_{n,\theta}(\bu_{n,\theta}^k(\bmu)) = \Psi'_{n,\theta}\bigg(\Psi_{n,\theta}(\bu_{0,h}(\bmu)) + \Dt \sum_{j=0}^{k} w_j \mathbf{F}_{n,\theta}(s_j;\bmu)\bigg), 
    \end{equation*}
    through the triangular inequality, we obtain
    \begin{equation*}
        \begin{aligned}
        \sup_{k \in \{1, \ldots, N_{\Dt}\}} \|\bu_h(t_k; \bmu) & - \hat{\bu}_h^k(\bmu) \|_{L^\infty(\mathcal{P}; \mathbb{R}^{N_h})} \le \\
        & \le \sup_{k \in \{1, \ldots, N_{\Dt}\}} \bigg[\|\Psi'_{n}(\bu_{n}(t_k; \bmu)) - \Psi'_{n,\theta}(\bu_{n}(t_k; \bmu))\|_{L^\infty(\mathcal{P}; \mathbb{R}^{N_h})} \\
        & + \|\Psi'_{n,\theta}(\bu_{n}(t_k; \bmu)) - \Psi'_{n,\theta}(\bu_{n,\theta}^k(\bmu)) \|_{L^\infty(\mathcal{P}; \mathbb{R}^{N_h})} \bigg] \\
        & \le E_{DEC} + E_{LAT}.
        \end{aligned}
    \end{equation*}
    Then, we can bound $E_{DEC}$ with
    \begin{equation*}
    \begin{aligned}
        E_{DEC} : &= \sup_{k \in \{1, \ldots, N_{\Dt}\}} \|\Psi'_{n}(\bu_{n}(t_k; \bmu)) - \Psi'_{n,\theta}(\bu_{n}(t_k; \bmu))\|_{L^\infty(\mathcal{P}; \mathbb{R}^{N_h})}  \\
        & \le\|\Psi'_{n}(\bu_n(t;\bmu)) - \Psi'_{n,\theta}(\bu_n(t;\bmu))\|_{L^\infty([0,T] \times \mathcal{P}; \mathbb{R}^{N_h})} \le 2^{-1}\varepsilon.
    \end{aligned}
    \end{equation*}
    Moreover, thanks to the $\Psi'_{n,\theta}$ being 1$-Lip$, we obtain
    \begin{equation*}
    \begin{aligned}
        E_{LAT} :&= \sup_{k \in \{1, \ldots, N_{\Dt}\}} \|\Psi'_{n,\theta}(\bu_{n}(t_k; \bmu)) - \Psi'_{n,\theta}(\bu_{n,\theta}^k(\bmu)) \|_{L^\infty(\mathcal{P}; \mathbb{R}^{N_h})} \\
        &\le \sup_{k \in \{1, \ldots, N_{\Dt}\}} \|\bu_{n}(t_k; \bmu) - \bu_{n,\theta}^k (\bmu)\|_{L^\infty(\mathcal{P}; \mathbb{R}^n)} = E_{DISCR} \le 2^{-1} \varepsilon.
    \end{aligned}
    \end{equation*}
    Finally, we can conclude the proof, obtaining that 
    \begin{equation*}
    \begin{aligned}
         \sup_{k \in \{1, \ldots, N_{\Dt}\}} \|\bu_h(t_k; \bmu) - \hat{\bu}_h^k(\bmu) \|_{L^\infty(\mathcal{P}; \mathbb{R}^{N_h})} &\leq  E_{DEC} + E_{LAT} \le \varepsilon.
    \end{aligned}
    \end{equation*}
\end{itemize}
\end{proof}

\begin{figure}
    \centering
    \includegraphics[width=\textwidth]{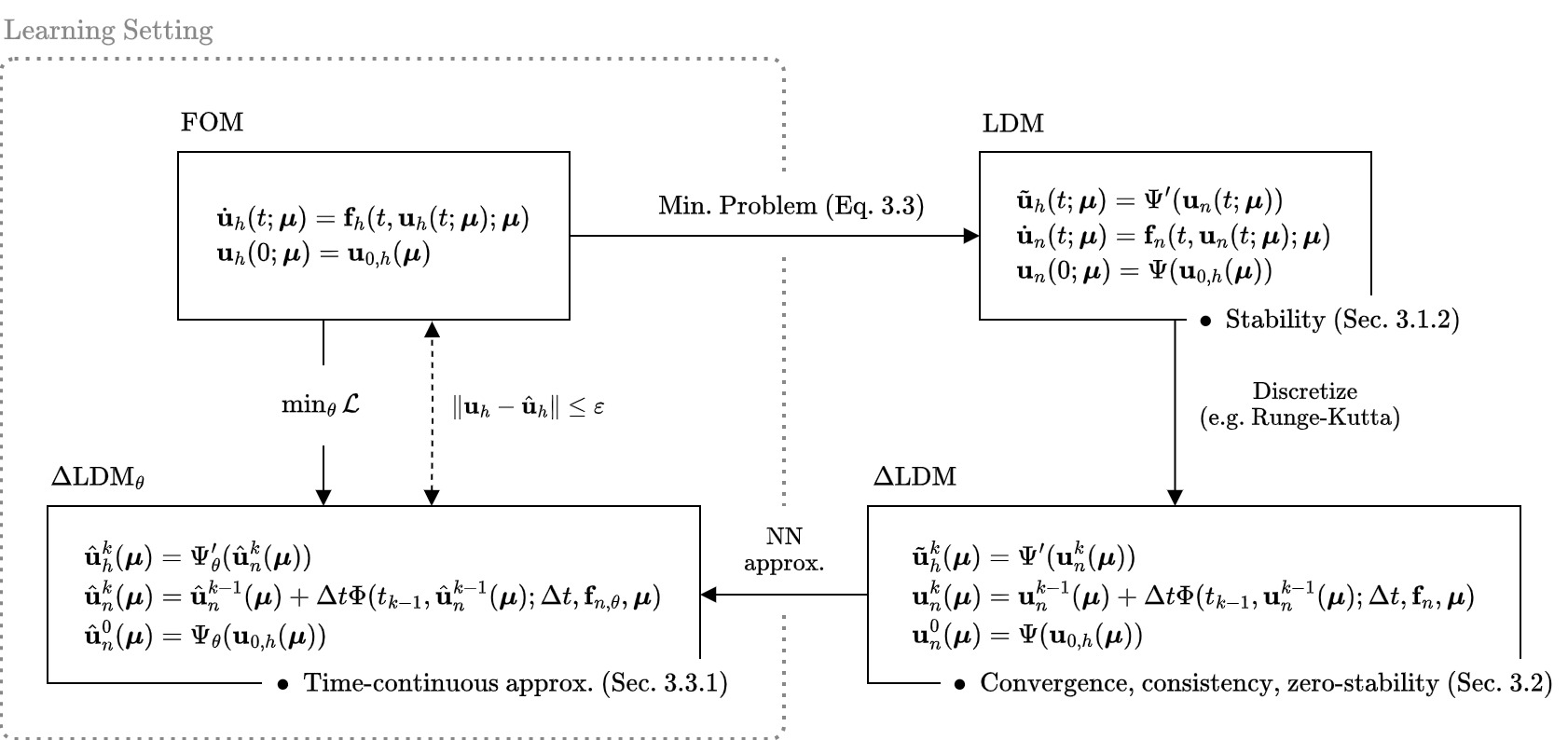}
    \caption{\textit{Framework overview.} The diagram summarizes the proposed framework, considering the learnable setting ($\dldm_\theta$), the time-continuous (LDM) and time-discrete  ($\dldm$) approximations, together with the respective main properties.}
    \label{fig:ldm_block_summary}
\end{figure}

\paragraph{Time-continuity in a learnable context.} 
Having introduced the $\dldm_\theta$, we aim at characterizing, within such learnable setting, the notion of \textit{time-continuous approximation} defined for a generic ROM in Section \ref{sec-2}. In particular, the dynamical nature of the $\dldm_\theta$, characterized by a \textit{continuous-time} inductive bias represented by the (discretized) latent IVP, allows us to query the learned solution with arbitrary precision by using suitably fine temporal discretizations that, in principle, could be different from the one employed during the learning procedure. 
In such a context, the time-continuous approximation property follows from the approximation result given in Theorem \ref{error_estimate}, and from the error bound \eqref{eq:error_decomposition} which leverages the convergence property of LDMs in discrete settings. Specifically, we aim at showing that 
\begin{equation*}
\label{eq:time-continuity-final-remark}
\sup_{t_k \in \mathcal{T}_{\Dt}}\|\bu_h(t_k;\bmu) - \hbu_h^k(\bmu)\| \overset{(i)}{\leq} C \sup_{t_k \in \mathcal{T}_{\Dt_{train}}}\|\bu_h(t_k;\bmu) - \hbu_h^k(\bmu)\| \overset{(ii)}{\leq} C \varepsilon
\end{equation*}
holds $\forall \Dt \leq \Dt_{train}$, independently of $\bmu \in \mathcal{P}$. To do that, we immediately notice that {\em (ii)} follows from Theorem 1 assuming $\Dt_{train} \leq \Dt^*$.
Then, we focus on proving {\em (i)} by characterizing the constant $C$ appearing in Definition \ref{def:timecont}. Thus, by following the same reasoning as in the derivation of the convergence properties of Section \ref{section:DLDM}, 
\begin{align*}
    E(\Delta t) := \sup_{t_k \in \mathcal{T}_{\Dt}}\|\bu_h(t_k;\bmu) - \hbu_h^k(\bmu)\|_{L^\infty(\mathcal{P})}  &\leq \|\mathbf{u}_h - \hat{\mathbf{u}}_h\|_{L^\infty([0,T]\times \mathcal{P})} + \sup_{t_k\in\mathcal{T}_{\Delta t}}\|\hat{\mathbf{u}}(t_k;\bmu) - \hat{\mathbf{u}}^k_h(\bmu)\|_{L^\infty(\mathcal{P})},\\
    &\leq E(0) + C_2 \Delta t^p.
\end{align*}
We remark that this draws a parallel with the error decomposition formula 
\eqref{eq:error_decomposition}. 
Under the reasonable assumption of bounded $E(0)=\|\bu_h(\bmu) - \hbu_h(\bmu)\|_{L^\infty([0,T]\times\mathcal{P})}$, it follows that $\forall \Dt_{train}>0, \ \exists C=C(\Dt_{train})>0$ such that
\begin{equation*}
     E(\Dt) \leq E(0) + C_2\Dt^p_{train} = CE(\Dt_{train}), \quad \forall \Dt \leq \Dt_{train}.
\end{equation*}
Thus, taking $\Dt_{train}\leq\Dt^*$ as in Theorem 1, the time-continuity property follows from 
\begin{equation*}
    E(\Dt) \leq C E(\Dt_{train}) \leq C\varepsilon, \quad \forall \Dt \leq \Dt_{train},
\end{equation*}
with $C = C(\Dt_{train},T,p) = (E(0) + C_2\Dt^p_{train})E(\Dt_{train})^{-1}$. In particular, this highlights that $C \rightarrow 1$ as the training discretization is refined ($\Dt_{train}\rightarrow 0$), which means that in the limit we are able to preserve the training accuracy over refinements of the testing discretization. Additionally, the dependence of $C$ on $p$ reveals the role of the convergence order of the integration method employed in the $\dldm_{\theta}$ within the context of time-continuity. We note that the dependence of the constant $C$ on the temporal interval length $T$ is implicit through $C_2$.

\begin{remark}
    We remark that Eq.\eqref{eq:ldm_theta_explicit} can be traced back to the original $\dldm_\theta$ formulation \eqref{eq:learnable_DLDM}. Moreover, we notice that the encoder architecture may suffer from the curse of dimensionality, as the number of active weights may scale as $O(\varepsilon^{-N_h})$, which becomes prohibitive as $N_h$ increases. However, the encoder can be replaced by a feedforward neural network $\bmu \mapsto \mathbf{g}(\bmu) \approx \Psi(\bu_{0,h}(\bmu))$ in a DL-ROM fashion \cite{fresca2021comprehensive, brivio2023error}, thus requiring at most $O(n \varepsilon^{-n_\mu})$ weights.
\end{remark}

\section{Deep learning-based LDMs}
\label{section:DLLDM}
In the following section, the architectural choices regarding the design space of the proposed framework, in a DL context, are described. As introduced in Section \ref{section:LLDM}, the aim of a \textit{learnable} LDM is to model the evolution of a parameterized FOM state in a data-driven manner, by means of learnable maps $\Psi_\theta, \Psi'_\theta$ tackling the dimensionality reduction task, and a learnable latent dynamics $\bff_{n,\theta}$ modeling the latent state evolution. In particular, the identification of a ROM state $\bu_n(t;\bmu)$, associated to a full-order parameterized time-dependent state $\bu_h(t;\bmu)$, and the modeling of its time-evolution, is a task of critical importance in the context of PDEs surrogate modeling. 
To this end, a commonly adopted data-driven approach consists of employing: {\em (i)} nonlinear dimensionality reduction strategies, such as \textit{autoencoders} (AEs) \cite{hinton1993autoencoders}, to learn a meaningful latent representation lying on a low-dimensional manifold \cite{lee2020model, Kim_2022}, and {\em (ii)} a proper time-stepping scheme, in order to model the time-evolution of the latent state. Multiple data-driven approaches, combining the concept of dimensionality reduction and latent time-evolution, have been proposed, mainly by relying on autoregressive or recurrent strategies \cite{gonzalez2018deep, Geneva_2020, 10.3934/mine.2023096, Conti_2023} in time-discrete settings. 

A more recent approach consists in relying on continuous-time modeling strategies, via \textit{neural ordinary differential equations} (NODEs) \cite{NEURIPS2018_69386f6b}, to model the evolution of the latent state 
\cite{Lee_2021, Di_Sante_2022,legaard2022constructing, Chen_2022, farenga2022neural, Lazzara2023}. In particular, the family of models that combines AEs and NODEs, specialized in the context of multi-query reduced order modeling by \cite{Lee_2021}, aligns with our theoretical LDM formulation. Specifically, within the introduced $\dldm_\theta$ learnable framework, an AE neural network identifies the composition $\Psi'_\theta \circ \Psi_\theta$, while a parameterized NODE represents the latent ODE $\dot{\bu}_n(t;\bmu) = \bff_{n,\theta}(t,\bu_n(t;\bmu);\bmu)$, within a DL context. 
Thus, building up on such framework, we proceed to describe the architectural choices that underlie the construction of DL-based LDMs.\\
In this respect, we remark that the proposed framework does not put constraints on the underlying neural networks' design. On the other hand, it is clear that the obtained theoretical results are related to the approximation error and do not take into account the optimization error. However, in practice, reducing the contribution of the optimization error is crucial to ensure an acceptable accuracy in the inference phase. To do that, ultimately aiming at enhancing  the convergence of the training procedure, we now introduce suitable architectural choices and we propose an adequate training strategy. In particular, our proposed design aims at enhancing the interpretability of the latent representation through: {\em (i)} an autoencoding strategy capable of retaining high-dimensional spatial features at the latent level, allowing to directly relate the latent and high-fidelity state evolutions, and {\em (ii)} an affinely-parametrized NODE that preserves the dynamical nature of the FOM solution in the latent space.

\subsection{Spatially-coherent autoencoding}
\label{sec:AE}
Recent approaches address the dimensionality reduction of a full-order state $\bu_h(t;\bmu)\in \mathbb{R}^{N_h}$ in a data-driven manner, by relying on AEs within an unsupervised learning setting. Specifically, AEs based on convolutional architectures are often adopted, due to their ability to leverage spatial information and their sparse nature, which allows them to handle high-dimensional data effectively while maintaining a reasonable parameter count \cite{lee2020model, fresca2021comprehensive, fresca2022pod, romor2023non}. 

In particular, we proceed by considering the full-order solution $\bu_h(t;\bmu)\in \mathbb{R}^{N_h}$ of \eqref{eq:FOM}, obtained by means of a spatial-discretization of the associated PDE, via traditional high-fidelity techniques, such as the Galerkin-finite element method. These methods entail the introduction of a, possibly unstructured, computational grid $\Omega_h$, serving as a discretization of the bounded domain $\Omega \subset \mathbb{R}^d$, with $h$ the spatial discretization parameter, representing the maximum diameter of the grid elements. The choice of $h$ influences the granularity of the discretization, that is, smaller values result in finer resolutions, leading to higher FOM dimensions $N_h$, and thus to increased computational costs. In traditional high-fidelity techniques, the solution is represented by a vector $\bu_h\in \mathbb{R}^{N_h}$ corresponding to the DoFs of the problem. However, this representation could lack \textit{spatial coherence}, especially in the context of unstructured grids. In such scenarios, neighboring entries in the vector may correspond to spatial points that are not necessarily adjacent in both the physical and computational domains. 
This lack of spatial coherence prevents convolutional architectures from effectively leveraging local spatial information.
Additionally, using reshape operations, whether at the input level to adjust tensor shapes or due to the inclusion of dense layers requiring flattening operations, further restricts the exploitation of spatial information.

We proceed to address these issues by adopting a dimensionality reduction strategy based on a fully-convolutional AE architecture, thus able to preserve spatial features at the latent level given convolutions' local nature. 
Specifically, in order to avoid employing input reshape operations, on which DL-ROMs architectures often rely  \cite{fresca2021comprehensive, franco2023deep}, a structured grid $\Omega^\natural_h$, obtained via a uniform discretization\footnote{The novel discretization $\Omega^\natural_h$ uses as discretization parameter a value as close as possible to the original one $h$, resulting in the same number of DoFs $N_h$, thereby we avoid introducing additional notation.} of the original domain $\Omega$, is introduced.
Being $\Omega^\natural_h$ a uniform discretization, it can be now processed by the AE convolutional layers without resorting to reshaping operations to suitably adapt the high-fidelity state's shape. Thus, the usually employed reshape operation is now replaced by an interpolation operation, so that the state $\bu_h\in\mathbb{R}^{N_h}$ computed on $\Omega_h$ is interpolated onto the uniform grid $\Omega^\natural_h$, by means of a suitable interpolation operator $I_\natural$. As a consequence, we obtain a state vector $I_\natural \bu_h\in\mathbb{R}^{N_h}$ which exhibits a spatially-coherent ordering, and is thus prepared to be inputted into the convolutional architecture, thanks to the adopted uniform discretization\footnote{Since the method will always rely on $I_\natural$ to process the FOM state, we will refer to the spatially-coherent FOM state $I_\natural \bu_h$ simply as $\bu_h$.}.

In the following, the internal structure of the convolutional AE is described. 
Specifically, the architecture relies on a composition of convolutional layers and interpolation operations. 
Both the encoder $\Psi_\theta$ and decoder $\Psi'_\theta$ are structured into $L$ down- and up-sampling stages, respectively indicated as $s_l:\mathbb{R}^{2^{-(l-1)}N_h}\rightarrow \mathbb{R}^{2^{-{l}}N_h}$ and $s'_l:\mathbb{R}^{2^{-{l}}N_h}\rightarrow \mathbb{R}^{2^{-(l-1)}N_h}$, with $l=1,\ldots,L$. Such stages are then composed, resulting in $\Psi_\theta =  s_L \circ s_{L-1} \circ \cdots \circ s_1$ and $\Psi'_\theta =  s'_1 \circ s'_{2} \circ \cdots \circ s'_{L}$, as shown in the following 
\[\begin{tikzcd}
{\mathbf{u}_h\in \mathbb{R}^{N_h}} && {\mathbb{R}^{2^{-1}{N_h}}} && \cdots && {\mathbb{R}^{2^{-L}{N_h}}\ni \mathbf{u}_n}
\arrow["{s_1}", shift left=2, shorten <=13pt, shorten >=13pt, from=1-1, to=1-3]
\arrow["{s'_1}", shift left=2, shorten <=13pt, shorten >=13pt, from=1-3, to=1-1]
\arrow["{s_2}", shift left=2, shorten <=14pt, shorten >=14pt, from=1-3, to=1-5]
\arrow["{s'_2}", shift left=2, shorten <=14pt, shorten >=14pt, from=1-5, to=1-3]
\arrow["{s_L}", shift left=2, shorten <=14pt, shorten >=14pt, from=1-5, to=1-7]
\arrow["{s'_L}", shift left=2, shorten <=14pt, shorten >=14pt, from=1-7, to=1-5]
\end{tikzcd}\]
with the first row referring to the encoder, while the second one refers to the decoder;
here, we denote by $n = 2^{-L}N_h$ the reduced state dimension. 
Both $l$-th level stages, $s_l$ and $s'_l$, rely on the same internal structure, reading respectively as
\begin{equation}
    s_l = I_{2^{l-1}h}^{2^{l}h} \circ R_l, \qquad
    s'_l = I_{2^{l}h}^{2^{l-1}h} \circ R'_l,
\end{equation}
where: 
\begin{enumerate}[(i)]
\item $R_l$ and $R'_l$ refer to a preactivation \textit{residual} convolutional block \cite{he2016identity}, with learnable convolutional kernels $\kappa_{in}, \kappa_{out}$, and suitable nonlinearity $\sigma$, both reading as
\begin{equation}
\bh_{l-1} \mapsto \bh_{l-1} + \kappa_{out}\ast \sigma(\kappa_{in}\ast\sigma(\bh_{l-1})).
\end{equation}
with $\bh_{l-1}$ a hidden state, output of the $(l-1)$-th level;
\item $I_{2^{l-1}h}^{2^{l}h}:\mathbb{R}^{2^{-(l-1)}N_h} \rightarrow \mathbb{R}^{2^{-l}N_h}$ and $I_{2^l h}^{2^{l-1}h}:\mathbb{R}^{2^{-l}N_h} \rightarrow \mathbb{R}^{2^{-(l-1)}N_h}$ are the interpolation operators. The former, in $s_l$, performs a down-sampling operation by halving the spatial dimension, resulting in a coarsening operator. On  the other hand, within $s_l'$, the operation is reversed, to perform up-sampling.
\end{enumerate}

Concerning the implementation details, we employ convolutional layers with a kernel size of 3. Both the encoder and the decoder are equipped with linear convolutional layers as their first and last layers. We adopt an ELU nonlinearity \cite{clevert2016fast}, being continuously differentiable with bounded first derivative, thus satisfying the Lipschitz requirement on $\Psi_\theta$ and $\Psi'_\theta$.
We emphasize that our approach is motivated by the utilization of regular domains. For the case of domains characterized by irregularities, or more complex geometries, approaches such as graph neural networks (GNNs) or mesh-informed strategies may be more suitable options \cite{franco2023geom, franco2023meshinformed, pfaff2021learning}; however, these aspects are beyond the focus of this work, and can represent further extensions of the proposed framework.

\begin{figure}[t]
    \centering
    \includegraphics[width=0.52\textwidth]{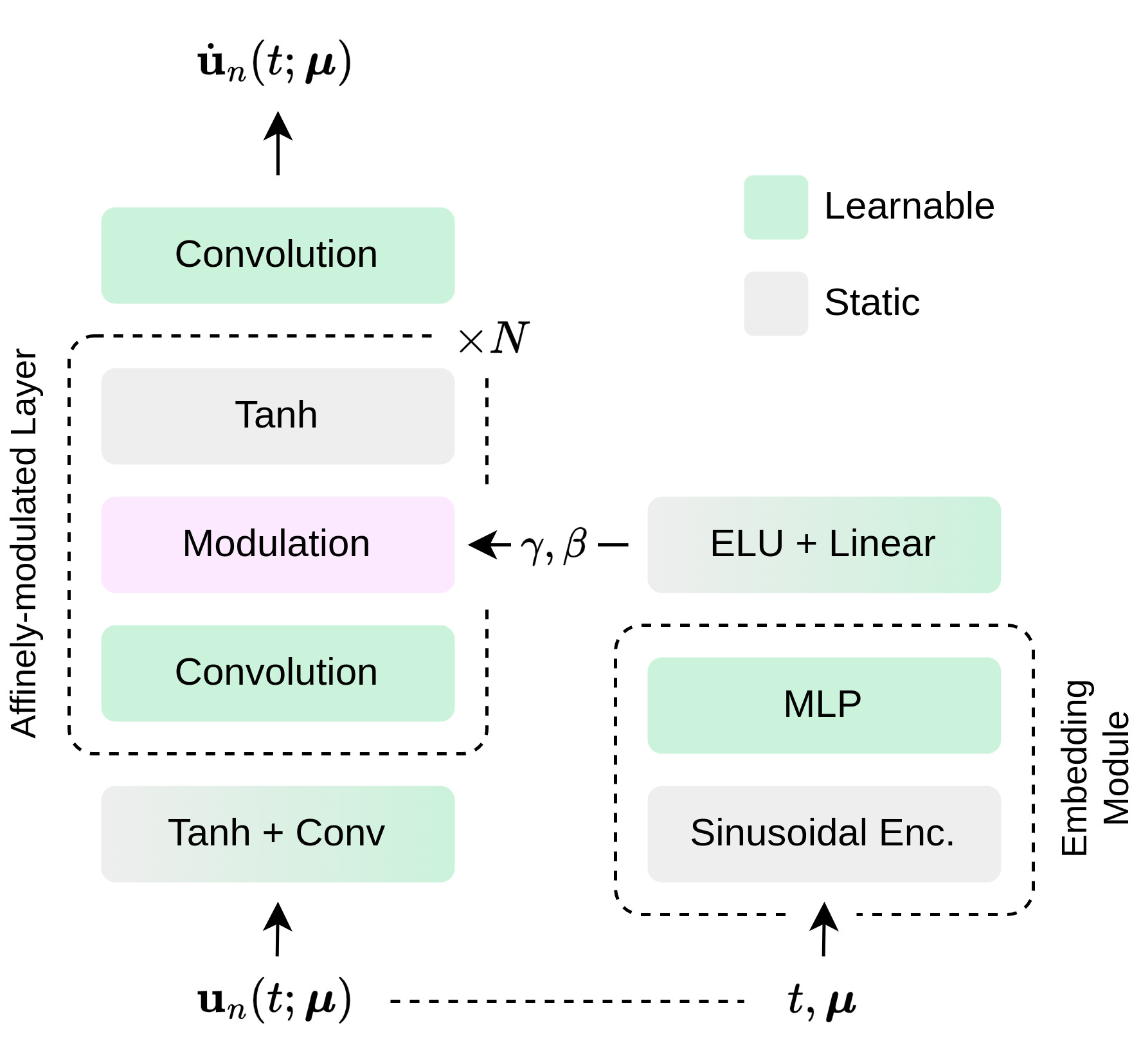}
    \caption{\textit{Affinely-parameterized latent dynamics architecture.}
    Detailed structure of the proposed convolutional parameterized latent dynamics. The two main components -- the embedding module, which processes the time-parameters inputs, and the stack of affinely-modulated convolutional layers -- are highlighted.
    }
    \label{fig:apnode}
\end{figure}

\subsection{Affinely-parameterized latent dynamics}
\label{sec:affine-param-latent-dynamics}
The concept of \textit{neural ordinary differential equation} (NODE) \cite{pmlr-v80-lu18d, NEURIPS2018_69386f6b, kidger2022neuraldifferentialequations} represents an implicitly-defined class of models which -- unlike traditional modeling approaches directly establishing an input-output relation $\bx\mapsto \by$ as $\by=\bff_\theta(\bx)$ -- adopts the following ODE-IVP formulation
\begin{equation*}
    \dot{\by}(t) = \bff_\theta(t,\by(t)),  \qquad \by(0)=\bx.  
\end{equation*}
Here, $\bff_\theta$ is now a neural network parameterizing the vector field by which the input is continuously evolved into the output over a prescribed time interval $t\in[0,T]$ (see Appendix \ref{appendix:continuous-time}). 
The inherently continuous-time nature of this family of models makes them well-suited for modeling time-dependent data \cite{NEURIPS2019_42a6845a}. Specifically, their formulation has led to a wide adoption of such architectures in the field of data-driven reduced order modeling, to model the reduced dynamics $\bff_{n,\theta}$ in combination with AEs \cite{Lee_2021, Di_Sante_2022, legaard2022constructing, Chen_2022}. Recently, \cite{Lee_2021} extended the NODE formulation to incorporate parametric information represented by $\bmu\in\mathcal{P}$. This enhancement improves the expressivity of the neural network, resulting in the concept of parameterized neural ODEs (PNODEs). In contrast to the prevailing approaches of using dense neural networks to parameterize the PNODE dynamics, we explore the adoption of a convolutional architecture. This choice is motivated by the intrinsic sparsity of convolutional layers, resulting in a smaller network size footprint. Moreover, such choice allows us to leverage the spatial information encoded by the convolutional AE, possibly enabling a spatial-coherence between the encoded and the high-dimensional states, thus enhancing interpretability. 
In the following, the architectural choices in the design of an \textit{affinely-parameterized latent dynamics} $\bff_{n,\theta}$, are described. Specifically, we delineate the embedding strategy employed for the temporal and parametric information $(t,\bmu)$, by means of \textit{affine modulation} techniques, departing from traditional concatenation or hypernetwork-based approaches. 

\paragraph{Time-parameters embedding.}
A modeling challenge is represented by the effective embedding of the time $t\in\mathbb{R}_+$ and parameters $\bmu\in\mathcal{P}$ into the reduced representation, leading to a parameterized latent state $\bu_n(t;\bmu)$. Such problem is naively tackled by directly providing $(t,\bmu)$ as inputs to the neural network architecture, after performing a rescaling operation to normalize their ranges. 
In recent NODE-based ROM approaches, the main strategies for injecting additional information into the dynamics is represented by either concatenation or by relying on hypernetworks \cite{Lee_2021,wen2023reduced}. Our approach departs from standard parameters embedding techniques. In particular, for embedding temporal and parametric information, we draw inspiration from recent advances in embedding strategies widely adopted in state-of-the-art DL architectures, from diffusion models to transformers \cite{vaswani2023attention}. 
Rather than using a basic rescaling of scalar features on $t$ and the components $\mu_j$ of $\bmu$, we propose to map them to a higher-dimensional vector space $\mathbb{R}^{k}$ via a \textit{sinusoidal encoding}, where $k\in\mathbb{N}$ (even) represents the number of frequency components. 
In particular, the employed mappings are $t\mapsto \tilde{\mathbf{t}}\in\mathbb{R}^k$ for the time feature, and $\bmu\mapsto \tilde{\bmu}\in\mathbb{R}^{k n_\mu}$ for the parameters vector, since the embedding is applied component-wise to each scalar parameter $\mu_j, \ j=1,\ldots,n_\mu$.
Moreover, sinusoidal encoding allows mapping scalar features to vectors with entries bounded in the range $[-1,1]$, thereby avoiding the need for separate rescaling operations. The specific form of the employed encoding reads as
\begin{equation}
    t\mapsto \tilde{\mathbf{t}}(t) = \begin{bmatrix}
        \sin(\omega_1 t) & \sin(\omega_2 t) & \cdots & \sin(\omega_{\frac{k}{2}} t) & \cos(\omega_1 t) & \cos(\omega_2 t) & \cdots & \cos(\omega_{\frac{k}{2}} t)
    \end{bmatrix},
    \label{eq:sinenc}
\end{equation}
\begin{equation*}
    \omega_j =  \frac{1}{T_\text{max}^{j-1}}, \qquad j=1,...,\frac{k}{2},
\end{equation*}
where $T_\text{max}$ is a large enough period, allowing  the scalar input $t$ to be represented by a $k$-dimensional signal $\tilde{\mathbf{t}}$ of period $2\pi T_\text{max}^{\frac{k}{2}-1}$, thus exceeding the typical time-scales of the modeled phenomena $(\gg T_\text{train}, T_\text{test})$. The mapping \eqref{eq:sinenc} can be extended in a straightforward way to the scalar components $\mu_j$ of the parameters vector $\bmu\in\mathbb{R}^{n_\mu}$ to perform the encoding of each component, resulting in $\tilde{\bmu}$.
Subsequently, the encoded features $(\tilde{\mathbf{t}}, \tilde{\bmu})$ are concatenated and processed by a multi-layer perceptron (MLP) with ELU nonlinearity, resulting in a combined embedding $\bxi\in \mathbb{R}^{d_e}$, containing both the temporal and parametric information. Similar embedding approaches, based on sinusoidal encoding, have been adopted in autoregressive architectures employed in the context of DL-based surrogate modeling, for embedding either the temporal or parametric information \cite{gupta2022multispatiotemporalscale, li2024latent}. Finally, the embedded information $\bxi$ is injected into the convolutional dynamics $\bff_{n,\theta}$ via \textit{affine modulation} strategies, as detailed in the following.

\paragraph{Parameterization via affine modulation.}
Effectively including additional information, alongside the primary input, to guide and control learning, is a task of central importance in DL.
In particular, in the context of multi-query DL-based surrogate modeling, this translates into the task of including the information related to $\bmu\in\mathcal{P}$ to construct the surrogate model.
Moreover, embedding the time-parameters information is of paramount importance in the specific case of learning a parameterized dynamics $\bff_{n,\theta}(t,\ \cdot\ ;\bmu)$, taking into account both the temporal and parametric information $(t,\bmu)\in\mathbb{R}_+\times\mathcal{P}$. A possible approach to tackle this problem consists in using hypernetworks, where an auxiliary network takes the parametric information as input and outputs the weights of the main network, represented in our case by the latent dynamics $\bff_{n,\theta}$. While effective, this approach often results in models of considerable size, due to the necessity to infer the main network's entire set of parameters, thus hindering scalability and possibly leading to memory constraints and computational overhead.
To cope with this issue, \cite{wen2023reduced} proposes a factorization-based approach, applied to the dense layers of the ODE-Net modeling the latent dynamics. Specifically, a hypernetwork is employed to infer only the diagonal entries $\Sigma(\bmu)_{ij}$ of a SVD-like decomposition of the weights matrices $W_\theta(\bmu) = U\Sigma(\bmu)V^T$ of the latent dynamics dense layers, thus reducing the cost associated to the hypernetwork. In contrast, mainly motivated by our fully-convolutional architecture, we adopt an \textit{affine modulation} approach, drawing inspiration from recent advancements in conditional image generation, where class or textual information influences the generative process.
In particular, in conditional normalization-based methods \cite{ioffe2015batch, ba2016layer}, a function of the conditioning input $\bxi$ is learned to output the normalization layers scaling $\bgamma(\bxi)$ and shifting $\bbeta(\bxi)$ parameters, performing an affine transformation of the following form
\begin{equation}
    \tilde{\mathbf{x}} = \bgamma(\bxi)\odot \mathbf{x} \ \oplus \ \bbeta(\bxi),
    \label{eq:affinemod}
\end{equation}
with $\mathbf{x}$ an arbitrary (possibly hidden) feature of the network being conditioned, to which a prior normalization technique has been applied. In particular, such conditioning technique has found broad application in diffusion models, to incorporate the timestep and class embedding into the residual blocks of the UNet employed in the diffusion process \cite{nichol2021improved, dhariwal2021diffusion}. A more general framework, complementing the previous ones, is represented by \cite{perez2017film}, which employs \textit{feature-wise} affine transformations to modulate the network's intermediate features, based on additional input sources. 
More broadly, as pointed out in \cite{perez2017film}, these methods of injecting conditional signals via affine feature modulation can be viewed as instances of hypernetwork architectures \cite{ha2016hypernetworks}, where a subsidiary network outputs the parameters of the main network, based on the conditioning input. However, affine modulation schemes are much more efficient since the number of parameters required for the affine transformation is significantly smaller than the main network's total parameters count.

Building up on affine modulation techniques, we define a convolutional affinely-parameterized NODE, to include both time $t$ and parameters $\bmu$ into the latent dynamics $\bff_{n,\theta}$, by means of the embedding $\bxi(t,\bmu)$. 
Thus, relying on the conditioned affine transformation \eqref{eq:affinemod} to include the embedded parametric information $\bxi(t,\bmu)$, we can define an affinely-modulated convolutional layer, representing the core of the convolutional \textit{affinely-parameterized latent dynamics} $\bff_{n,\theta}$ illustrated in Figure \ref{fig:apnode}, reading as 
\begin{equation}
    \bh_{i-1} \mapsto \bgamma_i(\bxi)\odot (\kappa_i \ast \bh_{i-1}) \ \oplus \ \bbeta_i(\bxi),
    \label{eq:modulatedconv}
\end{equation}
where $\kappa_i$ represent the learnable kernel of the $i$-th convolutional layer, with $n_c$ channels, and $\bh_{i-1}$ the modulated hidden state. The affine-parameterization is performed via the scaling and shifting parameters $\bgamma_i(\bxi), \bbeta_i(\bxi) \in \mathbb{R}^{n_c}$, which modulate the $i$-th layer preactivation along the channel dimension, via channel-wise multiplication $\odot$ and sum $\oplus$, respectively. The dependence of $\bgamma_i(\bxi), \bbeta_i(\bxi)$ on the time and parameter instances $(t,\bmu)$ is implicit, by means of $\bxi(t,\bmu)$, produced by the embedding module previously introduced. In particular, the mapping $\bxi \mapsto (\bgamma_i, \bbeta_i)$ is modeled by a linear layer preceded by a nonlinearity $\sigma$, reading as $W \sigma ( \ \cdot \ ) \ + \mathbf{b} : \mathbb{R}^{d_e} \rightarrow \mathbb{R}^{2n_c}$. The choice of performing the affine transformation along the channels dimension makes the $(t,\bmu)$-parameterization agnostic to the spatial dimension of the latent state, decoupling the size of the embedding network from the latent state spatial dimension, thus making the technique scalable and computationally efficient. 

Regarding the implementation details, convolutional layers with a kernel size of 3 are employed. The employed activation functions include an ELU nonlinearity within the embedding module, and $\tanh$ in the modulated convolutional layers,  both being Lipschitz continuous.
Finally, in Algorithm \ref{alg:latent_dynamics}, we outline the internal structure of the proposed \textit{affinely-parameterized latent dynamics} $\bff_{n,\theta}$, summarizing the components described above, in the case of two modulated convolutional layers of kernels $\kappa_1,\kappa_2$.
\begin{algorithm}[h]
\caption{Affinely-parameterized latent dynamics $\bff_{n,\theta}$}
\begin{algorithmic}[1]
\Require Latent state $\bu_n(t;\bmu)\in\mathbb{R}^n$, time instance $t\in\mathbb{R}_+$, parameter instance $\bmu\in\mathcal{P}\subset\mathbb{R}^{n_\mu}$ 
\Ensure  Latent state time-derivative $\dot{\bu}_n(t;\bmu)\in\mathbb{R}^n$
\State Encode time instance $\tilde{\mathbf{t}}\leftarrow \operatorname{SinusoidalEncoding}(t,k,T_\text{max})\in\mathbb{R}^k$
\State Encode parameter instance $\tilde{\mathbf{\bmu}}\leftarrow \operatorname{SinusoidalEncoding}(\bmu,k,T_\text{max})\in\mathbb{R}^{k n_\mu}$
\State Time-parameter embedding $\bxi=\operatorname{MLP}(\tilde{\mathbf{t}},\tilde{\bmu}) \in \mathbb{R}^{d_e}$
\State Affine modulation parameters $\bgamma_1,\bgamma_2,\bbeta_1,\bbeta_2 \leftarrow W\sigma(\bxi)+\mathbf{b} \in \mathbb{R}^{4n_c}$
\State Input convolution $\bh = \tanh(\kappa_\text{in}\ast \bu_n)$

\State Convolve and modulate $\bh = \tanh( \bgamma_1 \odot (\kappa_1 \ast \bh) \ \oplus \ \bbeta_1 )$
\State Convolve and modulate $\bh = \tanh( \bgamma_2 \odot (\kappa_2 \ast \bh) \ \oplus \ \bbeta_2 )$
\State Output convolution $\dot{\bu}_n(t;\bmu) = \kappa_\text{out}\ast \bh$
\end{algorithmic}
\label{alg:latent_dynamics}
\end{algorithm}

\begin{remark}
Adopting a bounded activation function within the parameterized dynamics $\bff_{n,\theta}$, as $\tanh$, enables a straightforward way to estimate $\|\bff_{n,\theta}\|_{\infty}$. Indeed, considering a 1-dimensional convolutional kernel $\kappa_\text{out}$ of size $2k+1$ with $c$ input channels and 1 output channel, and given $\tanh$ bounded range $(-1,1)$, it follows that
\begin{equation*}
    \|\bff_{n, \theta}\|_\infty \leq \sum_{i=1}^c\sum_{i=-k}^{k} |\kappa_\text{out}^{ij}|.
\end{equation*}
\end{remark}

\subsection{Training and testing schemes}
In the following, the training and testing procedures for the DL-based LDM are described. We proceed by denoting with $\bS_h\in\mathbb{R}^{N_h \times N_s}$ and $\bM\in\mathbb{R}^{(n_\mu+1)\times N_s}$ the matrices collecting the high-fidelity solutions $\bu_h(t_i;\bmu^j)$, and the associated time-parameter instances $(t_i,\bmu^j)$, respectively. We indicate by $N_s=N_t N_\mu$ the total number of available snapshots, with $N_t$ the length of the collected trajectories, and $N_\mu$ the total number of trajectories computed for different instances of the parameters $\bmu^j$, resulting in
\begin{equation*}
    \bS_h = \bigg[ \ \bu_h(t_0;\bmu^1)\ | \ \cdots \ | \ \bu_h(t_{N_t-1};\bmu^{1}) \ | \ \cdots \ | \ \bu_h(t_0;\bmu^{N_\mu})\ | \ \cdots \ | \ \bu_h(t_{N_t-1};\bmu^{N_\mu}) \ \bigg],
\end{equation*}
\begin{equation*}
    \bM = \Bigg[ \
    \begin{matrix}
        \bmu^1\\
        t_0
    \end{matrix}\ 
    \bigg | 
    \ \cdots \ 
    \bigg | 
    \ 
    \begin{matrix}
        \bmu^1\\
        t_{N_t-1}
    \end{matrix}\  
    \bigg | 
    \ \cdots \ 
    \bigg | 
    \
    \begin{matrix}
        \bmu^{N_\mu}\\
        t_0
    \end{matrix}\ 
    \bigg | 
    \ \cdots \ 
    \bigg | 
    \ 
    \begin{matrix}
        \bmu^{N_\mu}\\
        t_{N_t-1}
    \end{matrix}
    \ \Bigg].
\end{equation*}
Before the start of the training routine, the training and validation snapshot matrices $\bS_h^{train}, \bS_h^{val}$, obtained by means of a splitting along the temporal and parameters dimensions, accordingly to splitting ratios $\alpha,\beta \in(0,1)$, undergo a normalization step to rescale the features in the [-1,1] range, reading as
\begin{equation}
    \bS_h^{*} \mapsto 2\cdot\frac{\bS_h^{*}\ - \ \min(\bS^{train}_{h})}{\max(\bS^{train}_{h})-\min(\bS^{train}_{h})} - 1,
    \label{eq:scale}
\end{equation}
with $\bS_h^{*}$ indicating either $\bS_h^{train}$ or $\bS_h^{val}$. When dealing with a FOM state composed of multiple scalar components, transformation \eqref{eq:scale} is applied separately, rescaling each component. As follows, Algorithm \ref{alg:training} and \ref{alg:testing} provide a detailed description of the training and testing schemes, respectively. Specifically, the outlined training procedure builds up on the minimization problem \eqref{eq:dldp} and Algorithm \ref{alg:pseudo-training}.
In particular, given a collection of snapshots $\bS_h^*$,  either belonging to the training or validation set, we denote by $\bS_{0,h}^*$ the collection of the initial values $\{\bu_h(t_0;\bmu^j)\}_{j=1}^{N_\mu}$, obtained by extracting the slice referring to the initial time instance $t_0$. Then, $\bS_{0,n}^*$ refers to the collection of \textit{latent} initial values, obtained by applying the encoder function $\Psi_\theta$ to $\bS_{0,h}^*$. Similarly, $\bS_{n}^*$ refers to the evolution of the latent states, obtained by integrating the latent dynamics $\bff_{n,\theta}$ accordingly to a prescribed RK scheme. Then $\hat{\bS}_{h}^*$ denotes the LDM approximation of the FOM state evolution, output of the decoder function $\Psi'_\theta$, applied to the previously computed latent representations $\bS_{n}^*$. Given the approximation $\hat{\bS}_{h}^*$, the inverse transformation of \eqref{eq:scale} is applied in order to rescale the model output to the original range.
The optimization procedure involved in the training, outlined in Algorithm \ref{alg:training} (8-15), relies on Adam algorithm \cite{kingma2017adam} to perform the model parameters update, by employing a suitable learning rate $\eta$ and weight decay $\lambda$. 

\begin{algorithm}[ht]
\caption{$\dldm_\theta$ training algorithm (w/ sub-trajectories)}
\begin{algorithmic}[1]
\Require Snapshot matrix $\bS_h\in\mathbb{R}^{N_h\times N_s}$, parameter matrix $\bM\in\mathbb{R}^{(n_\mu+1)\times N_s}$, time-parameter training-validation split ratios $\alpha, \beta\in(0,1)$, learning rate $\eta$, weight decay $\lambda$, batch size $N_b$, sub-trajectories length $2\leq\ell\leq \min\{N_t^{train},N_t^{val}\}$, maximum number of epochs $N_{epochs}$, early-stopping criterion, Runge-Kutta scheme RK 
\Ensure  Optimal model parameters $\btheta^*= (\btheta_{\Psi}^*, \btheta_{\Psi'}^*, \btheta_{\bff_n}^*)$
\State Split data $\bS_h=[\bS_h^{train}, \bS_h^{val}], \ \bM=[\bM^{train}, \bM^{val}]$ 
\Statex \quad with $N_t^{train}=\alpha N_t, \ N_t^{val}=(1-\alpha) N_t$, and $N_\mu^{train}=\beta N_\mu, \ N_\mu^{val}=(1-\beta) N_\mu$ 
\State Normalize $\bS_h$ accordingly to \eqref{eq:scale}
\State Assemble \textit{train} sub-traj. $\tilde{\bS}_h^{train}\in\mathbb{R}^{N_h\times\ell\times N^{train}_\mu(N_t^{train}-\ell+1)}, \ \tilde{\bM}^{train}\in\mathbb{R}^{(n_\mu+1)\times\ell\times N^{train}_\mu(N_t^{train}-\ell+1)}$
\Statex \quad with $N_{batches}^{train} = N^{train}_\mu(N_t^{train}-\ell+1) / N_b$
\State Assemble \textit{val} sub-traj. $\tilde{\bS}_h^{val}\in\mathbb{R}^{N_h\times\ell\times N^{val}_\mu(N_t^{val}-\ell+1)}, \ \tilde{\bM}^{val}\in\mathbb{R}^{(n_\mu+1)\times\ell\times N^{val}_\mu(N_t^{val}-\ell+1)}$
\Statex \quad with $N_{batches}^{val} = N^{val}_\mu(N_t^{val}-\ell+1) / N_b$
\State i = 0
\While {$(\neg\text{early-stopping}\  \wedge \ i\leq N_{epochs})$}
\For{$k = 1 : N_{batches}^{train}$}
\State Get $k$-th mini-batch $(\bS_h^{batch},\bM^{batch})\subseteq(\tilde{\bS}_h^{train},\tilde{\bM}^{train})$
\State Extract initial values $\bS_{0,h}^{batch}$ from $\bS_h^{batch}$
\State Extract time and parameters instances $(\mathbf{t}^{batch}, \bmu^{batch})\leftarrow \bM^{batch}$
\State Project initial values $\bS_{0,n}^{batch} = \Psi_\theta(\bS_{0,h}^{batch};\btheta_{\Psi})$
\State Latent ODE integration $\bS_{n}^{batch} = \text{RK}(\bff_{n,\theta}, \bS_{0,n}^{batch}, \mathbf{t}^{batch}, \bmu^{batch}; \btheta_{\bff_n})$
\State Reconstruct trajectories $\hat{\bS}_h^{batch} = \Psi'_\theta(\bS_n^{batch};\btheta_{\Psi'})$
\State Accumulate loss $\mathcal{L}(\btheta)$ on $(\bS_h^{batch},\bM^{batch})$, compute $\nabla_\theta\mathcal{L}$
\State Update model parameters $\btheta\leftarrow \operatorname{Adam}(\eta,\lambda,\nabla_\theta\mathcal{L},\btheta)$
\EndFor
\State Repeat (7-13) on $(\bS_h^{val},\bM^{val})$ 
\State Accumulate loss $\mathcal{L}(\btheta)$ on $(\bS_h^{val},\bM^{val})$
\State i++
\EndWhile
\end{algorithmic}
\label{alg:training}
\end{algorithm}

\begin{remark}[\textit{Temporal regularization}]
\label{rmk:temporal_reg}
To reduce the dependence on the sub-trajectories' length $\ell$, a temporal regularization approach \cite{NEURIPS2020_a9e18cb5} may be adopted, involving randomly sampled integration interval lengths. It can be implemented in Algorithm \ref{alg:training} by fixing a maximum sub-trajectory length $\ell_{max} > 2$, used to construct the training matrices $\bS^{train}_h,\bM^{train}$. At each iteration, a trajectory length is then sampled $\ell\sim \mathcal{U}[2,\ell_{max}]$, allowing the selection of a batch with sub-trajectory length $\ell$ as $\bS^{batch}_{h,0:\ell}, \ \bM^{batch}_{0:\ell}$.
\end{remark}

\begin{algorithm}[h]
\caption{$\dldm_\theta$ testing algorithm}
\begin{algorithmic}[1]
\Require Initial snapshot matrix $\bS_{0,h}^{test}\in\mathbb{R}^{N_h\times N_\mu}$, parameter matrix $\bM^{test}\in\mathbb{R}^{(n_\mu+1)\times N_s}$,
optimal model parameters $\btheta^*= (\btheta_{\Psi}^*, \btheta_{\Psi'}^*, \btheta_{\bff_n}^*)$, Runge-Kutta scheme RK, perturbation $\bdelta_h\in\mathbb{R}^{N_h}$ 
\Ensure  Approximation $\hat{S}^{test}_h\in\mathbb{R}^{N_h\times N_s}$
\State Normalize $\bS_{0,h}^{test}$ accordingly to \eqref{eq:scale}
\If{$\bdelta_h \neq \mathbf{0}$}
\State Perturb initial values $\bS_{0,h}^{test} \leftarrow \bS_{0,h}^{test} + \bdelta_h$
\EndIf
\State Extract timesteps and parameters $(\mathbf{t}^{test}, \bmu^{test})\leftarrow \bM^{test}$
\State $\bS_{0,n} = \Psi_\theta(\bS_{0,h}^{test};\btheta_{\Psi}^*)$
\State $\bS_{n} = \text{RK}(\bff_{n,\theta}, \bS_{0,n}, \mathbf{t}^{test}, \bmu^{test}; \btheta_{\bff_n}^*)$
\State $\hat{\bS}_h^{test} = \Psi'_\theta(\bS_n;\btheta_{\Psi'}^*)$
\State Inversely normalize $\hat{\bS}_h^{test}$ accordingly to $\neg$\eqref{eq:scale}
\end{algorithmic}
\label{alg:testing}
\end{algorithm}

Additionally, the testing Algorithm \ref{alg:testing} includes the option to perturb the initial datum (line 3), to test LDMs' zero-stability via a suitable perturbation $\bdelta_h$. 

\section{Numerical results}
This section is concerned with the evaluation of the performance of the proposed framework in the context of reduced order modeling of parameterized dynamical systems, i.e., FOMs of the form \eqref{eq:FOM}, arising from the semi-discretization of parameterized nonlinear time-dependent PDEs. The goal of the following numerical tests is to assess LDMs' capabilities in {\em (i)} accurately handling time evolution, {\em (ii)} satisfying the time-continuous approximation property, {\em (iii)} ensuring zero-stability in a learnable context, and {\em (iv)} effectively capturing the parametric dependence in a multi-query context.

\subsection{Problems and experimental setup}
The considered problems on which  empirical tests are performed are {\em (i)} a one-dimensional Burger's equation, and {\em (ii)} a two-dimensional advection-diffusion-reaction (ADR) equation. The specific setup of each problem is described below.

\paragraph{1D Burgers' equation.} As a first problem, the one-dimensional Burgers' equation is considered, reading as
\begin{equation}
    \begin{cases}
    u_t - \nu u_{xx} + uu_x = 0  & \quad (x,t)\in\Omega \times (0,T],\\
    u_x(x,t) = 0 & \quad (x,t) \in \partial \Omega \times (0,T],\\
    u(x,0) = e^{-x^2} & \quad  x \in \Omega.
    \end{cases}
\end{equation}
Here $\nu \in \mathbb{R}$ denotes the viscosity coefficient, varying in the interval $\nu \in \mathcal{P} = [5\cdot 10^{-3},1]$, having in this case $\dim\mathcal{P} = n_\mu = 1$. Regarding the parameter space splitting and discretization, we consider $\mathcal{P} = [5\cdot 10^{-3}, 5\cdot 10^{-2}] \cup (5\cdot 10^{-2}, 1]$, where the former interval is used for training and interpolation testing purposes, while the latter is reserved for extrapolation testing. In particular, the first interval is discretized into $N_\mu^{train}=100$ equally spaced instances, forming the training set $\mathcal{P}_{train}$, where 20\% of these instances are kept for validation. For testing purposes with respect to the parametric dependence, two discrete sets are constructed. $\mathcal{P}_{interp}$ refers to the interpolation parameter instances, obtained by taking the midpoints of the training instances in $\mathcal{P}_{train}$, resulting in $N_\mu^{interp} = 99$ data points. Considering the extrapolation testing instances, $\mathcal{P}_{extrap}$ is constructed by considering $N_\mu^{extrap} = 50$ equally spaced points over $(5\cdot 10^{-2}, 1]$. The FOM is built by employing linear ($\mathbb{P}_1$) finite elements, with a spatial grid characterized by $N_h=1024$ DoFs over $\Omega = (-10,10)$. The system has been solved in time via implicit differentiation of order 1, with $N_t=1000$ time steps over the interval $[0,T]$ with $T=30$, resulting in a time step $\Dt_{FOM} = 0.03$. A temporal splitting of the form $[0,T_1]\cup(T_1,T_2]\cup(T_2,T]$ is adopted, where the first two intervals refer to the training and validation partitions, while the latter refers to the time-extrapolation interval. Specifically, $T_1=12, \ T_2=15$ are chosen as the endpoints of the training and validation intervals, respectively.

\paragraph{2D Advection-diffusion-reaction equation.}
As a second test case, we consider a time-dependent parameterized advection-reaction-diffusion problem on a two-dimensional domain $\Omega = (0,1)^2$, defined as follows
\begin{equation}
    \begin{cases}
    u_t - \nabla \cdot (\mu_1 \nabla u) + \bb(t)\cdot \nabla u + cu = f(\bx; \mu_2,\mu_3) & \quad (\bx,t)\in\Omega \times (0,T],\\
    \mu_1 \nabla u \cdot \bn = 0 & \quad (\bx,t) \in \partial \Omega \times (0,T], \\
    u(\bx,0) = 0 &  \quad \bx \in \Omega,
    \end{cases}
\end{equation}
with
\begin{equation*}
f(\bx; \mu_2,\mu_3) = 10\cdot\exp(-((x-\mu_2)^2+(y-\mu_3)^2)/0.07^2),\qquad\bb(t)=[\cos(t), \ \sin(t)]^T, \qquad c=1.
\end{equation*}
In this case $n_\mu=3$, so that the parameters left varying are $\bmu=(\mu_1,\mu_2,\mu_3)$, where the former refers to the diffusive effect, while the latter two set the position of the forcing term, thus acting on a geometrical aspect of the problem. In particular, the nonlinear dependence of the solution on $\mu_2$ and $\mu_3$ categorizes the problem as nonaffinely parameterized, posing a challenge for traditional projection-based techniques \cite{dal2019algebraic}.
The adopted parameters space is $\mathcal{P} = [1\cdot 10^{-2},6\cdot 10^{-2}]\times[0.3,0.7]^2$. Specifically, considering $[2\cdot 10^{-2},5\cdot 10^{-2}]\times[0.4,0.6]^2 \subset \mathcal{P}$, 10 uniformly spaced instances are taken over each interval, resulting in $N_\mu^{train}=1000$ training parameters instances collected into $\mathcal{P}_{train}$, with 20\% of these instances reserved for validation. As in the previous test case, the midpoints of the training set instances are collected into $\mathcal{P}_{interp}$, leading to $N_\mu^{interp}=729$ interpolation testing instances. The model's extrapolation capabilities are tested on $N_\mu^{extrap}=400$ instances uniformly sampled from $\mathcal{P} \setminus [2\cdot 10^{-2},5\cdot 10^{-2}]\times[0.4,0.6]^2$, forming  $\mathcal{P}_{extrap}$. The FOM is solved through a linear ($\mathbb{P}_1$) finite element discretization in space, and considering 32 nodes on each side of $\Omega$, resulting in $N_h=1024$ DoFs.  The system has been solved in time via implicit differentiation of order 1, with $N_t=1000$ time steps over the interval $[0,T]$ with $T=10\pi$, thus with a time step $\Dt_{FOM} \simeq 0.03$. A temporal splitting of the form $[0,T_1]\cup(T_1,T_2]\cup(T_2,T]$ is again adopted, choosing $T_1=\frac{2}{5}T, \ T_2=\frac{1}{2}T$ as the endpoints of the training and validation intervals, respectively.

\paragraph{Experimental setup.}
The architecture of the models employed in the proposed numerical experiments follows the structure given in Section \ref{section:DLLDM}, characterized by a convolutional autoencoder and a convolutional parameterized neural ODE, equipped with the proposed affine modulation parameterization strategy. The training procedure outlined in Algorithm 1 is followed, adopting sub-trajectories of maximum sequence length $\ell_{max}=40$, together with the \textit{temporal regularization} method outlined in Remark \ref{rmk:temporal_reg}. Adam optimizer \cite{kingma2017adam} is employed with a decaying learning rate schedule starting from $\eta = 5\cdot 10^{-4}$, and weight decay $\lambda = 10^{-5}$. 
The models' implementation has been performed using the PyTorch framework. Specifically, the ODE integration numerical routines have been implemented from scratch, rather than relying on external libraries, in order to deal with parameter-dependent latent dynamics, and leverage just-in-time (JIT) compilation for improved training and inference performance.

For the purpose of evaluating LDMs' performance, the following error indicators are introduced: 
\begin{enumerate}[(i)]
\item a scalar indicator $\epsilon_{rel}(\bu_h, \hbu_h)\in\mathbb{R}$ representing the relative error averaged over $N_\mu$ parameters' instances and $N_t$ temporal steps\footnote{Here, $N_t$ may be replaced by an arbitrary sub-trajectory length $\ell\leq N_t$, such as in the case of short-term predictions.} ($N_s = N_t N_\mu$), reading as
\begin{equation}
    \epsilon_{rel}(\bu_h, \hbu_h) = \frac{1}{N_s}\sum_{j=1}^{N_\mu} \sum_{i=0}^{N_t-1} \frac{\|\bu_h(t_i;\bmu^j)-\hbu_h(t_i;\bmu^j)\|}{\|\bu_h(t_i;\bmu^j)\|},
\label{eq:rel_err_avg}
\end{equation}
measuring the average performances over the time-parameter space;
\item a parameter-dependent scalar indicator $\epsilon_{rel}(\bmu;\bu_h, \hbu_h)\in\mathbb{R}$, which quantifies the relative error averaged over $N_t$ temporal steps, for a given parameter instance $\bmu$, reading as
\begin{equation}
    \epsilon_{rel}(\bmu;\bu_h, \hbu_h) =  \frac{1}{N_t}\sum_{i=0}^{N_t-1} \frac{\|\bu_h(t_i;\bmu)-\hbu_h(t_i;\bmu)\|}{\|\bu_h(t_i;\bmu)\|},
\label{eq:rel_err_mu}
\end{equation}
used to assess the accuracy over the parameter space.
\item a time-dependent scalar indicator $\epsilon_{rel}(t;\bu_h, \hbu_h,\bmu)\in\mathbb{R}$, representing the pointwise-in-time relative error for a fixed instance of the testing parameter $\bmu$, defined as
\begin{equation}
    \epsilon_{rel}(t; \bu_h, \hbu_h, \bmu) =  \frac{\|\bu_h(t;\bmu)-\hbu_h(t;\bmu)\|}{\|\bu_h(t;\bmu)\|},
\label{eq:rel_err_t}
\end{equation}
\item a time-dependent vector indicator $\beps_{rel}(t; \bu_h, \hbu_h, \bmu)\in\mathbb{R}^{N_h}$ providing the pointwise-in-time relative error for a fixed instance of the testing parameter $\bmu$, defined as
\begin{equation}
    \beps_{rel}(t; \bu_h, \hbu_h, \bmu) =  \frac{|\bu_h(t;\bmu)-\hbu_h(t;\bmu)|}{\|\bu_h(t;\bmu)\|}.
\label{eq:rel_err_t_vec}
\end{equation}
\end{enumerate}
To maintain consistency with the previously introduced relative error measures, the error indicator used to assess the time-continuous approximation property is the supremum of \eqref{eq:rel_err_t} over the employed time-grid $\{t_k\}_{k=0}^{N_t-1}$, namely
\begin{equation}
    \epsilon_{rel}^{sup}(\bu_h, \hbu_h, \bmu) =  \sup_{k\in\{0,...,N_t-1\}} \epsilon_{rel}(t_k; \bu_h, \hbu_h, \bmu).
\label{eq:rel_err_sup}
\end{equation}

\subsection{Computational experiments}

The following experiments aim to assess the properties and generalization capabilities of the proposed framework, with a focus on four key modeling aspects: {\em (i)} temporal evolution, {\em (ii)} time-continuity, {\em (iii)} zero-stability and {\em (iv)}
parameter dependence. Although these aspects are addressed separately in the following subsections, they are closely connected throughout the different experiments. Indeed, being in a multi-query context, the performances with respect to varying parameter instances are always assessed, and, exploiting the time-continuous nature of the framework, finer testing temporal discretizations than the one employed at training-time are employed.

Regarding the models' architectural details, a latent dimension of $n=16$ is used throughout the experiments, resulting in a dimensionality reduction by a factor of 64, being $N_h=1024$ in both benchmark problems. Second- and fourth-order Ralston's Runge-Kutta (RK) schemes \cite{Ralston1962}  are adopted for the numerical solution of the parameterized latent NODE. Coarser training temporal discretizations, compared to the ones used for solving the FOM, are employed. In particular, denoting the training time step by $\Dt_{train}$, for the first benchmark problem $\Dt_{train} = 2\Dt_{FOM}$, while for the second one $\Dt_{train} = 5\Dt_{FOM}$. Moreover, the LDMs used in the following studies are characterized by a parameter count of at most $\sim$200K, significantly smaller than that of many overparameterized DL-based ROMs.

\subsubsection{Temporal evolution}
Being primarily a temporal modeling scheme based on latent dynamics learning, the proposed framework is tested to assess the temporal evolution of both the high-dimensional approximation provided by the LDM, and the underlying evolution of the latent state. Specifically, two types of tests are performed: the first evaluates the \textit{short-term} predictive capabilities of LDMs, while the second assesses the \textit{long-term} predictive performance. We remark that both tests are conducted on the FOM temporal discretizations, thus being $\Dt_{test} = \Dt_{train}/2$ in the Burgers' test case, and $\Dt_{test} = \Dt_{train}/5$ in the ADR test case.
In the following, RK2 is employed for integrating the latent dynamics, with parameters encoding of size $k=16$ and $k=8$ employed in the Burgers' and ADR test cases, respectively.

\begin{table}[h]
\center
\renewcommand{\arraystretch}{1.7}
\resizebox{\textwidth}{!}{
\begin{tabular}{@{}lllllllllll@{}}
\toprule
 & \multicolumn{1}{c}{} & \multicolumn{3}{c}{$0<t_0<T_1$} & \multicolumn{3}{c}{$T_1<t_0<T_2$} & \multicolumn{3}{c}{$T_2<t_0<T$} \\ \midrule
 & \multicolumn{1}{c}{$\ell_{test} = $} & \multicolumn{1}{c}{20} & \multicolumn{1}{c}{40} & \multicolumn{1}{c}{80} & \multicolumn{1}{c}{20} & \multicolumn{1}{c}{40} & \multicolumn{1}{c}{80} & \multicolumn{1}{c}{20} & \multicolumn{1}{c}{40} & \multicolumn{1}{c}{80} \\ \midrule
\multirow{3}{*}{\rotatebox[origin=c]{90}{Burgers}} & \textbf{$\mathcal{P}_{train}$} & $1.23\cdot 10^{-3}$ & $1.14\cdot 10^{-3}$ & $1.08\cdot 10^{-3}$ & $5.28\cdot 10^{-3}$ & $5.21\cdot 10^{-3}$ & $5.28\cdot 10^{-3}$ & $1.67\cdot 10^{-2}$ & $1.65\cdot 10^{-2}$ & $1.75\cdot 10^{-2}$ \\
 & $\mathcal{P}_{interp}$ & $1.24\cdot 10^{-3}$ & $1.10\cdot 10^{-3}$ & $1.06\cdot 10^{-3}$ & $5.11\cdot 10^{-3}$ & $4.94\cdot 10^{-3}$ & $5.08\cdot 10^{-3}$ & $1.60\cdot 10^{-2}$ & $1.60\cdot 10^{-2}$ & $1.72\cdot 10^{-2}$ \\
 & $\mathcal{P}_{extrap}$ & $5.32\cdot 10^{-3}$ & $6.77\cdot 10^{-3}$ & $9.41\cdot 10^{-3}$ & $1.26\cdot 10^{-2}$ & $1.33\cdot 10^{-2}$ & $1.53\cdot 10^{-2}$ & $1.84\cdot 10^{-2}$ & $2.02\cdot 10^{-2}$ & $2.40\cdot 10^{-2}$ \\ \midrule
\multirow{3}{*}{\rotatebox[origin=c]{90}{ADR}} & \textbf{$\mathcal{P}_{train}$} & $5.75\cdot 10^{-3}$ & $5.56\cdot 10^{-3}$ & $5.34\cdot 10^{-3}$ & $9.23\cdot 10^{-3}$ & $1.16\cdot 10^{-2}$ & $1.41\cdot 10^{-2}$ & $1.44\cdot 10^{-2}$ & $2.35\cdot 10^{-2}$ & $3.19\cdot 10^{-2}$ \\
 & \textbf{$\mathcal{P}_{interp}$} & $5.35\cdot 10^{-3}$ & $5.38\cdot 10^{-3}$ & $5.40\cdot 10^{-3}$ & $8.82\cdot 10^{-3}$ & $1.14\cdot 10^{-2}$ & $1.38\cdot 10^{-2}$ & $1.50\cdot 10^{-2}$ & $2.38\cdot 10^{-2}$ & $3.25\cdot 10^{-2}$ \\
 & \textbf{$\mathcal{P}_{extrap}$} & $6.20\cdot 10^{-2}$ & $5.98\cdot 10^{-2}$ & $6.28\cdot 10^{-2}$  & $5.99\cdot 10^{-2}$ & $6.36\cdot 10^{-2}$ & $7.03\cdot 10^{-2}$ & $6.55\cdot 10^{-2}$ & $6.99\cdot 10^{-3}$  & $8.01\cdot 10^{-2}$ \\ \bottomrule
\end{tabular}
}
\caption{\textit{Short-term predictions.} Relative error $\epsilon_{rel}(\bu_h, \hbu_h)$ for short-term predictions on sub-trajectories of test length $\ell_{test}$ being $\frac{1}{2}\times, 1\times, 2\times$ the maximum training sub-trajectory length $\ell_{max}=40$. The tests are performed by sampling sub-trajectories from the train $[0,T_1]$, validation $(T_1,T_2]$ and extrapolation $(T_2,T]$ temporal intervals. Similarly, the rows correspond to the train ($\mathcal{P}_{train}$), interpolation ($\mathcal{P}_{interp}$) and extrapolation ($\mathcal{P}_{extrap}$) parameter instances.}
\label{tb:short-term}
\end{table}

To test short-term predictive capabilities, sub-trajectories of comparable length to the ones employed at training time are sampled from the FOM trajectories. In particular, the employed testing lengths are $\ell=20,40,80$, being $\frac{1}{2}\times, 1\times, 2\times$ the maximum training sub-trajectory length $\ell_{max}=40$ used for LDMs' training in both Burgers' and ADR test cases. The testing sub-trajectories are extracted from the different partitions of the temporal interval $[0,T]$, namely the training split $[0,T_1]$, the validation one $(T_1,T_2]$, and the extrapolation interval $(T_2,T]$.
In such setting, the extracted FOM sub-trajectory is denoted as $\{\bu_h(t_k;\bmu)\}_{k=0}^{\ell-1}$, with $\{t_k\}_{k=0}^{\ell-1}$ being a collection of time-instances belonging to one of the temporal splits.
Leveraging LDM's IVP-structure, $\bu_h(t_0;\bmu)$ is provided as input to the model, along with the temporal steps $\{t_k\}_{k=0}^{\ell-1}$ and the parameters $\bmu$, to produce the approximation $\{\hbu_h(t_k;\bmu)\}_{k=0}^{\ell-1}$. The sub-trajectories are sampled for different parameter instances $\bmu$, belonging to either the training set $\mathcal{P}_{train}$, interpolation set $\mathcal{P}_{interp}$ or extrapolation set $\mathcal{P}_{extrap}$.
The results, expressed in terms of the relative error indicator $\epsilon_{rel}(\bu_h, \hbu_h)$ (Eq. \eqref{eq:rel_err_avg}), are collected in Table \ref{tb:short-term}.
In particular, it can be observed that the LDM is able to provide accurate approximations in both the Burgers' equation and ADR equation test cases, effectively generalizing in time ($t>T_2$), and when performing extrapolation with respect to the parameters ($\bmu\in\mathcal{P}_{extrap}$), with errors on the order of $10^{-2}$, at most. In particular, the same error levels are maintained across different sub-trajectories lengths, highlighting LDMs' capability to handle short-term predictions of comparable length to the training sub-trajectories.

\begin{table}[h]
\center
\footnotesize
\renewcommand{\arraystretch}{1.7}
\resizebox{\textwidth}{!}{
\begin{tabular}{cccccclccc}
 &  &  &  &  &  &  & \multicolumn{2}{c}{CPU Time (s)} &  \\ \cline{1-6} \cline{8-10} 
 &  & $0<t<T_1$ & $T_1<t<T_2$ & $T_2<t<T$ & $[0,T]$ &  & FOM & LDM & Speed-up \\ \cline{1-6} \cline{8-10} 
\multirow{3}{*}{\rotatebox[origin=c]{90}{Burgers}} & $\mathcal{P}_{train}$ & $3.32\cdot 10^{-3}$ & $7.66\cdot 10^{-3}$ & $2.95\cdot 10^{-2}$ & $1.68\cdot 10^{-2}$ &  & 4.90 & $9.13\cdot 10^{-2}$ &  \\
 & $\mathcal{P}_{interp}$ & $3.17\cdot 10^{-3}$ & $7.16\cdot 10^{-3}$ & $2.89\cdot 10^{-2}$ & $1.65\cdot 10^{-2}$ &  & 4.65 & $7.84\cdot 10^{-2}$ & $\sim 40\times$ \\
 & $\mathcal{P}_{extrap}$ & $2.67\cdot 10^{-2}$ & $4.01\cdot 10^{-2}$ & $6.14\cdot 10^{-2}$ & $4.55\cdot 10^{-2}$ &  & 4.41 & $9.94\cdot 10^{-2}$ &  \\ \cline{1-6} \cline{8-10} 
\multirow{3}{*}{\rotatebox[origin=c]{90}{ADR}} & $\mathcal{P}_{train}$ & $6.21\cdot 10^{-3}$ & $1.31\cdot 10^{-2}$ & $4.75\cdot 10^{-2}$ & $2.73\cdot 10^{-2}$ &  & 9.25 & $3.74\cdot 10^{-2}$ &  \\
 & $\mathcal{P}_{interp}$ & $5.94\cdot 10^{-3}$ & $1.27\cdot 10^{-2}$ & $4.67\cdot 10^{-2}$ & $2.67\cdot 10^{-2}$ &  & 9.31 & $3.67\cdot 10^{-2}$ & $\sim 200\times$ \\
 & $\mathcal{P}_{extrap}$ & $7.73\cdot 10^{-2}$ & $8.74\cdot 10^{-2}$ & $1.16\cdot 10^{-1}$ & $9.81\cdot 10^{-2}$ &  & 9.84 & $3.68\cdot 10^{-2}$ &  \\ \cline{1-6} \cline{8-10} 
\end{tabular}
}
\caption{\textit{Long-term predictions and CPU times.} Relative error $\epsilon_{rel}(\bu_h, \hbu_h)$ for long-term predictions, obtained by providing $\bu_{0,h}$ at $t_0=0$, and integrating the LDM's latent dynamics over the entire temporal interval $[0,T]$. The first three columns show the error indicator values over the train, validation and extrapolation temporal splits, respectively, while the fourth column reports the relative error over the full interval. On the right, a comparison between FOM's and LDM's computational times is provided.}
\label{tb:long-term}
\end{table}

Long-term predictive capabilities are tested by considering the whole FOM trajectories on the complete interval $[0,T]$, resulting in the discretized sequence $\{\bu_h(t_k;\bmu)\}_{k=0}^{N_t-1}$. Thus, providing the initial value to the LDM, along with the time-grid and the parametric input, the latent dynamics is solved over $N_t$ steps, which, given the FOM discretizations in both test cases, accounts for a roll-out of size $N_t=1000$. Similarly to the previous case, long-term predictions tests are carried out in a multi-query context by considering training, interpolation and extrapolation parameter instances. The results are reported in Table \ref{tb:long-term}, again expressed in terms of the relative error indicator $\epsilon_{rel}(\bu_h, \hbu_h)$. In particular, considering the error averaged over the whole temporal interval $[0,T]$, it can be observed that the LDM's approximation error is below 10\%, even in the worst case.
Finally, in Table \ref{tb:long-term} (right), we report the CPU speed-ups achieved by the LDM framework compared to the FOM, for both test cases.

\begin{figure}[!t]
    \center
    \includegraphics[width=.95\textwidth]{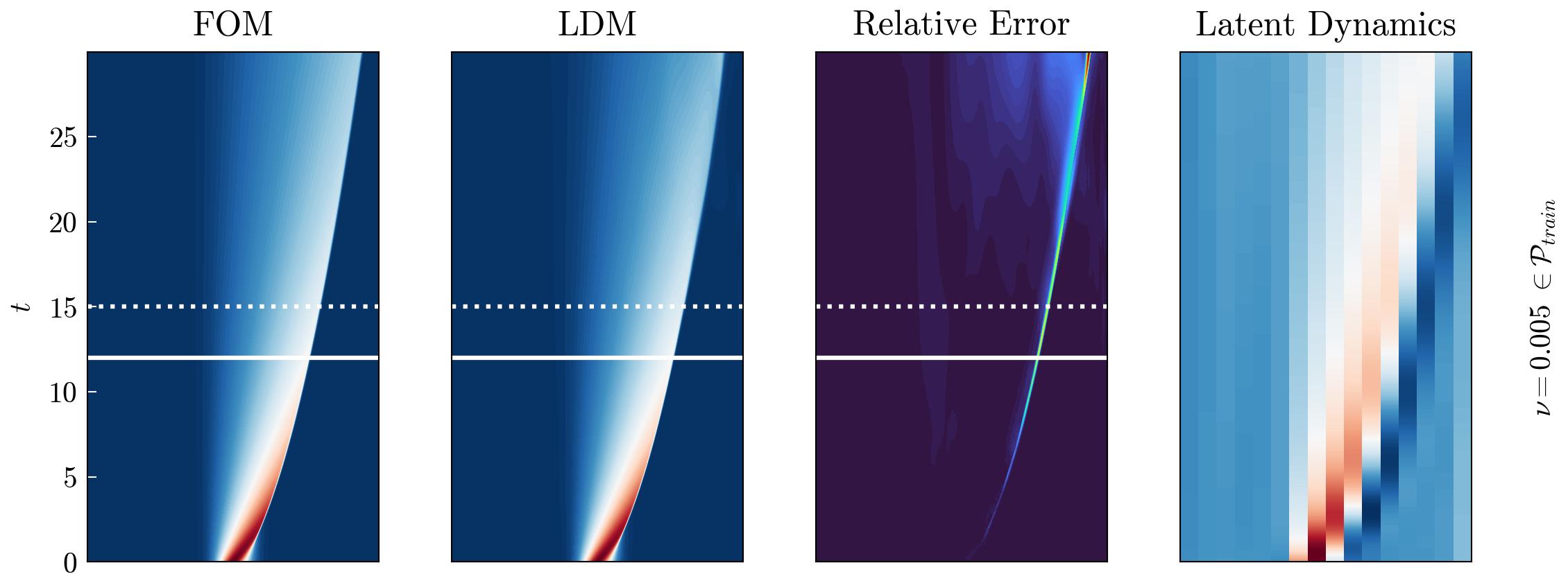}
    \includegraphics[width=.95\textwidth]{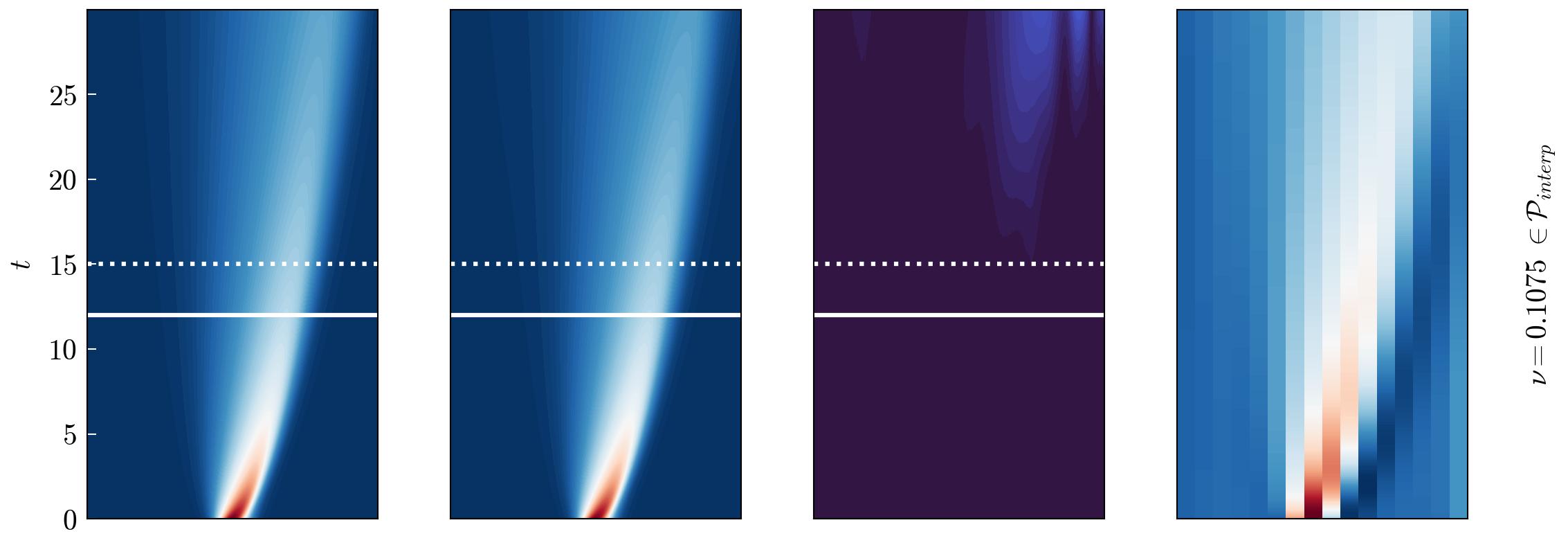}
    \includegraphics[width=.95\textwidth]{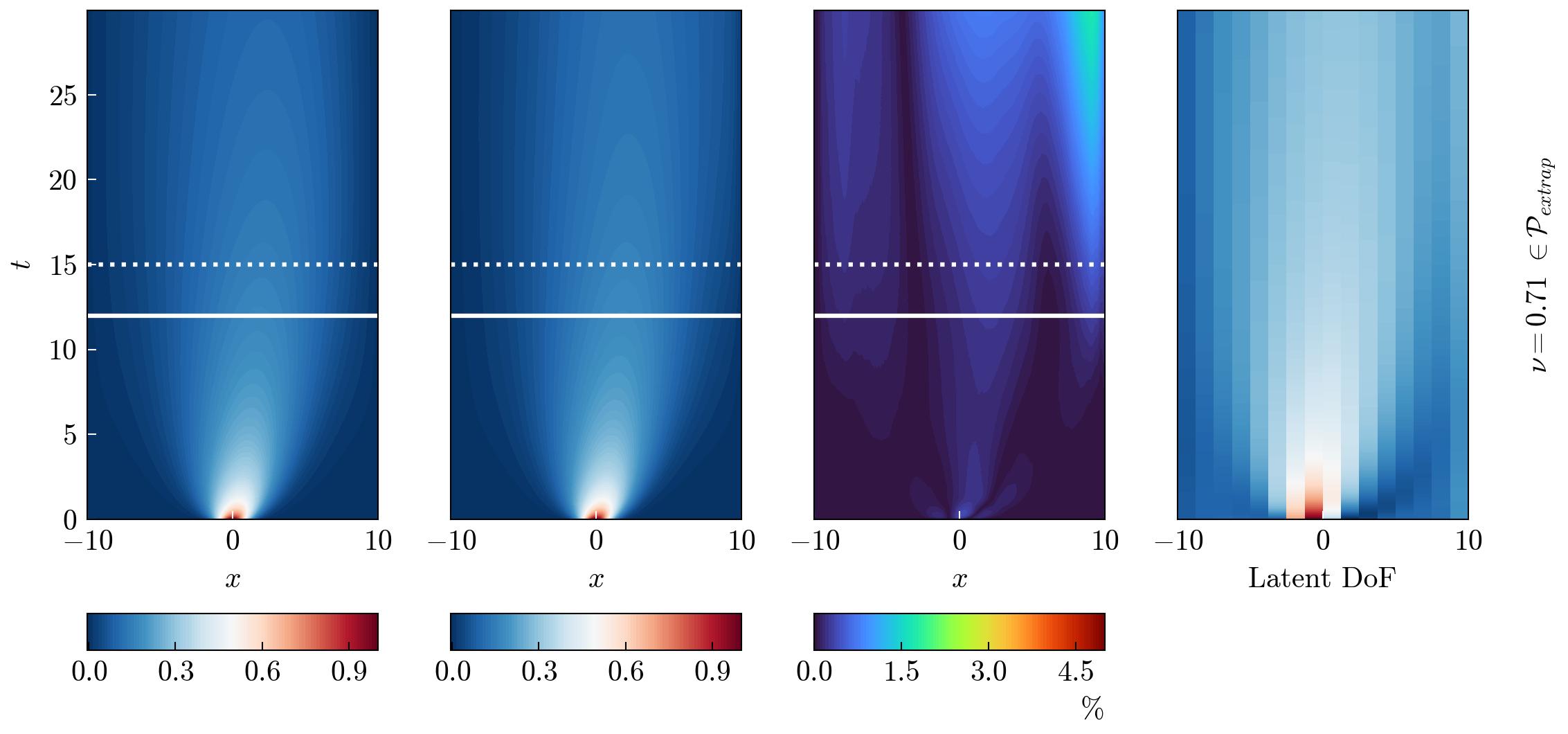}
    \caption{\textit{Burgers' equation.} Evolution of the FOM solution (left), LDM solution (center left), relative error $\beps_{rel}$ (center right), and of the latent state (right). The solutions are computed for three different values of the diffusivity parameter $\nu$, respectively belonging to the training, interpolation and extrapolation sets. 
    The horizontal white lines denote the end of the temporal training (solid) and validation (dotted) intervals, after which time-extrapolation is performed.}
    \label{fig:burgers_solution}
\end{figure}
In Figure \ref{fig:burgers_solution}, we report the comparison between the FOM and LDM Burgers' equation solutions, for three different instances of the diffusivity parameter $\nu = 0.005$, $\nu = 0.1075$ and $\nu = 0.71$, respectively belonging to the training, interpolation and extrapolation sets. Being a long-term roll-out, the three solutions are computed over the complete temporal interval $[0,30]$ with $N_t=1000$. Specifically, the white horizontal lines highlight the end of the training temporal interval (solid) and of the validation interval (dotted), thus performing time-extrapolation for $t>T_2 = 15$. The third column of plots displays the relative error $\beps_{rel}(t; \bu_h, \hbu_h, \bmu)$ defined in Eq. \eqref{eq:rel_err_t_vec}. As observed, the highest spikes of the error occur towards the end of the time-extrapolation region, on the order of $10^{-2}$.
As illustrated on the right-most plot, showing the evolution of the latent state, the proposed fully-convolutional architecture effectively maintains spatial coherence between the high- and low-dimensional states. Indeed, the learned latent state evolution mirrors the high-dimensional evolution on a coarser grid, with a reduced dimension of $n=16$, thus enhancing the model's interpretability, via a latent representation which preserves the key spatial features.

\begin{figure}[!t]
    \center
    \includegraphics[width=\textwidth]{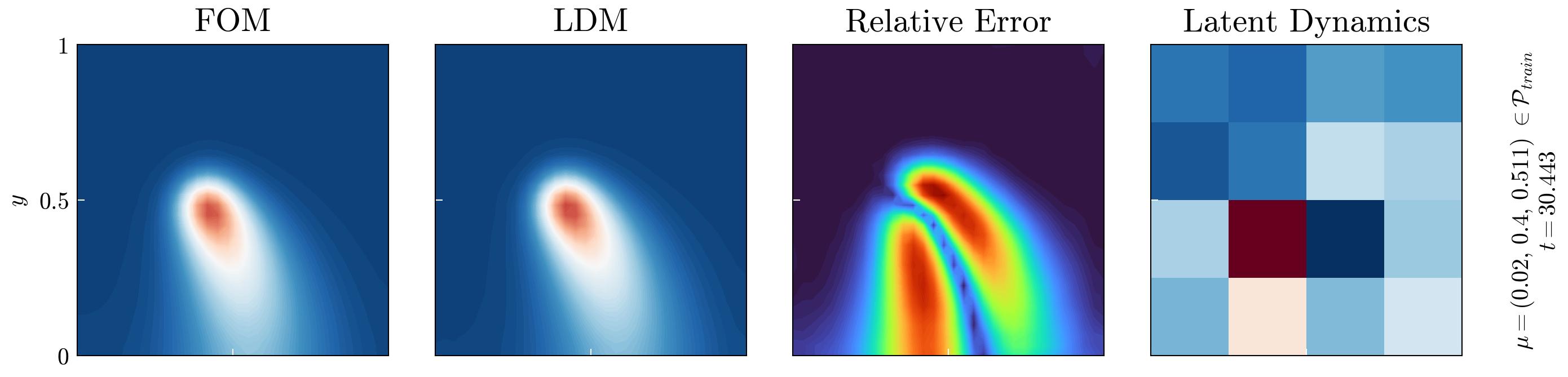}
    \includegraphics[width=\textwidth]{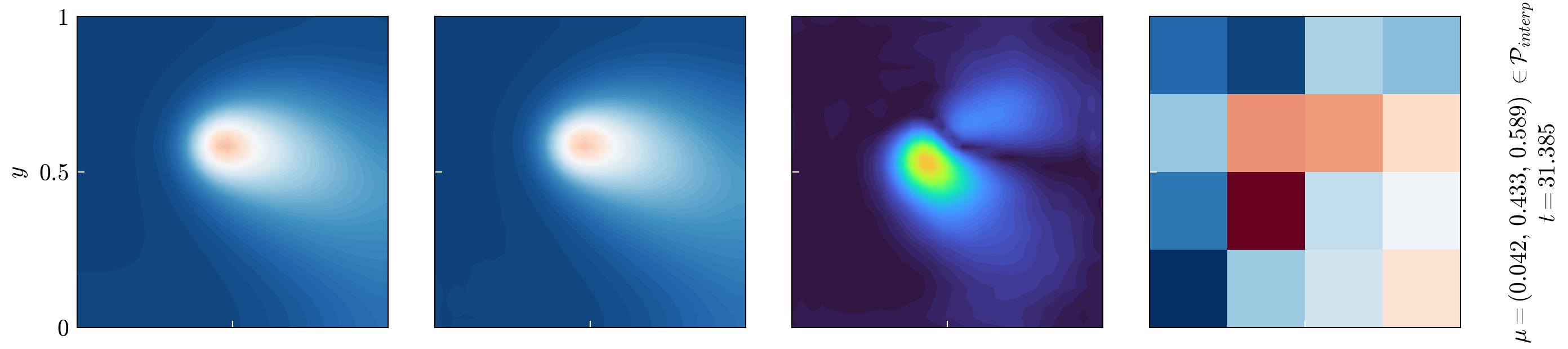}
    \includegraphics[width=\textwidth]{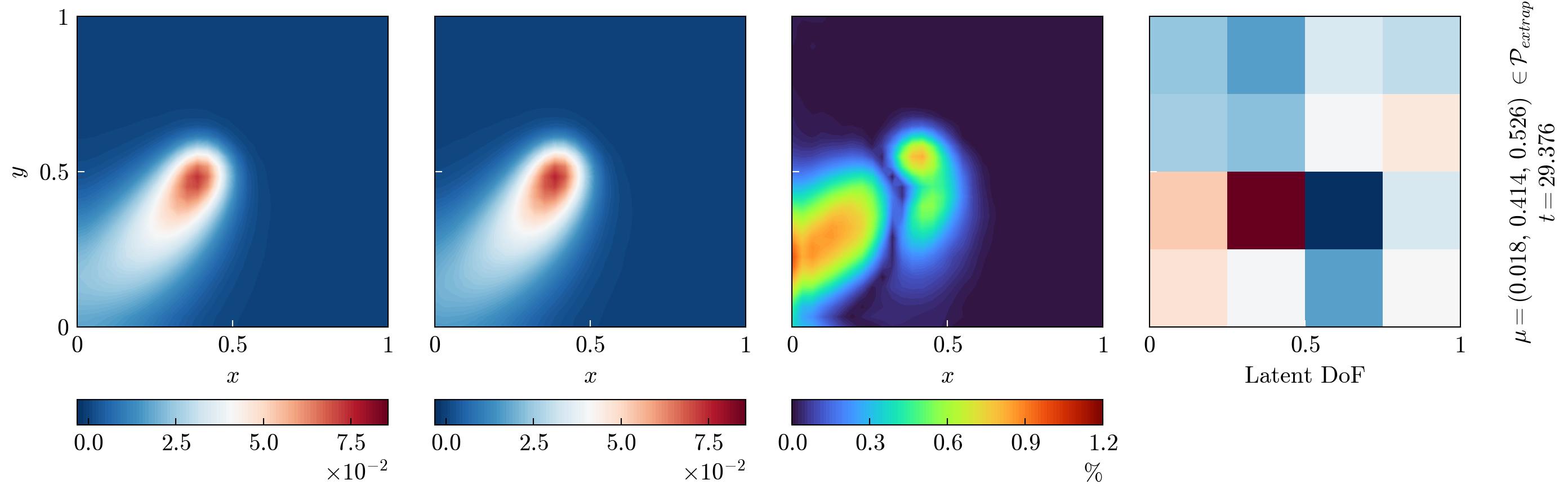}
    \caption{\textit{ADR equation.} 
    FOM solution (left), LDM solution (center left), relative error $\beps_{rel}$ (center right), and  latent state (right) for a fixed time instance $t$. The solutions are computed for three different instances of the parameter vector $\bmu$, respectively belonging to the training, interpolation and extrapolation sets. 
    All the solutions refer to the time-extrapolation interval, being $t>T_2$.}
    \label{fig:adr_solution}
\end{figure}

A comparison between ADR equation FOM and LDM solutions, for three testing parameter instances $\bmu = (0.02, 0.4, 0.511) \in \mathcal{P}_{train}$, $\bmu = (0.042, 0.433, 0.589)\in \mathcal{P}_{interp}$, $\bmu = (0.018, 0.414, 0.526)\in \mathcal{P}_{extrap}$, at the respective time instants $t = 30.433$, $t = 31.385$ and $t = 29.376$, all belonging to the time extrapolation interval ($t>T_2)$, are shown in Figure \ref{fig:adr_solution}. All three solutions refer to a $N_t = 1000$ steps roll-out on the complete temporal interval $[0,10\pi]$. 
Again, the relative error indicator $\beps_{rel}(t; \bu_h, \hbu_h, \bmu)$ is reported, showing values on the order of $10^{-2}$, mainly located in the region where the solution attains its maximum.
Despite being more subtle in this 2D case, the latent state, showed in the right-most plots, still retains visual features that directly correspond to the high-dimensional solution. Specifically, the latent state components with large magnitude (indicated by red color) align with the position of the tail of the high-dimensional solution, on a coarser grid with $n=16$.

\begin{figure}[!ht]
    \center
    \includegraphics[width=\textwidth]{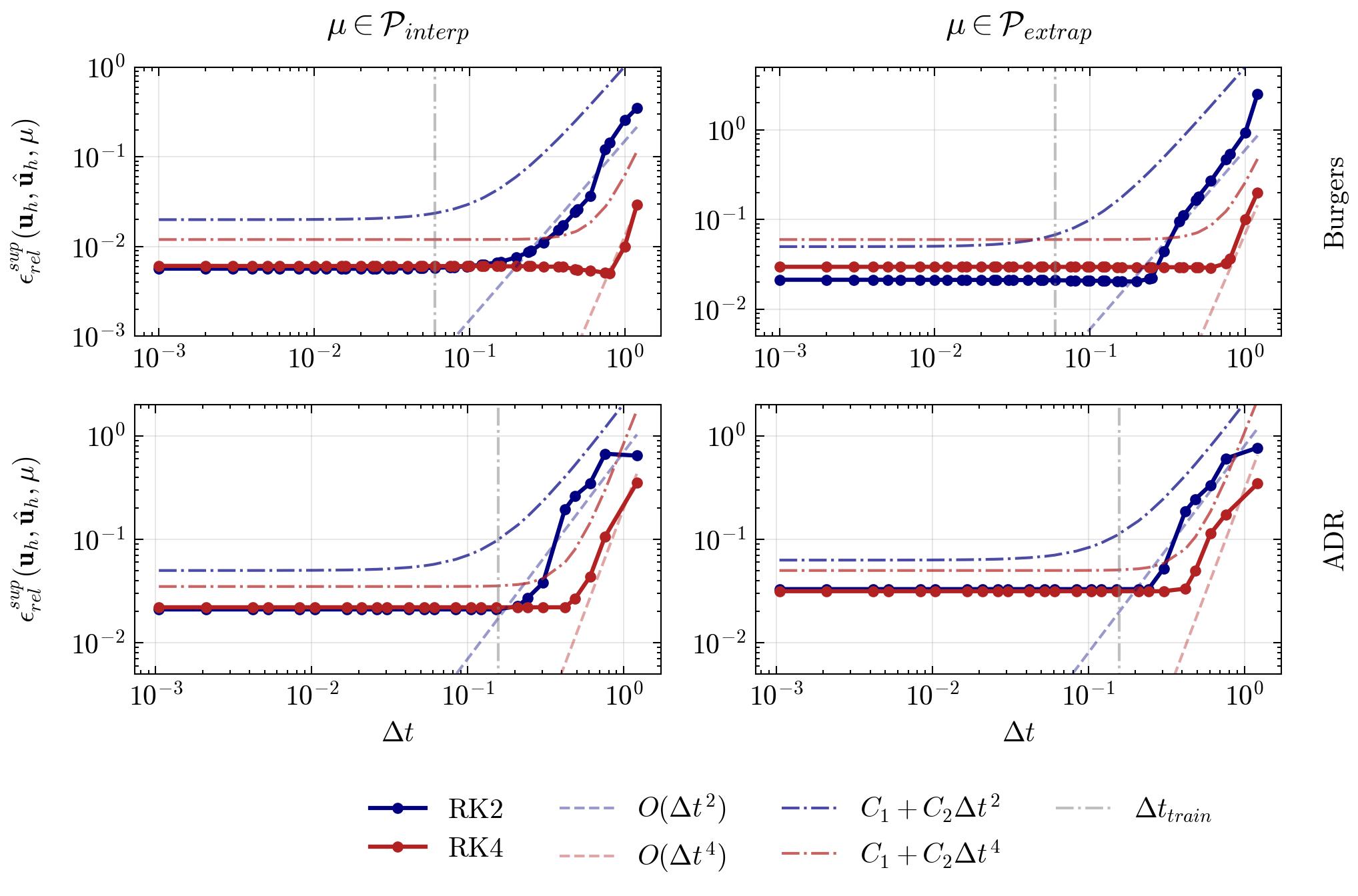}
    \caption{\textit{Time-continuous approximation.}  
    Assessment of the time-continuous approximation property for the two test cases: Burgers' equation (top), ADR equation (bottom). In both test cases, the property is evaluated using two integration schemes (RK2, RK4), and two different testing parameter instances (interpolation, extrapolation). Specifically, for the former test case, the adopted diffusivity coefficients are $\nu=0.2075$ and $\nu=0.61$. In the latter, $\bmu=(0.0283, 0.5, 0.4556)$ and $\bmu=(0.0337, 0.3754, 0.4767)$ are considered.}
    \label{fig:time_continuity}
\end{figure}

\subsubsection{Time-continuity}
In this section, the numerical validation of the ROM time-continuous approximation property (Def. \eqref{def:timecont}) is addressed. The tests are performed for both Burgers' equation and ADR equation test cases, demonstrating the effectiveness of the LDM framework in providing a time-continuous approximation.
Specifically, for each test case, we evaluate two models, with identical architectural details, but trained using different numerical methods for integrating the parameterized latent NODE, namely RK2 and RK4.
In assessing time-continuity of the approximation, we focus on two main aspects: (i) the ability of the LDM to demonstrate zero-shot super-resolution with respect to temporal resolution, meaning that it can provide the same level of accuracy over different (finer) temporal discretizations that were not encountered during training (being the LDM trained on a fixed temporal discretization with a step $\Dt_{train}$); and (ii) the validity of the upper bound \eqref{eq:error_decomposition}, arising from the error decomposition which accounts for the numerical error sources, within a learnable context.
To this mean, the relative error indicator $\epsilon_{rel}^{sup}(\bu_h, \hbu_h, \bmu)$, defined in Eq. \eqref{eq:rel_err_sup}, is used to measure the maximum error over the time instants of the employed testing temporal discretization. 
The tests involve finer discretizations of the time interval $[0,T]$ on which the FOM is computed, with a number of steps up to $N_t = 10^4$ and $N_t = 3\cdot 10^4$, for the Burgers' equation and ADR equation, respectively. 
The employed LDMs have been trained with time steps $\Dt_{train} = 2\Dt_{FOM}$ and $\Dt_{train} = 5\Dt_{FOM}$. Consequently, the tests are conducted on temporal discretizations that are up to $20\times$ and $150\times$ finer than those used during training.
Figure \ref{fig:time_continuity} illustrates the behavior of the error indicator $\epsilon_{rel}^{sup}$ with respect to the testing temporal discretization step $\Dt$. Specifically, the tests for both benchmark problems are reported, with the two columns referring to two different interpolation and extrapolation testing parameter instances.
In particular, the first row refers to the Burgers' equation test case, for the testing parameters instances $\nu=0.2075\in\mathcal{P}_{interp}$ and $\nu=0.61\in\mathcal{P}_{extrap}$. The second row refers to the ADR equation test case, for the testing parameter instances $\bmu=(0.0283, 0.5, 0.4556)\in\mathcal{P}_{interp}$ and $\bmu=(0.0337, 0.3754, 0.4767)\in\mathcal{P}_{extrap}$.
The training temporal discretization step $\Dt_{train}$ is indicated, and a constant behavior of the error indicator is observed in all the test cases for $\Dt\leq \Dt_{train}$.
It is evident that the resulting approximation provided by the LDM, in a discrete learnable context ($\dldm_\theta$), can be queried at each time instance while preserving the training accuracy, with the time-continuous property holding also for the extrapolation parameter instances.
Additionally, the validity of the upper bound, derived from the error decomposition formula \eqref{eq:error_decomposition}, can be observed, with the error behaving as $O(\Dt^p)$ for both $p=2$ and $p=4$, as the discretization becomes coarser. 

Interestingly, the employment of a higher-order scheme, such as RK4, does not necessarily lead to improvements in terms of accuracy, primarily due to the models' capacity limitations. Rather, the region for which the model shows time-continuous property extends towards the right compared to lower-order methods like RK2, likely related to the fact that RK4 has a larger $\mathcal{A}$-stability region.

\begin{figure}[!ht]
    \center
    \includegraphics[width=\textwidth]{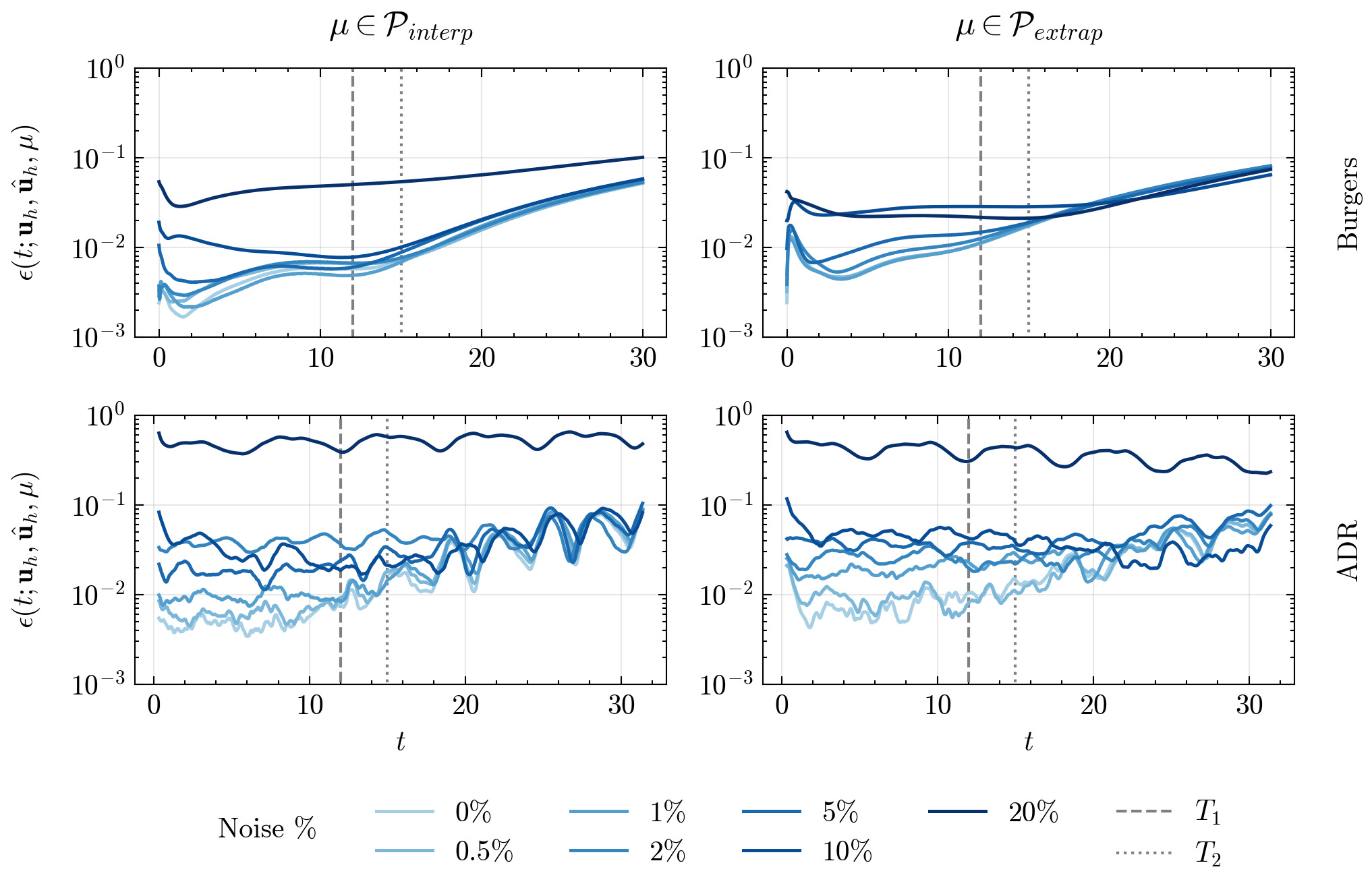}
    \caption{\textit{Zero-stability.} Test of the zero-stability property for multiple initial value perturbation of increasing magnitude in the two test cases: Burgers' equation (top), ADR equation (bottom). The tests are conducted for two testing parameter instances from the interpolation and extrapolation sets, respectively. In the first test case, $\nu=0.4775$ and $\nu=0.6$ are considered, while in the second $\bmu=(0.035, 0.43, 0.41)$ and $\bmu=(0.05278, 0.4557, 0.3927)$ are used.}
    \label{fig:zero_stability}
\end{figure}

\subsubsection{Zero-stability}
Here, we aim to assess whether the zero-stability property, derived for the $\dldm$ scheme in Section \ref{section:DLDM}, holds in a learnable context.
We test the framework's robustness with respect to initial value perturbations, by considering long-term predictions over the complete temporal interval $[0,T]$. Specifically, the experiment relies on the testing procedure outlined in Algorithm \ref{alg:testing}, with a randomly sampled Gaussian perturbation $\bdelta_h\sim\mathcal{N}(0,\sigma^2 I_{N_h})$, where the amount of noise on $\bu_{0,h}$ is progressively increased. 

The time-dependent relative error indicator $\epsilon_{rel}(t; \bu_h, \hbu_h, \bmu)$ (Eq. \eqref{eq:rel_err_t}) is employed for assessing the impact of the initial value perturbation, by benchmarking the perturbed LDM prediction directly against the FOM solution. In Figure \ref{fig:zero_stability}, the relative error evolution is reported, for varying levels of noise, going from 0\% up to 20\%. The first row refers to the Burgers' equation test case, for the testing parameters instances $\nu=0.4775\in\mathcal{P}_{interp}$ and $\nu=0.6\in\mathcal{P}_{extrap}$, over the time interval $[0,30]$. The second row illustrates the ADR equation test case, for the testing parameter instances $\bmu=(0.035, 0.43, 0.41)\in\mathcal{P}_{interp}$ and $\bmu=(0.05278, 0.4557, 0.3927)\in\mathcal{P}_{extrap}$, over the time interval $[0,10\pi]$. 
The error between the FOM and the perturbed LDM approximation demonstrates a bounded behavior, confirming the zero-stability property, even when performing parameter extrapolation.  Notably, when approaching the extrapolation interval ($t>T_2$), the errors for different levels of noise show a converging behavior, with an overall error magnitude of approximately $10^{-2}$. In both test cases, the error remains bounded even as the perturbation magnitude reaches the 20\% level, although it becomes significant.

\begin{figure}[!h]
    \center
    \includegraphics[width=.95\textwidth]{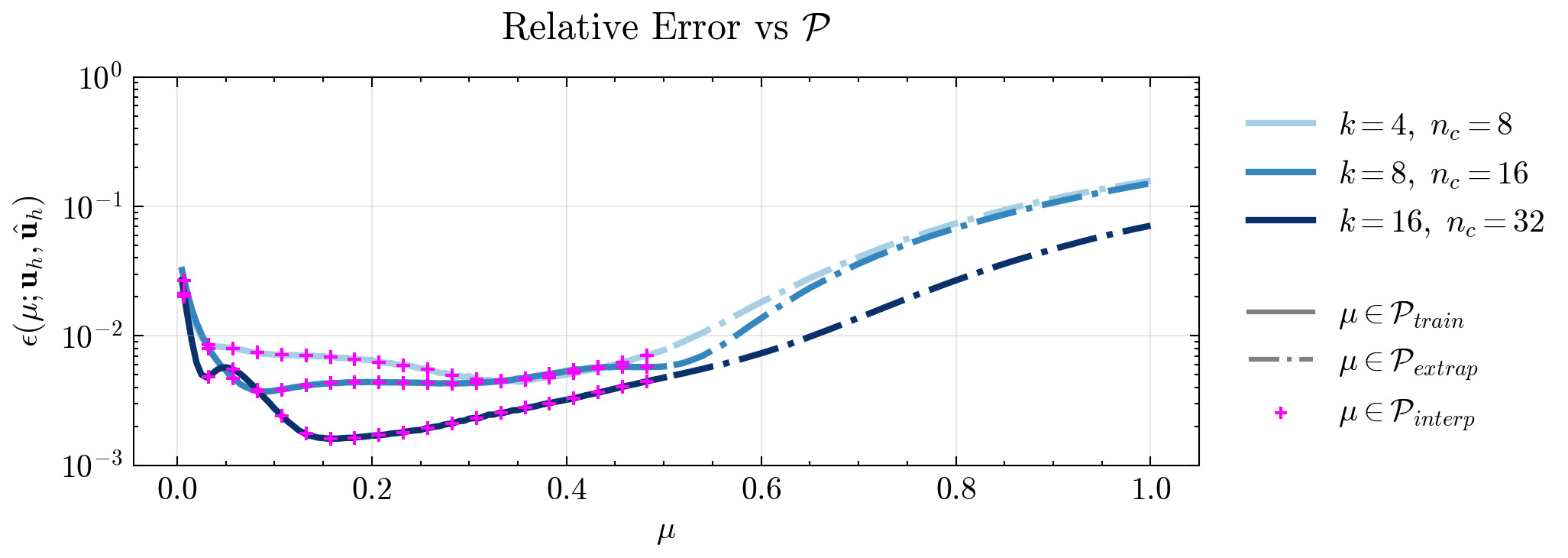}
    \caption{\textit{Burgers' equation.} Behavior of the relative error indicator $\epsilon_{rel}(\bmu;\bu_h, \hbu_h)$ over the parameter space $\mathcal{P}$. The three curves refer to different model configurations with increasing encoding dimensions ($k$) and growing number of channels ($n_c$) in the latent dynamics' $\bff_{n,\theta}$ modulated hidden layers. Continuous lines depict the error over the training instances ($\mathcal{P}_{train}$), while dash-dotted lines represent the error for the extrapolation instances ($\mathcal{P}_{extrap}$). Magenta crosses indicate a subset (for visualization purposes) of the interpolation instances ($\mathcal{P}_{interp}$).}
    \label{fig:burgers_params}
\end{figure}

\subsubsection{Parameter dependence}
Having introduced the LDM framework in a multi-query context, with a parameter-dependent formulation, we analyze the behavior of the LDM solution at extrapolation tasks with respect to the parameter input. 
We address two different aspects within the two benchmark problems. First, for Burgers' equation, we assess the impact of the sinusoidal encoding and the affine-modulation mechanism in the construction of the parameterized latent dynamics, proposed in Section \ref{sec:affine-param-latent-dynamics}. 

The influence of two main factors is studied: the encoding dimension $k$ and the capacity of the modulated convolutional NN employed within the latent dynamics. 
Our approach involves incrementally increasing the encoding size $k\in\{4,8,16\}$, alongside the number of affinely-modulated channels $n_c\in\{8,16,32\}$.
The results, presented in Figure \ref{fig:burgers_params}, illustrate how the parameter-dependent relative error $\epsilon_{rel}(\bmu;\bu_h, \hbu_h)$, defined in Eq. \eqref{eq:rel_err_mu}, varies with respect to the viscosity coefficient. As shown, the LDM employing a larger encoding dimension ($k=16$) and a higher number of modulated channels ($n_c=32$) achieves an error magnitude on the order $10^{-2}$ across the entire extrapolation interval $\nu \in (0.5,1]$, exhibiting a more gradual increase compared to other configurations.

\begin{figure}[!h]
    \center
    \includegraphics[width=\textwidth]{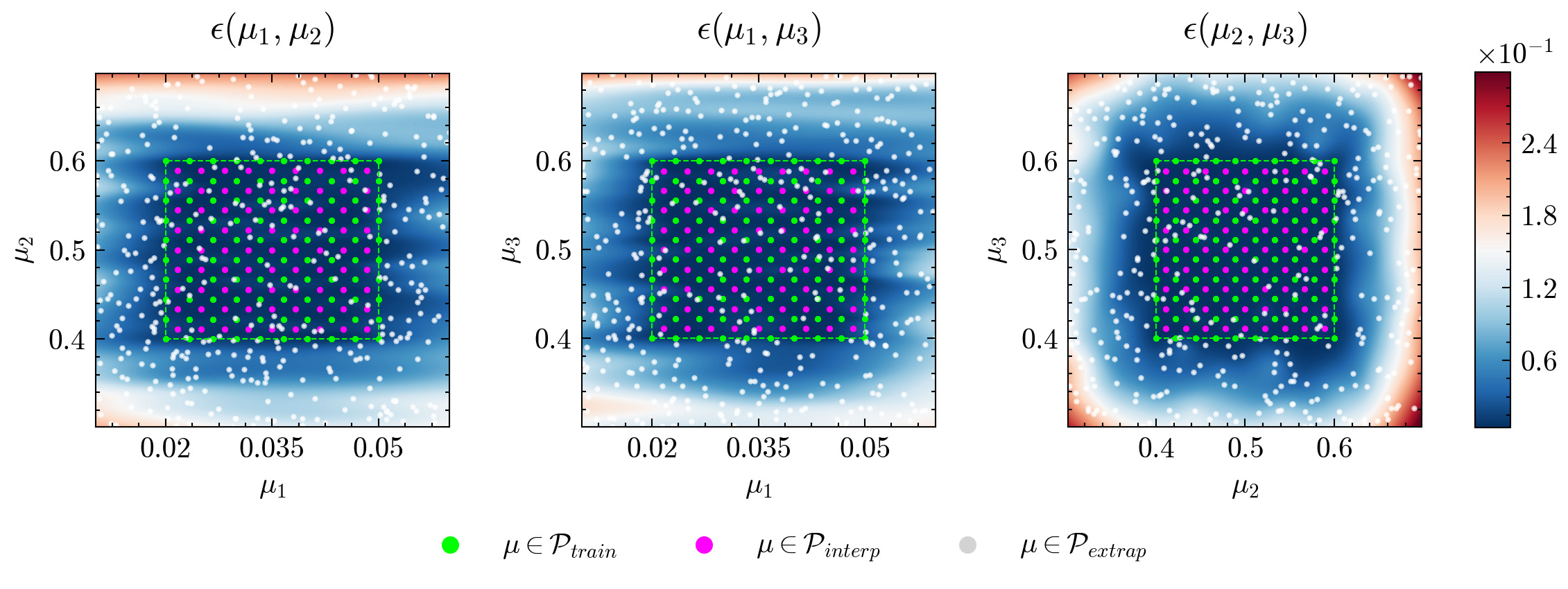}
    \caption{\textit{ADR equation.} Behavior of the relative error indicator $\epsilon_{rel}(\bmu;\bu_h, \hbu_h)$ with respect to the three parameters component $\bmu=(\mu_1,\mu_2,\mu_3)$ across the parameter space $\mathcal{P}$. Green points represent the training parameter instances ($\mathcal{P}_{train}$), while magenta points indicate interpolation testing instances ($\mathcal{P}_{interp}$). The white points correspond to extrapolation test instances ($\mathcal{P}_{extrap}$). Note that for pairs of parameters components $(\mu_i,\mu_j)$ falling within the training interval, the remaining component $\mu_k$ is within the extrapolation range, allowing for extrapolation testing across all components ($i,j,k \in \{1,2,3\}, \ i\neq j \neq k$).}
    \label{fig:adr_params}
\end{figure}

In the ADR test case, we focus on assessing the generalization capabilities for $n_\mu\geq2$. The error behavior with respect to the different parameter components $\bmu = (\mu_1, \mu_2, \mu_3)$ is illustrated in Figure \ref{fig:adr_params}. 
The plots highlight the training region (green) including the training instances ($\mathcal{P}_{train}$), the interpolation instances ($\mathcal{P}_{interp}$). Being $n_\mu = 3$, the extrapolation capabilities with respect to each one of the components are tested. It is evident that the parameters related to the position of the source term ($\mu_2$ and $\mu_3$) have a major impact on the LDM prediction accuracy. As shown in the first and second plots, the error shows less variability with respect to the diffusivity parameter $\mu_1$, whereas the third plot shows substantial increase in the error magnitude (order of $10^{-1}$) as the boundaries of the parameter space are approached.

\section{Conclusions}
The concept of latent dynamics underlies many of the recently proposed DL-based ROMs addressing data-driven nonlinear order-reduction of time-dependent problems. 
This paradigm has been mainly driven by the adoption of recurrent architectures modeling the time-evolution of the reduced state within the latent space.
However, RNN-based models, and more generally autoregressive ones, often require a temporal context, encoding current and previous high-fidelity FOM's snapshots, not allowing for the resulting ROM to be truly independent of the FOM in terms of online computational costs. 
Furthermore, these approaches are inherently tied to the temporal discretization employed during the offline training phase, limiting their ability to handle different testing temporal discretizations. Consequently, the learned dynamics cannot be interpreted in a continuous-sense, being constrained to a discrete temporal representation. As a result, further training or fine-tuning stages are required when encountering novel temporal discretizations, thereby entailing higher offline computational costs.

Neural ODEs (NODEs), coupled with autoencoders (AEs), present a promising alternative for constructing data-driven ROMs able to overcome these limitations.
Recent studies have demonstrated the time-continuous approximation properties of NODEs in modeling dynamical systems, particularly when higher-order RK schemes are employed, guaranteeing the same error magnitude across refinements of temporal discretization. However, the time-continuous approximation capabilities of the AE-NODE architecture in multi-query settings had not been addressed prior to our research.

To this end, we extended the notion of \textit{time-continuous approximation} towards a data-driven reduced order modeling context, and introduced the mathematical framework of \textit{latent dynamics models} (LDMs) to formalize the AE-NODE architecture, and address its properties.
Specifically, our framework develops along three level of approximation: time-continuous setting (LDM), time-discrete setting ($\dldm$), and learnable setting ($\dldm_\theta$).
Such chain of approximations enabled us to \textit{(i)}  analyze the error sources and derive stability estimates for the LDM approximation in a time continuous setting; \textit{(ii)} address the numerical aspects of the method when integration schemes are employed for the solution of the latent dynamics in a time-discrete setting; \textit{(iii)} derive a time-continuous approximation result in a learnable context when DNNs are employed to approximate the $\dldm$ components.
The resulting framework is characterized by an IVP-structure, requiring only the initial value to evolve the system, rather than a temporal context, leading to a ROM whose computational complexity is fully independent of the FOM's complexity. Additionally, the time-continuous approximation capabilities yield a ROM capable of providing the same level of accuracy for refined temporal discretizations relative to the training discretization.

From an architectural perspective, we propose adopting a fully convolutional structure for both the AE and the parameterized NODE. This approach enables the retention of spatial coherence at the latent level while substantially reducing the network size compared to classical AE-NODE dense architectures. Addressing the specific context of multi-query applications, we introduced a parameterized affine modulation mechanism to inject temporal and parametric information into the convolutional latent dynamics, enabling us to mitigate the additional costs typically associated with hypernetwork-based parameterization approaches.
The numerical and time-continuous approximation properties of the proposed architecture have been tested on high-dimensional dynamical systems arising from the semi-discretization of parameterized nonlinear time-dependent PDEs.

While this work has primarily focused on the temporal aspects of the proposed framework, its scalability with respect to the spatial discretization needs further investigation. Examining the generalization capabilities and computational efficiency across varying levels of spatial resolution is essential for extending the framework's applicability to large-scale problems.
In particular, possible future directions may consider coupling the LDM framework with GNNs to enhance its capability of handling complex domains and three-dimensional geometries. Additionally, the recent advancements in the field of operator learning can be leveraged to define an LDM architecture with continuous approximation capabilities in both the temporal and the spatial domains.

\section*{Acknowledgments}
SF, SB and AM are members of the Gruppo Nazionale Calcolo Scientifico-Istituto Nazionale di Alta Matematica (GNCS-INdAM) and acknowledge the project “Dipartimento di Eccellenza” 2023-2027, funded by MUR. 
SB acknowledges the support of European Union - NextGenerationEU within the Italian PNRR program (M4C2, Investment 3.3) for the PhD Scholarship "Physics-informed data augmentation for machine learning applications". AM and NF acknowledge the Project “Reduced Order Modeling and Deep Learning for the real-time approximation of PDEs (DREAM)” (Starting Grant No. FIS00003154), funded by the Italian Science Fund (FIS) - Ministero dell'Università e della Ricerca. AM and SF acknowledge the project FAIR (Future Artificial Intelligence Research), funded by the NextGenerationEU program within the PNRR-PE-AI scheme (M4C2, Investment 1.3, Line on Artificial Intelligence).

\newpage
\bibliographystyle{plain}
{\small
\bibliography{biblio}
}

\appendix
\newpage
\section{Additional proofs}
\label{sec:additional_proofs}
\subsection{Proof of Proposition 3.3}
\begin{proof}
Considering the LDM formulation over $[t_0,T]$, namely
\begin{equation*}
    \tbu_h(t;\bmu) = \Psi'\bigg(\Psi(\bu_{0,h}(\bmu)) -  \int_{t_0}^{t} \bff_n(s, \bu_n(s;\bmu); \bmu)ds \bigg) = \Psi'(\bu_n(t;\bmu)),
\end{equation*}
we observe that, thanks to the chain rule, the following IVP is satisfied
\begin{equation*}
\left\{
\begin{aligned}
    &\partial_t \tbu_h(t;\bmu) = \mathbf{J}_{\Psi',\bu_n}(t;\bmu) \partial_t \bu_n(t;\bmu) = \mathbf{J}_{\Psi',\bu_n}(t;\bmu) \bff_n(t, \bu_n(t;\bmu);\bmu), \quad t \in (t_0,T],  \\
    &\tbu_h(0;\bmu) = \Psi' \circ \Psi(\bu_{0,h}(\bmu)).
\end{aligned}
\right.
\end{equation*}
Thus, it is trivial to conclude
\begin{equation*}
\begin{aligned}
    \|\bu_h(t;\bmu) & - \tbu_h(t;\bmu)\|
     \\
    & = \bigg\|\bu_{0,h}(\bmu) + \int_{t_0}^{t} \bff_h(s, \bu_h(s;\bmu); \bmu) ds - \Psi' \circ \Psi (\bu_{0,h}(\bmu)) -  \int_{t_0}^{t} \mathbf{J}_{\Psi',\bu_n}(s;\bmu) \bff_n(s, \bu_n(s;\bmu); \bmu)ds\bigg\| \\
    &\le \| \bu_{0,h}(\bmu) - \Psi' \circ \Psi (\bu_{0,h}(\bmu)) \| + \bigg\| \int_{t_0}^{t} [\bff_h(s, \bu_h(s;\bmu); \bmu) - \mathbf{J}_{\Psi',\bu_n}(s;\bmu) \bff_n(s, \bu_n(s;\bmu); \bmu)]ds \bigg\|\\
    &\le \| \bu_{0,h} - \Psi' \circ \Psi (\bu_{0,h}) \|_{L^\infty(\mathcal{P}; \mathbb{R}^{N_h})} + |T-t_0| \|\bff_h(\bu_h) - \mathbf{J}_{\Psi',\bu_n} \bff_n(\bu_n)) \|_{L^\infty([t_0,T] \times \mathcal{P}; \mathbb{R}^{N_h})}.
\end{aligned}
\end{equation*}
\end{proof}

\newpage
\section{Continuous-depth and continuous-time paradigms}
\label{appendix:continuous-time}
In the context of deep learning, the notion of \textit{continuous} architectures has been recently introduced. In particular, we have to make a distinction between the strictly related concepts of \textit{continuous-depth} and \textit{continuous-time} models.

\paragraph{Continuous-depth.} \textit{Continuous-depth} architectures \cite{queiruga2020continuousindepth}, likewise \textit{neural ordinary differential equations} \cite{NEURIPS2018_69386f6b}, have been introduced by means of continuous limits of \textit{residual neural networks}, by considering the correspondence between the ResNets' residual formulation and forward Euler scheme \cite{pmlr-v80-lu18d}.
In particular, by considering the general case where we have to model input-output relations, considering feature-target pairs $(x,y)$, a ResNet with $L$ layers takes the form
\begin{equation}
    \label{eq:resnet}
    x\mapsto h_1 \mapsto h_2 \mapsto \cdots \mapsto h_L \mapsto \hat{y},
\end{equation}
where $\mapsto$ are residual blocks of the form $h_{l+1} = h_{l} + f_\theta(h_l,l)$, so the whole network reads as
\begin{equation*}
\begin{cases}
     \hat{y} = h_{L+1},\\
     h_{l+1} = h_{l} + f_\theta(h_l,l), & l=0,\ldots,L,\\
     h_0 = x,\\
\end{cases}
\end{equation*}
where $f_\theta(\cdot,l)$ represents the inner learnable block of weights $\theta$ of the $l$-th layer.
In this view, the discrete scheme outlined above may be seen as a discretization, by means of the explicit Euler scheme with $\Dt = 1$, of a continuous problem on $[0,T]$, reading as
\begin{equation*}
    \begin{cases}
        \hat{y} = h(T),\\
        \dot{h}(t) = f_\theta(t,h(t)), & t\in(0,T),\\
        h(0) = x.
    \end{cases}
\end{equation*}
Thus, via the above continuous formulation, any numerical method for the solution of ODEs can now be employed for integrating the network $f_\theta(t,h(t))$, and retrieving the final output $y = h(T)$, where $h(t)$ is now a continuously evolving hidden-state. The arbitrariness of the employed discretization of the integration interval leads to the concept of \textit{continuous-depth}, since, in this context, employing a finer time-discretization would mean taking a larger number of function evaluations within the numerical solver, i.e., a higher number of intermediate hidden states, naively corresponding to a \textit{deeper} residual architecture. Thus, the continuous-depth nature of such architecture  resides in: (i) the continuous nature of the variable now indexing the hidden layers $t\in\mathbb{R}$, rather than discrete indexes $l\in\mathbb{N}$, (ii) and in their continuous evolution, by construction.
We remark that, in this case, we are not interested in the evolution of the hidden state $h(t)$, but only in the final output $\hat{y}=h(T)$. Moreover, the fact that the hidden state now evolves continuously accordingly to an ODE, means that the relation $x = h(0) \mapsto h(T) = y$ is being modeled by a homeomorphism, i.e., a one-to-one continuous mapping with continuous inverse.

\paragraph{Continuous-time.}
The dynamical nature of continuous-depth architectures enables their employment in the context of time-series modeling, leading to the concept of \textit{continuous-time} models, to model relations between input-target pairs of the form $(x,y) = (x_0, \{x_1,...,x_n\})$ in a \textit{recurrent} manner, where the target is made by $n$ sampled steps of a continuous trajectory over $[0,T]$. In this setting, the sampled state $x(t)$ can be modeled continuously by means of $\hat{x}(t)$, evolving accordingly to the following \textit{neural} ODE
\begin{equation}
    \begin{cases}
        \dot{\hat{x}}(t) = f_\theta(t,\hat{x}(t)), & t\in(0,T),\\
        \hat{x}(0) = x_0.
    \end{cases}
    \label{eq:cont_time}
\end{equation}
where $f_\theta$ is the neural network parameterizing its dynamics. In particular, \eqref{eq:cont_time} can be seen as a continuous formulation of a recurrent neural network, modeling the evolution of $x(t)$. In such setting, since the whole trajectory is now the target, the objective being minimized is $\sum_{i=1}^n\|x(t_i) - \hat{x}(t_i)\|^2$, meaning that the prediction is being unrolled over the discretized interval, as in recurrent settings. 

\newpage
\section{Notation}

\subsection{Norms}
Throughout the paper we indicate by $\|\cdot\|$ the Euclidean norm, typically employed in either $\mathbb{R}^{N_h}$ or $\mathbb{R}^{n}$, for vectors lying in the high-dimensional space (e.g. $\bu_h, \hat{\bu}_h$) or in the latent space  (e.g. $\bu_n$), respectively.

Regarding functional spaces, we rely on Sobolev spaces $W^{k,p}(\Omega)$, and specifically on their norms. In the following their definition, accordingly to \cite{adams2003sobolev}, is provided.

\begin{definition}
Let $\Omega\subset \mathbb{R}^d$ be a domain, $1\leq p\leq\infty$, $k\in\mathbb{N}$, the Sobolev space $W^{k,p}(\Omega;\mathbb{R})$ is defined as
\begin{equation*}
    W^{k,p}(\Omega;\mathbb{R}) = \big\{f:\Omega \rightarrow \mathbb{R},\  D^\alpha f \in L^p(\Omega), \ \forall \alpha \in \mathbb{N}_0^d, \ |\alpha|\leq k\big\},
\end{equation*}
equipped with the following norm
\begin{equation*}
    \|f\|_{W^{k,p}(\Omega;\mathbb{R})} \coloneq \begin{cases}
        \Bigg(\sum_{|\alpha|\leq k}\|D^\alpha f\|_{L^p(\Omega)}^p\Bigg)^{1/p},  & 1\leq p < \infty,\\
        \\
        \max_{|\alpha|\leq k}\|D^\alpha f\|_{L^\infty(\Omega)}, & p = \infty,
    \end{cases}
\end{equation*}
where $D^\alpha$ indicates the weak (distributional) derivative of order $\alpha$.
\end{definition}

\noindent We note that we use $L^p$ instead of $W^{0,p}$ throughout the length of the manuscript for the sake of readability.

\subsection{Symbols}
For the sake of readability, in the following, we report a table summarizing the notation employed within the manuscript.
\begin{center}
\center
\begin{tabularx}{\textwidth}{@{}ll@{}}
\toprule
\textbf{Symbol} & \textbf{Description}\\ \midrule
$\bu_h(t;\bmu)$ & Full-order model state \\
$\bff_h(t,\bu_h(t;\bmu) ;\bmu)$ & Full-order model vector field \\
$\tbu_h(t;\bmu)$ & LDM full-order state approximation \\
$\bu_n(t;\bmu)$ & LDM latent state \\
$\bff_n(t,\bu_n(t;\bmu) ;\bmu)$ & LDM latent vector field \\
$\tilde{\bz}_h(t;\bmu)$ & Perturbed LDM full-order state approximation \\
$\bz_n(t;\bmu)$ & Perturbed LDM latent state \\
$\tbu^k_h(\bmu)$ & $\Delta$LDM full-order state approximation at $t_k$ \\
$\bu^k_n(\bmu)$ & $\Delta$LDM latent state at $t_k$\\ 
$\tilde{\bz}^k_h(\bmu)$ & Perturbed $\Delta$LDM full-order state approximation at $t_k$\\
$\bz^k_n(\bmu)$ & Perturbed $\Delta$LDM latent state at $t_k$\\ 
$\hbu_h(t;\bmu)$ & Learned $\text{LDM}_\theta$ full-order state approximation \\
$\hbu_n(t;\bmu)$ & Learned $\text{LDM}_\theta$ latent state \\
$\hbu_h^k(\bmu)$ & Learned $\dldm_\theta$ full-order state approximation at $t_k$\\
$\hbu_n^k(\bmu)$ & Learned $\dldm_\theta$ latent state at $t_k$\\ 
$\bff_{n,\theta}(t, \ \cdot \ ;\bmu)$ & Learned $(\Delta)\text{LDM}_\theta$ latent vector field \\
\bottomrule
\end{tabularx}
\end{center}

\end{document}